\documentclass[a4paper,twoside]{article}  
\usepackage{amssymb,amsmath,amsthm,dsfont,amsfonts,color,latexsym,CJK}

\usepackage{hyperref}
\usepackage{mathrsfs} 

\definecolor{blue}{rgb}{0.00,0.00,1.00}
\definecolor{red}{rgb}{1.00,0.00,0.00}

\renewcommand{\baselinestretch}{1.2}
\def\bq{\begin{equation}}
	\def\eq{\end{equation}}
\def\ba{\begin{array}{ccc}}
	\def\bal{\begin{array}{lll}}
		\def\ea{\end{array}}

	\def\({\left(}\def\){\right)}
	\def\[{\left[}\def\]{\right]}
	\def\<{\langle}\def\>{\rangle}
	
	\def \C   {\mathbb{C}}

	\def \R   {\mathbb{R}}

	\def\i    {\mathrm{i}}

	\def\S    {\mathbb{S}}
	\def\eps  {\epsilon}
	\def\intr {\int_{\R^3}}
	
	\def\ints {\int_{\S^2}}
	\def\intt {\int^t_0}

	\def \Q    {\mathcal{Q}}
	
	\def\BB{\mathbb{B}}
	\def\AA{\mathbb{A}}

	\def \pt   {\partial}
	
	\def \Dt   {\frac{\rm d}{{\rm d}t}}
	
	\def \dt    {\partial_t}

	\def \dxa   {\partial^{\alpha}_x}

	\def \dvb   {\partial^{\beta}_v}

	\def \divx  {{\rm div}_x}

	\def\Tdx   {\nabla_x}
	\def\Tdv   {\nabla_v}

	\def\bq{\begin{equation}}
		\def\eq{\end{equation}}
	\def\be{\begin{equation}}
		\def\ee{\end{equation}}
	\def\bma#1\ema{{\allowdisplaybreaks\begin{align}#1\end{align}}}
	\def\bmas#1\emas{{\allowdisplaybreaks\begin{align*}#1\end{align*}}}
	\def\bln#1\eln{{\allowdisplaybreaks\begin{aligned}#1\end{aligned}}}
	\def\nnm{\notag}
	\def\bgr#1\egr{\allowdisplaybreaks\begin{gather}#1\end{gather}}
	\def\bgrs#1\egrs{\allowdisplaybreaks\begin{gather*}#1\end{gather*}}

	\theoremstyle{plain}
	\newtheorem{lem}{\bf Lemma}[section]
	\newtheorem{thm}[lem]{\textbf{Theorem}}

	\newtheorem{remark}[lem]{\bf Remark}

\begin{document}
		
		\title{ Diffusion Limit and the optimal convergence rate of the classical solution to the one-species Vlasov-Maxwell-Boltzmann system }
		\author{  Ke Chen$^1$, Anita Yang$^2$, Mingying Zhong$^3$\\[2mm]
			\emph{\small\it $^1$Department of Applied Mathematics, The Hong Kong Polytechnic University,  Hong
				Kong.}\\
			{\small\it E-mail: k1chen@polyu.edu.hk} \\
			\emph{\small\it $^2$Department of Mathematics, The Chinese University of Hong Kong, Hong Kong.}\\
			{\small\it E-mail: ayang@math.cuhk.edu.hk} \\
			{\small\it  $^3$School of  Mathematics and Information Sciences,
				Guangxi University,  China.}\\
			{\small\it E-mail:\ zhongmingying@gxu.edu.cn}\\[5mm]
		}
		\date{ }
		
		\pagestyle{myheadings}
		\markboth{Vlasov-Maxwell-Boltzmann system }%
		{ K. Chen, A. Yang, M.-Y. Zhong }
		
		\maketitle
		
		\thispagestyle{empty}
		
		\begin{abstract}\noindent
			In the present paper, we study the diffusion limit of the strong solution to the one-species  Vlasov-Maxwell-Boltzmann (VMB) system with initial data near a global Maxwellian.  Based on spectral analysis techniques, we prove the convergence and establish the convergence rate of the classical solution to the VMB system towards the solution to the incompressible Navier--Stokes--Maxwell  system with a precise estimation on the initial layer.
			
			
		\end{abstract}

		%
		
		\tableofcontents

		\section{Introduction}
		The Vlasov-Maxwell-Boltzmann system is a fundamental model in plasma physics for the describing the time evolution of dilute charged particles,
		such as electrons and ions,  under the influence of the self-induced Lorentz forces governed by Maxwell equations, cf. \cite{ChapmanCowling}
		for derivation and the physical background.  The collision term in the kinetic equation is the Boltzmann operator that describes the collision of diluted charged particles.  In general, the rescaled one-species VMB system  for incompressible diffusive regimes takes the form \cite{Arsenio}
		\be \label{VMB1a}
		\left\{\bln
		&\dt F_{\eps}+\frac{1}{\eps}v\cdot\Tdx F_{\eps}+\frac{1}{\eps}(\alpha E_{\eps}+\beta v\times B_{\eps})\cdot\Tdv F_{\eps}=\frac{1}{\eps^2}Q(F_{\eps},F_{\eps}) ,\\
		&\gamma\dt E_{\eps}-\Tdx\times B_{\eps}=- \frac{\beta}{\eps^2}\intr F_{\eps} vdv, \\
		&\gamma\dt B_{\eps}+\Tdx\times E_{\eps}=0,\\
		&\Tdx\cdot E_{\eps}= \frac{\alpha}{\eps^2}\(\intr F_{\eps} dv-1\),\quad \Tdx\cdot B_{\eps}=0,
		\eln\right.
		\ee
		where $\eps>0$ is a small parameter related to the mean free path,  and $\alpha,\beta,\gamma$ describe the
		qualitative behaviors of the systems, namely:
		\begin{itemize}
			\item  $\alpha$ measures the electric repulsion according to Gauss's law;
			\item $\beta$ measures the magnetic induction according to Amp\`{e}re's law;
			\item $\gamma$ is the ration of the bulk velocity to the speed of light.
		\end{itemize}
		Notice that these parameters $\alpha,\beta,\gamma$ are  constrained to the relation
		$$\beta=\frac{\alpha\gamma}{\eps}.$$
		In \eqref{VMB1a}, $F_{\eps}=F_{\eps}(t,x,v)$ is the number
		density function of charged particles with  $(t,x,v)\in \R_+\times \R^3_x\times\R^3_v$, and $E_{\eps}(t,x)$, $B_{\eps}(t,x)$ denote the electro and magnetic fields, respectively. Throughout this paper, we assume $\eps\in(0,1)$. The collision between particles is given by the
		standard Boltzmann collision operator $Q(F,G)$ as below
		$$
		Q(F,G)=\intr\ints
		|(v-v_*)\cdot\omega|(F(v')G(v'_*)-F(v)G(v_*))dv_*d\omega,\label{binay_collision}
		$$
		where  $v,v_*$ are the velocities of gas particles before collision and $v',v'_*$ are the velocities after
		collision:
		$$
		v'=v-[(v-v_*)\cdot\omega]\omega,\quad
		v'_*=v_*+[(v-v_*)\cdot\omega]\omega,\quad \omega\in\S^2.
		$$

		In the following, we consider the following scaled one-species VMB ($\alpha = \eps, \beta = \eps, \gamma= \eps$):
		\be \label{VMB1}
		\left\{\bln
		&\dt F_{\eps}+\frac{1}{\eps}v\cdot\Tdx F_{\eps}+ ( E_{\eps}+ v\times B_{\eps})\cdot\Tdv F_{\eps}=\frac{1}{\eps^2} Q(F_{\eps},F_{\eps}),\\
		& \eps\dt E_{\eps}-\Tdx\times B_{\eps}=- \frac{1}{\eps}\intr F_{\eps} vdv, \\
		& \eps\dt B_{\eps}+\Tdx\times E_{\eps}=0,\\
		&\Tdx\cdot E_{\eps}= \frac{1}{\eps}\(\intr F_{\eps}dv-1\),\quad \Tdx\cdot B_{\eps}=0.
		\eln\right.
		\ee

		The Vlasov-Maxwell-Boltzmann (VMB) system has been the subject of extensive research, yielding significant advancements in understanding its mathematical properties  \cite{Duan4,Duan5,Guo4,Jang,Strain}.
		Notably, the global existence of a unique strong solution with initial data near the normalized global Maxwellian has been established in both spatially periodic domains \cite{Guo4} and the three-dimensional whole space \cite{Strain} for hard sphere collisions. This was later extended to general collision kernels, with or without the angular cut-off assumption, in \cite{Duan1,Duan2}.
		Regarding long time behaviors, it was shown in \cite{Duan5} that the total energy of the linearized one-species VMB system decays at the rate $(1+t)^{-\frac38}$, and in \cite{Duan4} that the total energy of nonlinear two-species VMB system decays at the rate $(1+t)^{-\frac34}$.
		The spectrum structure and the optimal decay rate of the global solution  to the VMB systems  for both one-species and two-species were investigated in \cite{Li3}.  The diffusive limit for two-species VMB system  was explored in \cite{Arsenio,Jang,Jiang2,Jiang3,JiangLei}. More recently, applying spectrum analysis techniques, Yang and the third author \cite{YZ} established the optimal convergence rate of the solution for the two-species VMB system to the NSMF system, and obtained the accurate estimation of the initial layer.  By employing the weighted energy method,  the diffusion limit of one species VMB system to the NSMF system  with $(\alpha,\beta,\gamma)=(\eps^2,\eps,1)$ was proved in  \cite{WZGX}.
		
		On the other hand, the diffusion limit to the Boltzmann equation is a classical problem with pioneering work by Bardos-Golse-Levermore in \cite{FL-1}, and significant  progress on the limit of renormalized solutions to the Leray solution to Navier-Stokes system in \cite{Golse-SR}.
		In contrast to the works on Boltzmann equation~\cite{FL-1,FL-2,FL-3,Guo5,FL-4,FL-5}, the VPB system \cite{Guo2,Li1,Wang1} and the VMB system \cite{Arsenio,Jang,Jiang2,JiangLei},  the convergence rate of the classical solution to the  VMB system \eqref{VMB1}  towards its fluid dynamical limits and the estimation of the initial layer have not been given despite of its importance.

		In this paper, we study the diffusion limit
		of the strong solution to the rescaled VMB system~\eqref{VMB1}  with initial data near the equilibrium $(F_*, E_*,B_*)=(M(v),0,0)$, where $M(v)$ is the  normalized global Maxwellian given by
		$$
		M=M(v)=\frac1{(2\pi)^{3/2}}e^{-\frac{|v|^2}2},\quad v\in\R^3.
		$$
		Hence, we define the perturbation of $(F_{\eps}, E_{\eps},B_{\eps})$ as
		$$
		F_{\eps}=M+\eps\sqrt{M}f_{\eps} .
		$$
		Then Cauchy problem of the VMB system~\eqref{VMB1} for $ (f_{\eps}, E_{\eps},B_{\eps})$ can be rewritten as
		\bma
		&\dt f_{\eps}+\frac{1}{\eps}v\cdot\Tdx f_{\eps}-\frac{1}{\eps^2}Lf_{\eps}- \frac{1}{\eps}v\sqrt{M}\cdot E_{\eps}=H_{\eps},\label{VMB4}\\
		& \dt E_{\eps}- \frac{1}{\eps}\Tdx\times B_{\eps}=-  \frac{1}{\eps}\intr f_{\eps}v\sqrt{M}dv,  \label{VMB4b}\\
		& \dt B_{\eps}+ \frac{1}{\eps}\Tdx\times E_{\eps}=0,\label{VMB4c}\\
		&\Tdx\cdot E_{\eps}= \intr f_{\eps}\sqrt{M}dv,\quad \Tdx\cdot B_{\eps}=0,\label{VMB4d}
		\ema
		where the nonlinear term $H_{\eps}$ is defined by
		\be
		H_{\eps}=\frac12 (v\cdot E_{\eps})f_{\eps}-\(E_{\eps}+  v\times B_{\eps}\)\cdot\Tdv f_{\eps}
		+\frac{1}{\eps} \Gamma(f_{\eps},f_{\eps}). \label{VMB5}
		\ee
		The initial condition is given by
		\be f_{\eps}(0,x,v)=f_{0}(x,v)  ,\quad  (E_{\eps},B_{\eps})(0,x)=(E_0, B_0)(x) ,  \label{VMB2i}\ee
		by assuming that the initial data $(f_{0}, E_0,B_0)$ is independent of $\eps$. The initial data should satisfy the compatibility conditions
		\be  \Tdx\cdot E_0(x)=\intr f_{0}\sqrt{M}dv,\quad  \Tdx\cdot B_0(x)=0. \label{com}\ee
		In \eqref{VMB4}--\eqref{VMB5}, the linear operator $L$ and the nonlinear operator $\Gamma(f,g)$ are defined by
		\be \label{gamma}
		\left\{\bln
		&Lf=\frac1{\sqrt M}[Q(M,\sqrt{M}f)+Q(\sqrt{M}f,M)],\\
		&\Gamma(f,g)=\frac1{\sqrt M}Q(\sqrt{M}f,\sqrt{M}g).
		\eln\right.
		\ee
		
		The linearized  operator $L$ can be written as (cf. \cite{Cercignani})
		\be\label{L_1}
		\left\{\bln
		&(Lf)(v)=(Kf)(v)-\nu(v) f(v),\quad (Kf)(v)=\intr k(v,v_*)f(v_*)dv_*,\\
		&\nu(v)=\sqrt{2\pi}\bigg(e^{-\frac{|v|^2}2}+\(|v|+\frac1{|v|}\)\int^{|v|}_0e^{-\frac{|u|^2}2}du\bigg),\\
		&k(v,v_*)=\frac2{\sqrt{2\pi}|v-v_*|}e^{-\frac{(|v|^2-|v_*|^2)^2}{8|v-v_*|^2}-\frac{|v-v_*|^2}8} -\frac{|v-v_*|}{2\sqrt{2\pi}}e^{-\frac{|v|^2+|v_*|^2}4},
		\eln\right.
		\ee
		where $\nu(v)$,  the collision frequency, is a real  function, and $K$ is a self-adjoint compact operator
		on $L^2(\R^3_v)$ with  real symmetric integral kernels $k(v,v_*)$.
		In addition, $\nu(v)$ satisfies
		\be
		\nu_0(1+|v|) \leq\nu(v)\leq \nu_1(1+|v|). \label{nu}
		\ee
		The nullspace of the operator $L$, denoted by $N_0$, is a subspace
		spanned by the orthonormal basis $\{\chi_j,\ j=0,1,\cdots,4\}$  with
		\bq \chi_0=\sqrt{M},\quad \chi_j=v_j\sqrt{M} \,\, (j=1,2,3), \quad
		\chi_4=\frac{(|v|^2-3)\sqrt{M}}{\sqrt{6}}. \label{basis}\eq
		Let   $L^2(\R^3)$ be a Hilbert space of complex-value functions $f(v)$
		on $\R^3$ with the inner product and the norm
		$$
		(f,g)=\intr f(v)\overline{g(v)}dv,\quad \|f\|=\(\intr |f(v)|^2dv\)^{1/2}.
		$$
		And let $P_{ 0}$ be the projection operators from $L^2(\R^3_v)$ to the subspace $N_0$ with
		\be
		P_{ 0}f=\sum_{j=0}^4(f,\chi_j)\chi_j,\quad P_1=I-P_{ 0}. \label{P10}
		\ee
		For any $U=(f,X,Y)\in L^2\times \R^3\times \R^3$, define the projections ${ P_A}$ and ${P_B}$ by
		\be { P_A}U=(P_0f, X,Y), \quad { P_B}U=(I-{ P_A})U=(P_1f, 0,0). \label{defp23} \ee
		From the Boltzmann's H-theorem, the linearized collision operators $L$ and $L_1$ are non-positive, precisely,  there is a constant $\mu>0$ such that \be
		(Lf,f)\leq -\mu \| P_1f\|^2, \quad  \ f\in D(L),\label{L_4}
		\ee
		where $D(L)$ is the domain of $L$ given by
		$$ D(L)=\left\{f\in L^2(\R^3)\,|\,\nu(v)f\in L^2(\R^3)\right\}.$$
		Without loss of generality, we assume in this paper that $\nu(0)\ge \nu_0\ge \mu>0$.
		
		This paper  aims to prove the
		convergence and establish the convergence rate of the classical solution $(f_{\eps}, E_{\eps},B_{\eps})$ to the system \eqref{VMB4}--\eqref{VMB2i} towards  $(u_1,E,B)$, where $u_1=n \chi_0+m\cdot v\chi_0+q\chi_4$  with $(n,m,q, E,B)(t,x)$ being the solution to the following incompressible Navier-Stokes-Maxwell-Fourier (NSMF) system:
		\be \label{NSM_2}
		\left\{\bal
		\Tdx\cdot m=0,\quad \Tdx n+ \sqrt{\frac23}\Tdx q- E =0, \\
		\dt (m-\Delta^{-1}_xm)-\kappa_0\Delta_x m +\Tdx p = n E+m\times B-m\cdot\Tdx m, \\
		\dt (q-\sqrt{\frac23}n) -\kappa_1\Delta_x q= \sqrt{\frac23} m\cdot E- \frac{5}{3}m\cdot\Tdx q, \\
		E=\Tdx \Delta_x^{-1} n, \quad \Tdx\times B=m, \quad \Tdx\cdot B=0,
		\ea\right.
		\ee
		where $p$ is the pressure, and the initial data $(n,m,q,E,B)(0)$ satisfies
		\be\label{NSP_5i}
		\left\{\bal
		m(0) = (f_{0},v\chi_0)-\Delta^{-1}_x\Tdx\divx (f_{0},v\chi_0),\\
		q(0)-\sqrt{\frac23}n(0)= (f_0, \chi_4-\sqrt{\frac23}\chi_0), \\
		\Tdx n(0)+\sqrt{\frac23}\Tdx  q(0) -E(0)=0, \ \   \Tdx\times B(0)=m(0),\\
		E(0)=\nabla_x\Delta_x^{-1}n(0),\ \ \nabla_x\cdot B(0)=0.
		\ea\right.
		\ee
		Here,  the viscosity coefficients $\kappa_0$, $\kappa_1>0$ are defined by
		\be\label{coe}
		\kappa_0=-(L^{-1}P_1(v_1\chi_2),v_1\chi_2),\quad \kappa_1=-(L^{-1}P_1(v_1\chi_4),v_1\chi_4).
		\ee

		In general, the convergence is not uniform near $t=0$ because of the appearance of  an initial layer.
		However, we can show that if the initial data $(f_{0}, E_0,B_0)$ satisfies
		\be \label{initial}
		\left\{\bal
		f_{0}(x,v)=n_{0}(x)\chi_0+m_{0}(x)\cdot v\chi_0+q_{0}(x)\chi_4,\\
		m_{0}=\nabla_x\times B_0, \ \ \Tdx n_{0} + \sqrt{\frac23} \Tdx  q_{0}-E_0=0, \\
		E_0=\nabla_x\Delta_x^{-1}n_0,\quad \Tdx\cdot B_0=0,
		\ea\right.
		\ee
		then the convergence is uniform  up to $t=0$.
		
		\bigskip

		\noindent\textbf{Notations:} \ \ Before presenting the main results of this paper, we list some notations. For any $\alpha=(\alpha_1,\alpha_2,\alpha_3)\in \mathbb{N}^3$ and $\beta=(\beta_1,\beta_2,\beta_3)\in \mathbb{N}^3$, denote
		$$\dxa=\pt^{\alpha_1}_{x_1}\pt^{\alpha_2}_{x_2}\pt^{\alpha_3}_{x_3},\quad \dvb=\pt^{\beta_1}_{v_1}\pt^{\beta_2}_{v_2}\pt^{\beta_3}_{v_3}.$$
		The Fourier transform of $f=f(x,v)$
		is denoted by
		$$\hat{f}(\xi,v)=\mathcal{F}f(\xi,v)=\frac1{(2\pi)^{3/2}}\intr f(x,v)e^{-  i x\cdot\xi}dx.$$
		
		For any $q\in [1,\infty]$, we define the Sobolev Space $L^{q}=L^q_x(L^2_v)$  for function $f=f(x,v)$ or $L^{q}=L^q_x(L^2_v)\times L^q_x\times L^q_x$ for vector $U=(g(x,v),E(x),B(x)) $ with the norm
		\bmas
		\|f\|_{L^{q}}&=\bigg(\intr\bigg(\intr|f(x,v)|^2 dv\bigg)^{q/2}dx\bigg)^{1/q},\\
		\|U\|_{L^q}&=\bigg(\intr\bigg(\intr|g(x,v)|^2 dv\bigg)^{q/2}dx\bigg)^{1/q}+\bigg(\intr (|E(x)|^q+|B(x)|^q) dx\bigg)^{1/q}.
		\emas
		For any integer $k\ge 1$ and $q\in [1,\infty]$, define the Sobolev Space $H^{k}=H^{k}_x(L^2_v)$ (or $H^{k}=H^{k}_x(L^2_v)\times H^{k}_x\times H^{k}_x$) for any function $f=f(x,v)$ (or  vector $U=(g(x,v),E(x),B(x)) $) with the norm
		\bmas
		\|f\|_{H^k}&=\(\intr (1+|\xi|^2)^k\intr |\hat{f}|^2dv d\xi \)^{1/2}, \\
		\|U\|_{H^k}&=\(\intr (1+|\xi|^2)^k \(\intr |\hat{g}|^2dv+ |\hat E|^2+| \hat B|^2\)d\xi\)^{1/2},
		\emas
		and denote  $ X^k_{l} $ ($X^k=X^k_0$) the weighted Sobolev space equipped with the norm
		\bmas
		\|f\|_{X^k_{l}}&=\sum_{|\alpha|+|\beta|\le k}\|\nu^l\dxa\dvb f\|_{L^2 },\\
		\|U\|_{X^k_{l}}&=\sum_{|\alpha|+|\beta|\le k}\|\nu^l\dxa\dvb g\|_{L^2 }+ \sum_{|\alpha| \le k}(\|\dxa E\|^2_{L^2_x }+\| \dxa B\|^2 _{L^2_x }).
		\emas

		
		
		We have the following main results.
		The first theorem presents the unique existence of solution for the VMB system and the NSMF system, respectively.
		
		\begin{thm}\label{thm1.1}
			For any $\eps\in (0,1)$, there exists a small constant $\delta_0>0$ such that if the initial data $U_0=(f_0,E_0,B_0)$ satisfy that $\|U_{0}\|_{ X^5_1 }+\|U_{0}\|_{L^1 } \le \delta_0$,  then the VMB system \eqref{VMB4}--\eqref{VMB2i}  admits a unique global solution  $U_{\eps}(t,x,v)= (f_{\eps},E_{\eps},B_{\eps})$ satisfying the following  time-decay estimate:
			\be
			\|U_{\eps}(t)\|_{X^4_1} \le C\delta_0 (1+t)^{-\frac38}.
			\ee
			
			There exists a small constant $\delta_0>0$ such that if $\|U_{0}\|_{H^4 }+\|U_{0}\|_{L^1 } \le \delta_0$, then
			the NSMF system \eqref{NSM_2}--\eqref{NSP_5i} admits a unique global solution $\tilde{U}(t,x)=(n,m,q,E,B) $  satisfying the following  time-decay estimate:
			\be
			\|\tilde{U}(t)\|_{H^4_x} \le C\delta_0 (1+t)^{-\frac38}. 
			\ee
		\end{thm}
		
		The convergence rate of the solution to the VMB system towards the solution to the NSMF system is characterised by the following theorem.
		\begin{thm}\label{thm1.2}
			There exist small positive constants $\delta_0$ and $\eps_0$ such that if the initial data $U_0=(f_0,E_0,B_0)$ satisfies that $\|U_{0}\|_{X^6_1 }+\|U_{0}\|_{L^1 } \le \delta_0$, then there exists a unique function $U_1=(u_1,E,B) $ such that for any $\eps,\sigma\in (0,\eps_0)$,  the solution $U_{\eps}=(f_{\eps},E_{\eps},B_{\eps})$  to the VMB system \eqref{VMB4}--\eqref{VMB2i} satisfies
			\bma  \|U_{\eps} (t)-U_1 (t)\|_{L^{\infty} }\le C\delta_0\bigg(\eps |\ln\eps|^4(1+t)^{-\frac{5-\sigma}8} +\(1+ \frac{t}{\eps}\)^{-1}\bigg), \label{limit0}
			\ema
			where $u_1=n \chi_0+m\cdot v\chi_0+q\chi_4$  with $(n,m,q, E,B)(t,x)$ being the solution to the incompressible NSMF system \eqref{NSM_2}--\eqref{NSP_5i}.
			
			Moreover, if the initial data $U_0$ satisfies \eqref{initial} and   $\|U_{0}\|_{ H^6} +\|U_{0}\|_{L^1}\le \delta_0$, then we have
			\be
			\|U_{\eps} (t)-U_1 (t)\|_{L^{\infty} }\le C\delta_0\eps (1+t)^{-\frac{5-\sigma}8}. \label{limit_1a}
			\ee
		\end{thm}
		
		\begin{remark}\label{remark1.4}
			Under the first assumption of Theorem \ref{thm1.2}, we have
			\be
			\|U_{\eps}(t)-U_1(t)-U^{osc}_{\eps}(t)-e^{\frac{t}{\eps^2}\AA_{\eps}}P_BU_0\|_{L^{\infty}} \le C\delta_0\eps (1+t)^{-\frac{5-\sigma}8},
			\ee
			where $U^{osc}_{\eps}(t)=U^{osc}_{\eps}(t,x,v)$ is the high oscillation part of $U_{\eps}$  defined by \eqref{osc}. Hence, $U^{osc}_{\eps}(t) $ and $e^{\frac{t}{\eps^2}\AA_{\eps}}P_BU_0$ are the essential components for generating the initial layer.
		\end{remark}
		
		In the remainder of the introduction, we outline the core concepts and analytical methodology underlying our proof.
		The convergence rates given in Theorem~\ref{thm1.2}  for the diffusion limit of the VMB system  are proved based on spectral analysis \cite{Li4} and the ideas inspired by \cite{FL-3,Li1}.
		First of all, the solution $U_{\eps}(t)=(f_{\eps},E_{\eps},B_{\eps})(t)$  to the VMB system \eqref{VMB4}--\eqref{VMB2i} can be represented by
		$$
		U_{\eps}(t)=e^{\frac{t}{\eps^2}\AA_{\eps}}U_0+\intt e^{\frac{t-s}{\eps^2}\AA_{\eps}} (H_{\eps}(s),0,0)  ds,
		$$
		and the solution $U_1(t)=(u_1,E,B )(t)$ with $u_1 =n\chi_0+m\cdot v\chi_0+q\chi_4$ to the NSMF system  \eqref{NSM_2}--\eqref{NSP_5i} can be represented by
		$$
		U_1(t)=Y_1(t)P_AU_0+\intt Y_1(t-s) (H_3(s),0,0) ds,
		$$
		where $\AA_{\eps}$ is the linear  VMB operators defined by \eqref{B5}, $Y_1(t)$ is the semigroup defined in \eqref{v1}, and $H_3$ is the nonlinear term given by
		$$
		H_3=\(nE+m\times B-m\cdot\Tdx m\)\cdot v\chi_0+\bigg(\sqrt{\frac23}m\cdot E-\frac53m\cdot\Tdx q\bigg) \chi_4.
		$$
		Our proof proceeds in two key steps. First, we estimate the convergence rate from $e^{\frac{t}{\eps^2}\AA_{\eps}}$ to $Y_1(t)$ by using spectral analysis. Then the convergence rates from $U_{\eps}(t)$ to $U_{1}(t)$  is established by combining  semigroup convergence estimates with a bootstrap argument. 

		Due to the influence of the electric-magnetic field,  the linear VMB operator $\AA_{\eps}(\xi)$ given in \eqref{B3} has no scaling property.
		To study the corresponding eigenvalue problem, we will use  a novel non-local implicit function theorem to show that for $\eps |\xi| $ being small, there exist nine eigenvalues $\lambda_j(|\xi|,\eps)$, $-1\le j\le 7$ of $\AA_{\eps}(\xi)$ with expansions (See Lemma \ref{eigen_4a}):
		\bma
		\lambda_{j}(|\xi|,\eps) &=\eps \eta_j(|\xi|)- \eps^2b_{j}(|\xi|) +O(\eps^3|\xi|^3),\quad j=-1,0,1,4,5,6,7,\label{expan2}\\
		\lambda_{k}(|\xi|,\eps) &=\eps \eta_k(|\xi|)- \eps^2b_{k}(|\xi|) +O\(\frac{\eps^3|\xi|^5}{1+|\xi|^2}\), \quad k=2,3,\label{expan3}
		\ema
		where $\eta_j(|\xi|)$ and $b_j(|\xi|)$ are defined by \eqref{bj}.
		
		Moreover, we can decompose the semigroup $e^{ \frac{t}{\eps^2}\AA_{\eps}(\xi)}$ into two fluid parts for low frequency and high frequency and the remainder part, where the remainder part has the decay rate $e^{-\frac{dt}{\eps^2}}$ (See Theorem \ref{rate1}). Then by applying the expansion \eqref{expan2}--\eqref{expan3} to the fluid part, we obtain
		\bmas
		e^{ \frac{t}{\eps^2}\AA_{\eps}(\xi)}\hat{U}_0&=\sum_{j=0,2,3} e^{-b_j(|\xi|)t }\tilde{P}_{0j}\hat{U}_0 +\sum_{j\ne 0,2,3} e^{\frac{\eta_j(|\xi|)}{\eps}t-b_j(|\xi|)t }\tilde{P}_{0j}\hat{U}_0\\
		&\quad+O(\eps e^{- b_jt}) +O( e^{-\frac{ct}{\eps|\xi|}})1_{\{|\xi|\ge \frac{r_1}{\eps}\}}+O(e^{-\frac{dt}{\eps^2}}),
		\emas
		where $\tilde{P}_{0j}$, $-1\le j\le 7$ are first order eigenprojections corresponding to $\lambda_j$. Crucially, we employ  the following key estimate:
		$$\bigg\|\mathcal{F}^{-1}\(e^{\frac{\pm i\sqrt{1+|\xi|^2}}{\eps}t}(1+|\xi|^2)^{-\frac54}\)\bigg\|_{L^p_x}\le C \(\frac{t}{\eps}\)^{\frac{3}{p}-\frac32},\quad  p\in [2, \infty]$$
		to establish the optimal convergence rate of the semigroup $e^{\frac{t}{\eps^2}\AA_{\eps}}$ to its first and second order fluid limits in $L^\infty$ norm as listed in Lemmas \ref{fl1} and \ref{fl2}.

		By using the estimates on the convergence rates for the fluid limits of the linear VMB system, we can  prove the convergence and establish the optimal convergence rate of the strong solution $U_{\eps}$ to the nonlinear VMB system towards the solution $U_1$ to the NSMF system.  Hence, we obtain the precise estimation on the initial layer.

		The rest of this paper will be organized as follows.
		In Section~\ref{sect2}, we present the results about the spectrum analysis of the linear operator related to the linearized VMB system. In Section~\ref{sect3}, we  establish  the first and second order fluid approximations of the solution to the  linearized VMB system.
		In Section~\ref{sect4}, we prove the convergence and establish the convergence rate of the global solution to the original nonlinear VMB system towards the solution to the nonlinear NSMF system.

		\section{Spectral analysis}
		\setcounter{equation}{0}
		\label{sect2}

		In this section, we are concerned with the spectral analysis of the linear VMB operator $\AA_{\eps}(\xi)$ defined by
		\eqref{Axi}, which will be applied to study diffusion limit of the solution to the VMB system \eqref{VMB4}--\eqref{VMB4d}.

		From the system \eqref{VMB4}--\eqref{VMB4d}, we have the following linearized VMB system for  $U_{\eps}=(f_{\eps},E_{\eps},B_{\eps})^T$ as
		\be\label{LVMB1}
		\left\{\bln
		&\eps^2\dt U_{\eps}=\AA_{\eps}U_{\eps},\quad t>0,\\
		&\Tdx\cdot E_{\eps}=(f_{\eps},\chi_0),\quad \Tdx\cdot B_{\eps}=0,\\
		&U_{\eps}(0,x,v)=U_0(x,v)=(f_0,E_0,B_0),
		\eln\right.
		\ee
		where
		\be
		\AA_{\eps}=\left( \ba
		L-\eps v\cdot\Tdx & \eps v\chi_0 &0\\
		-\eps P_m &0 &\eps\Tdx\times\\
		0 &-\eps\Tdx\times &0
		\ea\right) \label{B5}
		\ee with $P_mh=(h,v\chi_0)$ for any $h\in L^2(\R^3_v)$.

		Taking Fourier transform to  \eqref{LVMB1} in $x$ to get
		\be\label{LVMB3}
		\left\{\bln
		&\eps^2\dt \hat{U}_{\eps}= \AA_{\eps}(\xi)\hat U_{\eps},\quad t>0,\\
		&\Tdx\cdot \hat{E}_{\eps}=(\hat{f}_{\eps},\chi_0),\quad \Tdx\cdot \hat{B}_{\eps}=0,\\
		&\hat{U}_{\eps}(0,\xi,v)=\hat{U}_0(\xi,v)=(\hat{f}_0,\hat{E}_0,\hat{B}_0),
		\eln\right.
		\ee
		where  
		\be
		\AA_{\eps}(\xi)=\left( \ba
		L- i\eps v\cdot\xi & \eps v\chi_0 &0\\
		-\eps P_m &0 &i\eps\xi\times\\
		0 &-i\eps\xi\times &0
		\ea\right). \label{B3}
		\ee
		
		By the identity $X=(X\cdot y)y-y\times y\times X$ for any $X\in \R^3$ and $y\in \mathbb{S}^2$, we can transform the system \eqref{LVMB3} to a new system for $\hat V_{\eps}=(\hat f_{\eps},\omega\times\hat  E_{\eps},\omega\times\hat B_{\eps})^T$ with $\omega=\xi/|\xi|$ as
		\bq
		\left\{\bln            \label{LVMB2a}
		&\dt \hat{V}_{\eps}=\tilde{\AA}_{\eps}(\xi)\hat{V}_{\eps}, \quad t>0,\\
		&\hat{V}_{\eps}(0,\xi,v)=\hat{V}_0(\xi,v)=(\hat{f}_0,\omega\times\hat{E}_0,\omega\times\hat{B}_0),
		\eln\right.
		\eq
		with 
		\be \label{Axi}
		\tilde{\AA}_{\eps}(\xi)=\left( \ba
		\tilde\BB_{\eps}(\xi) &-\eps v\chi_0\cdot\omega\times &0\\
		-\eps\omega\times P_m &0 &i\eps\xi\times\\
		0 &-i\eps\xi\times &0
		\ea\right).
		\ee
		Here,  for $\xi\ne0$,
		\be
		\tilde\BB_{\eps}(\xi) =L-i\eps v\cdot\xi -i\eps \frac{ v\cdot\xi }{|\xi|^2}P_{ d}.  \label{B(xi)}
		\ee
		
		\begin{remark}\label{rem1.1}
			The eigenvalues of the operator $\AA_{\eps}(\xi)$ are the same as those
			of $\tilde{\AA}_{\eps}(\xi)$,  and the eigenfunctions of $\AA_{\eps}(\xi)$ can be obtained as linear combinations
			of those for  $\tilde{\AA}_{\eps}(\xi)$. In fact, let   $\beta$  be an eigenvalue with
			the corresponding eigenvector denoted by $\mathcal{U}=(\phi,E,B)$ of $\tilde{\AA}_{\eps}(\xi)$. Then $U=(\phi,-i\frac{\xi}{|\xi|^2}(\phi,\chi_0)-\frac{\xi}{|\xi|}\times E,-\frac{\xi}{|\xi|}\times B)$ is the corresponding
			eigenvector with the eigenvalue  $\beta$ of  $\AA_{\eps}(\xi)$.
		\end{remark}

		
		\subsection{Spectrum structure}
		Introduce a weighted Hilbert space $L^2_{\xi}(\R^3)$ for $\xi\ne 0$
		as
		$$
		L^2_{\xi}(\R^3)=\Big\{f\in L^2(\R^3)\,|\,\|f\|_{\xi}=\sqrt{(f,f)_{\xi}}<\infty\Big\},
		$$
		with the inner product defined by
		$$
		(f,g)_{\xi}=(f,g)+\frac1{|\xi|^2}(P_{d} f,P_{d} g).
		$$
		For any fixed $\xi\ne 0$, define a subspace of $\C^3$ by
		$$\mathbb{C}^3_{\xi}=\{y\in \mathbb{C}^3\,|\, y\cdot \xi=0\}.$$
		
		For any vectors $U=(f,E_1,B_1),V=(g,E_2,B_2)\in L^2_\xi(\R^3_v)\times \mathbb{C}^3\times \mathbb{C}^3$,  define a weighted inner product and the corresponding
		norm by
		$$ (U,V)_\xi=(f,g)_\xi+(E_1,E_2)+(B_1,B_2),\quad \|U\|_\xi=\sqrt{(U,U)_\xi}, $$
		and another $L^2$ inner product and norm by
		$$ (U,V)=(f,g)+(E_1,E_2)+(B_1,B_2),\quad \|U\|=\sqrt{(U,U)}. $$
		
		Since $P_{d}$ is a self-adjoint projection operator, it follows that
		$(P_{d} f,P_{d} g)=(P_{d} f, g)=( f,P_{d} g)$ and hence
		\bq (f,g)_{\xi}=\(f,g+\frac1{|\xi|^2}P_{d}g\)=\(f+\frac1{|\xi|^2}P_{d}f,g\).\label{C_1a}\eq
		By
		\eqref{C_1a}, we have for any $f,g\in L^2_{\xi}(\R^3_v)\cap D(\tilde\BB_{\eps}(\xi))$,
		\be
		(\tilde\BB_{\eps}(\xi)f,g)_{\xi}=\(\tilde\BB_{\eps}(\xi) f,g+\frac1{|\xi|^2}P_{d} g\)=(f,\tilde\BB_{\eps}(-\xi)g)_{\xi}. \label{L_7}
		\ee
		Moreover,  $\tilde\BB_{\eps}(\xi)$ is a dissipative operator in $L^2_{\xi}(\R^3)$:
		$$
		{\rm Re}(\tilde\BB_{\eps}(\xi)f,f)_{\xi}=(Lf,f)\le 0.
		$$

		Note that $\tilde\BB_{\eps}(\xi)$ is a linear operator from the space $L^2_{\xi}(\R^3)$ to itself, and  for any $ y\in \mathbb{C}^3_{\xi}$,
		\bq \frac{\xi}{|\xi|}\times\frac{\xi}{|\xi|}\times y=-y .\label{rotat}\eq
		It is easy to verify that $L^2_{\xi}(\R^3_v)\times \mathbb{C}^3_{\xi}\times \mathbb{C}^3_{\xi}$ is an invariant subspace of the operator $\tilde{\AA}_{\eps}(\xi)$. Thus,  $\tilde{\AA}_{\eps}(\xi)$ can
		be regarded as a linear operator on $L^2_{\xi}(\R^3_v)\times \mathbb{C}^3_{\xi}\times \mathbb{C}^3_{\xi}$.

		Denote the spectrum of the operator $A$ by  $\sigma(A)$.
        The essential spectrum of $A$, denoted by $\sigma_{ess}(A)$, is the set
        $\{\lambda\in \C \,|\, \lambda-A~{\rm is~not~a~Fredholm~operator}\}$ (cf. \cite{Kato}). The discrete spectrum
        of $A$, denoted by $\sigma_d(A)$, is the set $\sigma(A)\setminus \sigma_{ess} (A)$ which consists of  all isolated eigenvalues with finite multiplicity.
        And  $\rho(A)$ denotes the resolvent set of the operator $A.$
		
		
		\begin{lem}\label{SG_1}
			The operator $\tilde{\AA}_{\eps}(\xi)$ generates a strongly continuous contraction semigroup on
			$L^2_{\xi}(\R^3)\times \C^3_{\xi}\times \C^3_{\xi}$, which satisfies
			$$
			\|e^{t\tilde{\AA}_{\eps}(\xi)}U\|_{\xi}\le\|U\|_{\xi}, \quad \forall\, t>0,\,U\in
			L^2_{\xi}(\R^3_v)\times \C^3_{\xi}\times \C^3_{\xi}.
			$$
		\end{lem}
		\begin{proof}
			We first show that both $\tilde{\AA}_{\eps}(\xi)$ and $\tilde{\AA}_{\eps}(\xi)^*$ are
			dissipative operators on $L^2_{\xi}(\R_v^3)\times \C^3_{\xi}\times \C^3_{\xi}$. By \eqref{L_7}, we
			obtain for any $U,V\in D(\tilde\BB_{\eps}(\xi))\times \mathbb{C}^3_{\xi}\times \mathbb{C}^3_{\xi}$ that
			$$
			(\tilde{\AA}_{\eps}(\xi)U,V)_{\xi}=(U,\tilde{\AA}_{\eps}(\xi)^*V)_{\xi},
			$$
			where
			\begin{align*}
				\tilde{\AA}_{\eps}(\xi)^*=\left( \ba
				\tilde\BB_{\eps}(-\xi) &\eps v\chi_0\cdot\omega\times   &0\\
				\eps\omega\times P_m &0 & -i\eps\xi\times\\
				0 &  i\eps\xi\times &0
				\ea\right).
			\end{align*}
			Moreover, both $\tilde{\AA}_{\eps}(\xi)$ and $\tilde{\AA}_{\eps}(\xi)^*$ are dissipative,  namely,
			$$\mathrm{Re}(\tilde{\AA}_{\eps}(\xi)U,U)_{\xi}=\mathrm{Re}(\tilde{\AA}_{\eps}(\xi)^*U,U)_{\xi}=(Lf,f)\leq0,\quad \forall\, U=(f,X,Y).$$
			Since $\tilde{\AA}_{\eps}(\xi)$ is a densely defined closed operator, it
			follows from Lemma \ref{S_1-1} that the operator $\tilde{\AA}_{\eps}(\xi)$
			generates a $C_0$-contraction semigroup on $L^2_{\xi}(\R^3_v)\times \mathbb{C}^3_{\xi}\times \mathbb{C}^3_{\xi}$.
		\end{proof}

		Set
		\begin{align*}  D_{\eps}(\xi)=-\nu(v)-i\eps v\cdot\xi, \quad N_1(\xi)=\left(\ba  0 & i\xi\times \\ -i\xi\times & 0 \ea\right)_{6\times6}.
		\end{align*}
		Since $\mathbb{C}^3_{\xi}\times \mathbb{C}^3_{\xi}$ is an invariant subspace of the operator $N_1(\xi)$, we can regard $N_1(\xi)$ as an operator on $\mathbb{C}^3_{\xi}\times \mathbb{C}^3_{\xi}$. 
		Moreover, the operator $\lambda-N_1(\xi)$ is invertible on $\mathbb{C}^3_{\xi}\times \mathbb{C}^3_{\xi}$  for any $\lambda\ne \pm i|\xi|$ and satisfies (cf. \cite{Li4})
		\bq  \|(\lambda-N_1(\xi))^{-1}\|= \max_{j=\pm1}|\lambda-ji|\xi||^{-1}.\label{b_1(xi)}\eq

		\begin{lem}\label{Egn}
			The following conditions hold for all $\xi\ne 0$ and $\eps\in [0,1)$.
			\begin{enumerate}
				\item[\rm (1)]
				$\sigma_{ess}(\tilde{\AA}_{\eps}(\xi))\subset \{\lambda\in \mathbb{C}\,|\, {\rm Re}\lambda\le -\nu_0\}$ and $\sigma(\tilde{\AA}_{\eps}(\xi))\cap \{\lambda\in \mathbb{C}\,|\, -\nu_0<{\rm Re}\lambda\le 0\}\subset \sigma_{d}(\tilde{\AA}_{\eps}(\xi))$.
				\item[\rm (2)]
				If $\lambda$ is an eigenvalue of $\tilde{\AA}_{\eps}(\xi)$, then ${\rm Re}\lambda<0$ for any $\eps \ne 0$  and $ \lambda=0$ if and only if $\eps=0$.
			\end{enumerate}
		\end{lem}
		
		\begin{proof}
			We decompose $\tilde{\AA}_{\eps}(\xi)$ into
			$$ \tilde{\AA}_{\eps}(\xi)=G^1_{\eps}(\xi)+G^2_{\eps}(\xi), $$
			where
			\bma
			G^1_{\eps}(\xi)&=\left( \ba
			D_{\eps}(\xi) &0 &0\\
			0 &0 &i\eps\xi\times\\
			0 &-i\eps\xi\times &0
			\ea\right), \label{A1}\\
			G^2_{\eps}(\xi)&=\left( \ba
			K-i\eps \frac{v\cdot\xi}{|\xi|^2}P_{d} &-\eps v\chi_0\cdot\omega\times &0\\
			-\eps\omega\times P_m &0 &0\\
			0 &0 &0
			\ea\right). \label{A2}
			\ema
			By  \eqref{b_1(xi)} and \eqref{nu}, the operator $\lambda-G^1_{\eps}(\xi)$ is invertible  on $L^2_{\xi}(\R^3_v)\times \mathbb{C}^3_{\xi}\times \mathbb{C}^3_{\xi}$ for ${\rm
				Re}\lambda>-\nu_0$ and $\lambda\ne \pm i\eps^2|\xi|$, and it satisfies
			$$
			(\lambda-G^1_{\eps}(\xi))^{-1}= \left(\ba (\lambda -D_{\eps}(\xi))^{-1} & 0 \\  0 & (\lambda -\eps N_1(\xi))^{-1} \ea\right)_{7\times7}.
			$$
			Since $G^2_{\eps}(\xi)$ is a compact
			operator on $L^2_{\xi}(\R^3_v)\times \mathbb{C}^3_{\xi}\times \mathbb{C}^3_{\xi}$ for any fixed $\xi\ne 0$,
			$\tilde{\AA}_{\eps}(\xi)$ is a compact perturbation of $G^1_{\eps}(\xi)$. Hence, it follows from
			Weyl's Theorem (Theorem 5.35 on p.244 of \cite{Kato}) that  
			$$\sigma_{ess}(\tilde{\AA}_{\eps}(\xi))=\sigma_{ess}(G^1_{\eps}(\xi))={\rm Ran}(D_{\eps}(\xi)).$$
			Thus
			the spectrum of $\tilde{\AA}_{\eps}(\xi)$ in the domain ${\rm Re}\lambda>-\nu_0$ consists of
			discrete eigenvalues $\lambda_j(\xi)$ with possible accumulation
			points only on the line ${\rm Re}\lambda= -\nu_0$. This proves the part (1).

			We claim that for any  $\lambda \in \sigma_d(\tilde{\AA}_{\eps}(\xi))$  in the region $\mathrm{Re}\lambda>-\nu_0$, it holds that ${\rm Re}\lambda <0$
			for $\eps \ne 0$. Indeed, set $s=|\xi|$, $\omega=\xi|\xi|^{-1}$, and let $U=(f,E,B)\in L^2_{\xi}(\R^3_v)\times \mathbb{C}^3_{\xi}\times \mathbb{C}^3_{\xi}$ be the
			eigenvector corresponding to the eigenvalue $\lambda$ so that
			\bq           \label{L_6}
			\left\{\bal
			\lambda f=Lf-i\eps s(v\cdot\omega) (f+\frac1{s^2}P_{d} f )-\eps v\chi_0\cdot(\omega\times E),\\
			\lambda E=-\eps\omega\times (f,v\chi_0)+i \eps\xi\times B,\\
			\lambda B=-i\eps\xi\times E.
			\ea\right.
			\eq
			Taking the inner product  $\eqref{L_6}_1$ with $f+\frac1{s^2}P_df$, we have
			$$
			\text{Re}\lambda\(\|f\|^2_{\xi} +|E|^2+|B|^2\)= (Lf,f)\le 0,
			$$
			which  implies $\text{Re}\lambda\leq 0$.

			Furthermore, if there exists an eigenvalue $\lambda$ with ${\rm
				Re}\lambda=0$, then it follows from the above that $(Lf,f)=0$,
			namely, 
			$f=C_0\sqrt M+C\cdot v\chi_0+C_4\chi_4\in N_0$. Substitute this into $\eqref{L_6}_1$, we obtain
			$$
			\begin{aligned}
				\lambda f&=-i\eps s  P_0(v\cdot\omega) \(f+\frac1{s^2}C_0\chi_0\)-\eps v\chi_0\cdot(\omega\times E),\\
				0&=-i\eps s P_1(v\cdot\omega)(C\cdot v\chi_0+C_4\chi_4) ,
			\end{aligned}
			$$
			which   implies that $f= 0$ and $\omega\times E=0$ unless $\eps= 0$ and $\lambda= 0$. When $\eps\ne 0$,  $f= 0$ and $\omega\times E= 0$ and hence $E=0$ and $B\equiv0$. This is a contradiction and
			thus it holds $\text{Re}\lambda<0$ for $\eps \ne 0$. This proves the part (2) and completes the proof of the lemma.
		\end{proof}
		
		Now denote by $T$ a linear operator on $L^2(\R^3_v)$ or
		$L^2_{\xi}(\R^3_v)$, and we define the corresponding norms of $T$ by
		$$
		\|T\|=\sup_{\|f\|=1}\|Tf\|,\quad
		\|T\|_{\xi}=\sup_{\|f\|_{\xi}=1}\|Tf\|_{\xi}.
		$$
		Obviously,
		\be
		(1+|\xi|^{-2})^{-1/2}\|T\|\le \|T\|_{\xi}\le (1+|\xi|^{-2})^{1/2}\|T\|.\label{eee}
		\ee

		First, we consider the spectrum and resolvent sets of $\tilde{\AA}_{\eps}(\xi)$ for $\eps|\xi|$ large. For ${\rm Re}\lambda>-\nu_0$ and $\lambda\ne \pm i\eps^2|\xi|$, we  decompose  $\lambda-\tilde{\AA}_{\eps}(\xi)$ into
		\bma
		\lambda-\tilde{\AA}_{\eps}(\xi)&=\lambda-G^1_{\eps}(\xi)-G^2_{\eps}(\xi)\nnm\\
		&=\(I-G^2_{\eps}(\xi)(\lambda-G^1_{\eps}(\xi))^{-1}\)\(\lambda-G^1_{\eps}(\xi)\), \label{B_d}
		\ema
		where $G^1_{\eps}(\xi)$ and $G^2_{\eps}(\xi)$ are defined by \eqref{A1} and \eqref{A2} respectively.
		A direct computation shows that
		\be\label{X_2}
		\left\{\bln
		&G^2_{\eps}(\xi)(\lambda-G^1_{\eps}(\xi))^{-1}= \left(\ba X^1_{\eps}(\lambda,\xi) & X^2_{\eps}(\lambda,\xi) \\  X^3_{\eps}(\lambda,\xi) & 0 \ea\right)_{7\times7},\\
		&X^1_{\eps}(\lambda,\xi)=\(K-i\eps\frac{ v\cdot\xi}{|\xi|^2} P_d\)(\lambda -D_{\eps}(\xi))^{-1},
		\\
		& X^2_{\eps}(\lambda,\xi)=\(\eps v\chi_0\cdot\omega\times,0_{1\times3}\)_{1\times6}(\lambda -\eps N_1(\xi))^{-1},
		\\
		&X^3_{\eps}(\lambda,\xi)=\left(\ba -\eps\omega\times P_m(\lambda-D_{\eps}(\xi))^{-1} \\ 0_{3\times1} \ea\right)_{6\times1}.
		\eln\right.
		\ee

		Then, we have the estimates on the right hand terms of \eqref{X_2} as follows.
		
		\begin{lem} [{\cite[Lemma 2.4]{YZ}}]\label{LP03}
			There exists a constant  $C>0$ so that the following holds
			\begin{enumerate}
				\item[\rm (1)] For any $\delta>0$, if ${\rm Re}\lambda\ge -\nu_0+\delta$, then we have
				\be
				\|K(\lambda-D_{\eps}(\xi))^{-1}\| \le  C\delta^{-\frac12}(1+\eps|\xi|)^{-\frac12}.  \nnm
				\ee
				\item[\rm (2)] For any $\delta>0,\, \tau_0>0$, if  ${\rm Re}\lambda\ge -\nu_0+\delta$ and $\eps|\xi|\le \tau_0$, then we have
				\be
				\|K(\lambda-D_{\eps}(\xi))^{-1}\| \leq C\delta^{-1}(1+\tau_0)^{\frac12}(1+|{\rm Im}\lambda|)^{-\frac12}. \nnm
				\ee
				
				\item[\rm (3)]  For any $\delta>0$,  if ${\rm Re}\lambda\ge -\nu_0+\delta$, then we have
				\bma
				\|P_{m}(\lambda-D_{\eps}(\xi))^{-1}\| &\le
				C\delta^{-\frac12}(1+\eps|\xi|)^{-\frac12} ,  \nnm\\
				\|P_{m}(\lambda-D_{\eps}(\xi))^{-1}\| &\le
				C(\delta^{-1}+1) (1+\eps|\xi|)|\lambda|^{-1}.  \nnm
				\ema
				
				\item[\rm (4)]  For any $\delta>0$,  if ${\rm Re}\lambda\ge -\nu_0+\delta$, then we have
				\bma
				\|(v\cdot\xi)|\xi|^{-2}P_{d}(\lambda-D_{\eps}(\xi))^{-1}\|
				&\leq C \delta^{-1} |\xi|^{-1},  \nnm
				\\
				\|(v\cdot\xi)|\xi|^{-2}P_{d}(\lambda-D_{\eps}(\xi))^{-1}\|
				&\leq C(\delta^{-1}+1)(|\xi|^{-1}+1)  |\lambda|^{-1}. \nnm
				\ema
			\end{enumerate}
		\end{lem}

		\begin{lem}[\cite{Li4}]\label{inver}
			Let $K,K_4$ be the operators on the space $X$ and $Y$, and $K_2,K_3$ be the operators  $Y\to X$ and $X\to Y$ respectively. Let $K$ be a matrix operator on $X\times Y$ defined by
			$$
			K=\left(\ba K & K_2 \\  K_3 & K_4 \ea\right).
			$$
			If the norms of $K,K_2,K_3$ and $K_4$ satisfy 
			$$\|K\|<1,\quad \|K_4\|<1,\quad \|K_2\|\|K_3\|<(1-\|K\|)(1-\|K_4\|),$$
			then the operator $I+K$ is invertible on $X\times Y$.
		\end{lem}
		
		By Lemmas \ref{LP03}, \ref{inver} and the similar argument as Lemma 2.4 in \cite{Li4}, we have the structure of the spectrum set
		of the operator $\tilde\AA_{\eps}(\xi)$ for $\eps|\xi|$ large.

		\begin{lem}
			\label{LP01}Fixed $\eps\in (0,1)$.
			The following statements hold.
			\begin{enumerate}
				\item[\rm (1)]  For any $\delta>0$, there
				exists $ R_1= R_1(\delta)>0$  such that for $\eps|\xi|>R_1$,
				\bq
				\sigma(\tilde{\AA}_{\eps}(\xi))\cap\{\lambda\in\mathbb{C}\,|\,\mathrm{Re}\lambda\ge-\frac{\nu_0}2\}
				\subset
				\sum_{j=\pm1}\{\lambda\in\mathbb{C}\,|\,|\lambda-\eps ji|\xi||\le\delta\}.\label{sg4}
				\eq
				\item[\rm (2)] For any $r_1>r_0>0$, there
				exists $\alpha =\alpha(r_0,r_1)>0$ such that for  $r_0\le \eps|\xi|\le r_1$,
				\bq \sigma(\tilde{\AA}_{\eps}(\xi))\subset\{\lambda\in\mathbb{C}\,|\, \mathrm{Re}\lambda<-\alpha\} .\label{sg3}\eq
			\end{enumerate}
		\end{lem}
		
		\begin{proof}
			We prove \eqref{sg4} first.
			By Lemma \ref{LP03}, \eqref{b_1(xi)} and \eqref{eee}, there exists $R_1=R_1(\delta)>0$ such that for $\mathrm{Re}\lambda\ge-\nu_0/2$, $ |\lambda-\eps ji|\xi||>\delta$ and $\eps|\xi|>R_1$,
			$$
			\|X^1_{\eps}(\lambda,\xi)\|_{\xi}\leq 1/2,\quad \|X^2_{\eps}(\lambda,\xi)\|\le \eps\delta^{-1}, \quad \| X^3_{\eps}(\lambda,\xi)\| \leq \eps\delta/4.
			$$
			This and Lemma \ref{inver} imply that the operator  $I-G^2_{\eps}(\xi)(\lambda-G^1_{\eps}(\xi))^{-1}$ is invertible on
			$L^{2}_{\xi}(\R^3_v)\times \mathbb{C}^3_{\xi}\times \mathbb{C}^3_{\xi}$ and thus  $\lambda-\tilde{\AA}_{\eps}(\xi)$ is invertible on $L^{2}_{\xi}(\R^3_v)\times \mathbb{C}^3_{\xi}\times \mathbb{C}^3_{\xi}$ and
			satisfies
			$$
			(\lambda-\tilde{\AA}_{\eps}(\xi))^{-1}
			=\(\lambda-G^1_{\eps}(\xi)\)^{-1}\( I-G^2_{\eps}(\xi)(\lambda-G^1_{\eps}(\xi))^{-1}\)^{-1}.
			$$
			Therefore, it holds that for $ \eps|\xi|> R_1$,
			$$ \rho(\tilde{\AA}_{\eps}(\xi))\supset\{\lambda\in\mathbb{C}\,|\,\min_{j=\pm1}{|\lambda-\eps ji|\xi||}>\delta,\, \mathrm{Re}\lambda\ge-\frac{\nu_0}2\},
			$$
			which leads to \eqref{sg4}.

			Next, we turn to prove \eqref{sg3}.
			By Lemma \ref{LP03}, \eqref{b_1(xi)}  and \eqref{eee}, there exists
			$y_1=y_1(r_0,r_1)>0$ large enough such that for $\mathrm{Re}\lambda\geq -\nu_0/2$, $|\mathrm{Im}\lambda|>y_1 $ and $r_0\le \eps|\xi|\le r_1$,
			$$
			\|X^1_{\eps}(\lambda,\xi)\|_{\xi}\leq 1/6,\quad \|X^2_{\eps}(\lambda,\xi)\|\le 1/6, \quad \| X^3_{\eps}(\lambda,\xi)\| \leq 1/6.
			$$
			This implies that the operator $I-G^2_{\eps}(\xi)(\lambda-G^1_{\eps}(\xi))^{-1}$
			is invertible on $L^2_{\xi}(\R^3_v)\times \mathbb{C}^3_{\xi}\times \mathbb{C}^3_{\xi}$, which together with
			\eqref{B_d} yield that  $\lambda-\tilde{\AA}_{\eps}(\xi)$ is also invertible on $L^2_{\xi}(\R^3_v)\times \mathbb{C}^3_{\xi}\times \mathbb{C}^3_{\xi}$ when
			$\mathrm{Re}\lambda\geq -\nu_0/2$,  $|\mathrm{Im}\lambda|>y_1 $ and $r_0\le \eps|\xi|\le r_1$.
			Hence, we conclude that for $r_0\le \eps|\xi|\le r_1$,
			\be
			\sigma(\tilde{\AA}_{\eps}(\xi))
			\cap\{\lambda\in\mathbb{C}\,|\,\mathrm{Re}\lambda\ge-\frac{\nu_0}2\}
			\subset
			\{\lambda\in\mathbb{C}\,|\,\mathrm{Re}\lambda\ge
			-\frac{\nu_0}2,\,|\mathrm{Im}\lambda|\le y_1 \}.   \label{SpH}
			\ee

			By \eqref{SpH}, it is sufficient to prove \eqref{sg3} for $|{\rm Im}\lambda|\le y_1$.
			If it does not
			hold, then there exists a sequence of
			$\{(\xi_n,\lambda_n,U_n)\}$ satisfying $\eps|\xi_n|\in[r_0,r_1]$, $U_n=(f_n,E_n,B_n)\in
			L^2_{\xi_n}(\R^3)\times \mathbb{C}^3_{\xi_n}\times \mathbb{C}^3_{\xi_n}$ with $\|U_n\|_{\xi_n} =1$, and $\lambda_nU_n=
			\tilde{\AA}_{\eps}(\xi_n)U_n$ with $|{\rm Im}\lambda_n|\le y_1$ and ${\rm Re}\lambda_n\to0$ as $n\to\infty$. That is,
			$$
			\left\{\bal
			\lambda_nf_n=(L-i\eps v\cdot\xi_n-i\eps\frac{v\cdot\xi_n}{|\xi_n|^2}P_{ d} )f_n-\eps v\chi_0\cdot( \omega_n\times E_n),\\
			\lambda_n E_n=-\eps\omega_n\times  (f_n,v\chi_0)+i\eps \xi_n\times B_n,\\
			\lambda_n B_n=-i\eps \xi_n\times E_n.
			\ea\right.
			$$
			Rewrite the first  equation as
			$$
			(\lambda_n+\nu+i\eps v\cdot\xi_n)f_n=Kf_n-i\eps\frac{v\cdot\xi_n}{|\xi_n|^2}P_{ d} f_n- \eps v\chi_0\cdot( \omega_n\times E_n).
			$$
			Since $K$ is a compact operator  on $L^2(\R^3)$, there exists a
			subsequence $\{f_{n_j}\}$ of $\{f_n\}$ and $g_1\in L^2(\R^3)$ such that
			$$
			Kf_{n_j}\rightarrow g_1 \quad \mbox{as}\quad j\to\infty.
			$$
			By using  the fact that  $\eps|\xi_n|\in[r_0,r_1]$ and  $P_{ d} f_n=C_0^n\sqrt{M}$ with
			$|C_0^n|^2+|E_n|^2+|B_n|^2\le 1$, there exists a subsequence of (still denoted by) $\{( \xi_{n_j},f_{n_j},E_{n_j},B_{n_j})\}$, and $( \xi_0,C_0,E_0, B_0)$ with $\eps|\xi_0|\in[r_0,r_1]$ and $|C_0|^2+|E_0|^2+|B_0|^2\leq1$
			such that $(\xi_{n_j},C^{n_j}_0,E_{n_j},B_{n_j})\to (\xi_0,C_0,E_0,B_0)$ as $j\to\infty$. In particular,
			$$
			\frac{v\cdot\xi_{n_j}}{|\xi_{n_j}|^{2}}P_{ d} f_{n_j}
			\rightarrow \frac{v\cdot\xi_0}{|\xi_0|^{2}}C_0\sqrt{M}=:g_2, \quad \frac{ \xi_{n_j}}{|\xi_{n_j}|}\times E_{n_j} \to \frac{ \xi_0}{|\xi_0|}\times E_0=:Y_0 \ \ \mbox{as} \ \ j\to\infty.
			$$
			Since $|\text{Im}\lambda_n|\leq y_1$ and ${\rm Re}\lambda_n\to 0$,
			we can extract a subsequence of (still denoted by) $\{\lambda_{n_j}\}$
			such that $\lambda_{n_j}\rightarrow \lambda_0$ with ${\rm
				Re}\lambda_0=0$. Then
			$$
			\lim_{j\rightarrow\infty}f_{n_j}
			=\lim_{j\rightarrow\infty}\frac{g_1-\eps g_2-\eps(v\cdot Y_0)\chi_0}{\lambda_{n_j}+\nu+i\eps(v\cdot\xi_{n_j})}
			=\frac{g_1-\eps g_2-\eps(v\cdot Y_0)\chi_0}{\lambda_0+\nu+i\eps(v\cdot\xi_0)}=:f_0 \ \ {\rm in} \ \ L^2.
			$$
			It follows that $\tilde{\AA}_{\eps}(\xi_0) U_0=\lambda_0 U_0$ with $U_0=(f_0,E_0,B_0)\in
			L^2_{\xi_0}(\R^3)\times \mathbb{C}^3_{\xi_0}\times \mathbb{C}^3_{\xi_0}$ and $\lambda_0$ is an eigenvalue of $\tilde{\AA}_{\eps}(\xi_0)$ with ${\rm Re}\lambda_0=0$. This
			contradicts to the fact that  ${\rm Re} \lambda <0$
			for $\eps\xi\ne 0$ obtained by Lemma~\ref{Egn}. Thus, the proof the lemma is completed.
		\end{proof}

		We now investigate the spectrum and resolvent sets of $\tilde{\AA}_{\eps}(\xi)$ for $\eps |\xi|$ small. Let $P_A$, $P_B$ be the orthogonal  projection operators defined by
		\begin{align*}
			P_A=\left(\ba P_0 &  0 & 0\\ 0 & I_{3\times 3} & 0 \\ 0 & 0 & I_{3\times 3} \ea\right),\quad P_B=\left(\ba P_1 &  0 & 0\\ 0 & 0 & 0 \\ 0 & 0 & 0 \ea\right).
		\end{align*}
		Based on macro-micro decomposition,  we can decompose  $\lambda-\tilde{\AA}_{\eps}(\xi)$ into
		\bq
		\lambda-\tilde{\AA}_{\eps}(\xi)=\lambda P_A-G^3_{\eps}(\xi)+\lambda P_B-G^4_{\eps}(\xi)+G^5_{\eps}(\xi),
		\label{Bd3}\eq
		where
		\bma\label{defG3}
		G^3_{\eps}(\xi)&=P_A\tilde{\AA}_{\eps}(\xi)P_A=\eps \left(\ba -i P_0(v\cdot\xi)P_0-i\frac{v\cdot\xi}{|\xi|^2}P_d & - v\chi_0\cdot \omega\times & 0\\ - \omega\times P_m & 0 & i \xi \times \\ 0 & -i \xi\times & 0 \ea\right),\\
		G^4_{\eps}(\xi)&=P_B\tilde{\AA}_{\eps}(\xi)P_B=\left(\ba L-i\eps P_1(v\cdot\xi)P_1 & 0 & 0\\ 0 & 0 & 0 \\ 0 & 0 & 0 \ea\right),\label{defG4}\\
		G^5_{\eps}(\xi)&= \eps \left(\ba iP_0(v\cdot\xi)P_1+iP_1(v\cdot\xi)P_0 & 0 & 0\\ 0 & 0 & 0 \\ 0 & 0 & 0 \ea\right).\nnm
		\ema
		
		\begin{lem}
			{\rm (1)}  $ G^3_{\eps}(\xi)=\eps G_3(\xi)$ is an operator in $N_0\times \C^3_{\xi}\times \C^3_{\xi}$, where $ G_3(\xi)$ admits nine eigenvalues $\{\eta_j(|\xi|),\, -1\le j\le 7\}$ given by
			\be
			\left\{\bln
			&\eta_{\pm1}(|\xi|)=\pm i \mbox{$\sqrt{1+\frac53|\xi|^2}$},\quad \eta_j(|\xi|)=0,\quad j=0,2,3, \\
			&\eta_4(|\xi|)=\eta_5(|\xi|)=- i \mbox{$\sqrt{1+ |\xi|^2}$},\quad \eta_6(|\xi|)=\eta_7(|\xi|)= i \mbox{$\sqrt{1+ |\xi|^2}$},
			\eln\right.\label{defalp}
			\ee
			
			{\rm (2)} The eigenvectors $\Lambda_j(\xi)=(h_j(\xi),X_j(\xi),Y_j(\xi))\in N_0\times \C^3_{\xi}\times \C^3_{\xi}$, $-1\le j\le 7$ are given by
			\be \label{hj}
			\left\{\bln
			&\(\Lambda_j(\xi), \Lambda_k(\xi)\)_{\xi}=(h_j,h_k)_\xi+(X_j,X_k)+(Y_j,Y_k)=\delta_{jk},\quad -1\le j,k\le 7,\\
			&\Lambda_j(\xi)=(h_j(\xi),0,0), \ \ j=-1,0,1, \\
			&h_0(\xi)=\frac{\sqrt{2}|\xi|^2}{\sqrt{3+5|\xi|^2}\sqrt{1+|\xi|^2}} \chi_0-\frac{\sqrt{3+3|\xi|^2}}{\sqrt{3+5|\xi|^2}}\chi_4,\\
			&h_{\pm1}(\xi)=\frac{\sqrt{3/2}|\xi|}{\sqrt{3+5|\xi|^2}} \chi_0\pm \sqrt{\frac12}v\cdot \frac{\xi}{|\xi|}  \chi_0+\frac{ |\xi|}{\sqrt{3+5|\xi|^2}}\chi_4;\\
			& \Lambda_k(\xi)=\frac{1}{\sqrt{1+|\xi|^2}}\(|\xi|v\cdot W^k \chi_0, 0 , -iW^k \), \ \ k=2,3, \\
			& \Lambda_l(\xi)=\frac{1}{\sqrt{2(1+|\xi|^2)}}\(v\cdot W^l \chi_0, \eta_l\omega\times W^l , i|\xi| W^l \), \ \ l=4,5,6, 7,
			\eln\right.
			\ee
			where  $W^k$, $2\le k\le 7$ are normal vectors satisfying $W^2=W^4=W^6$, $W^3=W^5=W^7$, and $W^k\cdot\omega=W^2\cdot W^3=0$
		\end{lem}
		
		\begin{proof}We consider the eigenvalue problem
			$$
			G^3_{\eps}(\xi)U=\eps z U,\quad U=(f_0,X,Y)\in N_0\times \mathbb{C}^3_{\xi}\times \mathbb{C}^3_{\xi}. \label{G_6-5}
			$$
			Since $f_0\in N_0$, we rewrite $f_0$ as
			$$f_0=\sum^4_{j=0}W_j\chi_j,\quad W_j=(f_0,\chi_j).$$
			Thus, we obtain the eigenvalue problem for $(\lambda,W_0,W,W_4,X,Y)$ with $W=(W_1,W_2,W_3)$:
			\bma
			z W_0&=-i \xi\cdot W,\label{1C_1}\\
			z W&=-i \(\xi+\frac{\xi}{|\xi|^2}\) W_0-i \sqrt{\frac23}\xi W_4- \omega\times X, \label{1C_2}\\
			z W_4&=-i \sqrt{\frac23}\xi\cdot W,\\
			z X&=i \xi\times Y- \omega\times W,\\
			z Y&=-i \xi\times X.\label{1C_4}
			\ema
			{ Let $s=|\xi|$ and $\omega=\xi/|\xi|.$} By taking $\omega\cdot$ and $\omega\times$ to \eqref{1C_2}, we simplify the system \eqref{1C_1}--\eqref{1C_4} to the following two systems:
			\be\label{1C_1a}
			\left\{\bln
			&z W_0=-i s(W\cdot\omega),\\
			&z (W\cdot\omega)=-i\(s+\frac1s\) W_0-i\sqrt{\frac23}s W_4 ,  \\
			&z W_4=-i\sqrt{\frac23}s(W\cdot\omega),
			\eln\right.
			\ee
			and
			\be\label{1C_2a}
			\left\{\bln
			&z \omega\times W=  X,  \\
			&z X=i \xi\times Y- \omega\times W, \\
			&z Y=-i \xi\times X.
			\eln\right.
			\ee
			Let $U_1=(C_0,C_1,C_2)\in \C^3$ and $U_2=(Z,X,Y)\in (\C^3_{\xi})^3$. We rewrite \eqref{1C_1a} and \eqref{1C_2a} as
			$$ \lambda U_k= D_k(\xi)U_k,\quad k=1,2, $$
			where
			\bmas
			D_1(\xi)&= \(\bal
			0 & -i s & 0\\
			-i(s+\frac1s) & 0 & -i\sqrt{\frac23}s\\
			0 &- i\sqrt{\frac23}s & 0
			\ea\)_{3\times 3},\\
			D_2(\xi)&= \(\bal
			0 & I_{3\times3} & 0\\
			-I_{3\times3} & 0 & i \xi\times\\
			0 & -i \xi\times & 0
			\ea\)_{9\times 9}.
			\emas
			Thus, the operator $D_1(\xi)$: $\C^3\to \C^3$ has three eigenvalues $\eta_j(s) $ and eigenvectors $ V_j(s)$, $j=-1,0,1$ given by
			\be\label{VV}
			\left\{\bln
			&\eta_{j}(s)=j i \mbox{$\sqrt{1+\frac53s^2}$}, \quad j=-1,0,1, \\
			&V_0(s)=\bigg(\frac{\sqrt{2}s^2}{\sqrt{3+5s^2}\sqrt{1+s^2}},0,-\frac{\sqrt{3+3s^2}}{\sqrt{3+5s^2}}\bigg),\\
			&V_{\pm1}(s)= - \bigg(\frac{\sqrt{3}s}{\sqrt{3+5s^2}},\pm  1,\frac{\sqrt{2} s}{\sqrt{3+5s^2}}\bigg);
			\eln\right.
			\ee
			the operator $D_2(\xi)$: $(\C^3_{\xi})^3\to (\C^3_{\xi})^3$ has six eigenvalues $\eta_j(s) $ and eigenvectors $ V_j(\xi)$, $j=2,3,4,5,6,7$ given by
			\be\label{VVV}
			\left\{\bln
			&\eta_2(s)=\eta_3(s)=0,\quad \eta_4(s)=\eta_5(s)=- i \mbox{$\sqrt{1+ s^2}$},\\
			& \eta_6(s)=\eta_7(s)= i \mbox{$\sqrt{1+ s^2}$},\\
			&V_k(\xi)= \(is W^k , 0 , -\omega\times W^k \), \ \ k=2,3, \\
			& V_l(\xi)=\( W^l, \eta_l W^l , is \omega\times W^l\), \ \  l=4,5,6, 7,
			\eln\right.
			\ee where $(V_2,V_3)=(V_4,V_5)=(V_6,V_7)=0$.
			
			Combining \eqref{VV} and \eqref{VVV}, we obtain that
			the operator $G_3(\xi)$ in $N_0\times \C^3_{\xi}\times  \C^3_{\xi}$  has nine eigenvalues $\eta_j(s) $ and normalized eigenvectors $ \Lambda_j(\xi)$, $-1\le j\le 7$ given by
			\be \label{lamj}
			\left\{\bln
			&\Lambda_j(\xi)=(h_j(\xi),0,0), \ \ j=-1,0,1, \\
			&h_0(\xi)=\frac{\sqrt{2}s^2}{\sqrt{3+5s^2}\sqrt{1+s^2}} \chi_0-\frac{\sqrt{3+3s^2}}{\sqrt{3+5s^2}}\chi_4,\\
			&h_{\pm1}(\xi)=\frac{\sqrt{3/2}s}{\sqrt{3+5s^2}} \chi_0\pm \sqrt{\frac12}v\cdot \omega  \chi_0+\frac{ s}{\sqrt{3+5s^2}}\chi_4;\\
			& \Lambda_k(\xi)=\frac{1}{\sqrt{1+s^2}}\(s v\cdot W^k \chi_0, 0 ,  -iW^k \), \ \ k=2,3, \\
			& \Lambda_l(\xi)=\frac{1}{\sqrt{2(1+s^2)}}\(v\cdot W^l \chi_0, \eta_l\omega\times W^l , -is W^l \), \ \ l=4,5,6, 7.
			\eln\right.
			\ee
			Rewrite the eigenvalue problem as
			$$G_3(\xi)\Lambda_j(\xi)=\eta_j(\xi)\Lambda_j(\xi), \quad -1\leq j\leq 7.$$
			Taking the inner product $(\cdot,\cdot)_{\xi}$ of the above equation with $\Lambda_k(\xi)$ and using the fact that
			$$
			(i G_3(\xi)U,V)_{\xi} =(U,i G_3(\xi) V)_{\xi},\quad U,V\in N_0\times \C^3_{\xi}\times  \C^3_{\xi},
			$$
			we have
			$$
			i(\eta_j(|\xi|)-\eta_k(|\xi|))(\Lambda_j(\xi),\Lambda_k(\xi))_{\xi}=0,\quad -1\le j, k\le 7.
			$$
			Thus,
			$$
			(\Lambda_j(\xi),\Lambda_k(\xi))_{\xi}=0,\quad -1\leq j\neq k\leq 7.
			$$
			This completes the proof of the lemma.
		\end{proof}
		
		\begin{lem}\label{LP}
			Let $\xi\neq0$, we have the following properties for the linear operators $G^3_{\eps}(\xi)$ and $G^4_{\eps}(\xi)$ defined by \eqref{defG3} and \eqref{defG4}.
			\begin{enumerate}
				\item[\rm (1)]    If $\lambda\neq\eps\eta_j(\xi)$, then  the operator $\lambda  P_A-G^3_{\eps}(\xi)$ is
				invertible on $N_0\times \mathbb{C}^3_\xi\times \mathbb{C}^3_\xi$ and satisfies
				\bgr
				\|(\lambda -G^3_{\eps}(\xi))^{-1}\|_\xi
				=\max_{-1\leq j \leq 7}\(|\lambda-\eps\eta_j(\xi)|^{-1}\),\label{S_2a}
				\\
				\| G_5(\xi) (\lambda -G^3_{\eps}(\xi))^{-1} P_A\|_\xi
				\le C|\xi|\max_{-1\leq j \leq 7}\(|\lambda-\eps\eta_j(\xi)|^{-1}\),\label{S_2}
				\egr
				where $\eps\eta_j(\xi)$, $-1\le j\le 7$, are the eigenvalues of $G^3_{\eps}(\xi)$
				defined by \eqref{defalp}.
				
				\item[\rm (2)]    If $\mathrm{Re}\lambda>-\mu $, then the operator $\lambda  P_B-G^4_{\eps}(\xi)$ is
				invertible on $N_0^\bot\times \{0\}\times \{0\}$ and satisfies
				\bgr
				\|(\lambda -G^4_{\eps}(\xi))^{-1}\|\leq(\mathrm{Re}\lambda+\mu )^{-1},  \label{S_3}
				\\
				\| G_5(\xi) (\lambda -G^4_{\eps}(\xi))^{-1} P_B\|_\xi
				\leq
				C(1+|\lambda|)^{-1}[(\mathrm{Re}\lambda+\mu )^{-1}+1](|\xi|+|\xi|^2). \label{S_5}
				\egr
			\end{enumerate}
		\end{lem}

		\begin{proof}
			Since the operator $i G^3_{\eps}(\xi)$ is self-adjoint on $N_0\times \mathbb{C}^3_{\xi}\times \mathbb{C}^3_{\xi}$, namely,
			$$
			(i G^3_{\eps}(\xi)U,V)_{\xi} =(U,i G^3_{\eps}(\xi)V)_{\xi},\quad \forall\,U,V\in N_0\times \mathbb{C}^3_{\xi}\times \mathbb{C}^3_{\xi},
			$$
			we can prove \eqref{S_2a}--\eqref{S_2}. And by the dissipative property of the operator $G^4_{\eps}(\xi)$ on $N_0^\bot\times \{0\}\times \{0\}$, namely,
			$$
			{\rm Re}([\lambda -G^4_{\eps}(\xi)]U,U)   \ge (\mathrm{Re}\lambda+\mu )\|U\|^2 ,\quad \forall\, U\in N_0^\bot\times \{0\}\times \{0\},
			$$
			we can prove \eqref{S_3}--\eqref{S_5}. This completes the proof of the lemma.
		\end{proof}

		\begin{lem}\label{spectrum2}For fixed $\eps\in (0,1)$, the following facts hold.
			\begin{enumerate}
				\item[\rm (1)]  For any  $\delta>0$, there are two constants
				$r_1=r_1(\delta),\,y_1=y_1(\delta)>0$ such that for all $|\xi|\ne 0$,
				\bq \label{rb1}
				\rho(\tilde{\AA}_{\eps}(\xi))\supset
				\left\{\bln
				&\{\lambda\in\mathbb{C}\,|\,
				\mathrm{Re}\lambda\ge-\frac{\nu_0}{2},\, |\lambda\pm\eps i|\xi||\ge \delta\}
				\cup \C_+, \quad  \eps|\xi|\ge r_1; \\
				&\{\lambda\in\mathbb{C}\,|\,
				\mathrm{Re}\lambda\ge-\frac{\mu}{2},\,|\mathrm{Im}\lambda|\geq y_1\}
				\cup\C_+, \qquad~~ \, \eps|\xi|\le r_1,
				\eln\right.
				\eq where $\C_+=\{\lambda\in\mathbb{C}\,|\,\mathrm{Re}\lambda>0\}$.
				
				\item[\rm (2)]  For any $\delta>0$, there exists $r_0=r_0(\delta)>0$ such that for $\eps |\xi|\leq r_0$,
				\bq
				\sigma(\tilde{\AA}_{\eps}(\xi))\cap\{\lambda\in\mathbb{C}\,|\,\mathrm{Re}\lambda\ge-\frac{\mu}{2}\}
				\subset
				\sum^7_{j=-1}\{\lambda\in\mathbb{C}\,|\,|\lambda-\eps\eta_j(\xi)|\le\delta\}.   \label{sg4a}
				\eq
			\end{enumerate}
		\end{lem}

		\begin{proof}
			By Lemma \ref{LP01}, there exists $r_1=r_1(\delta)>0$ so that the first part of \eqref{rb1} holds. Thus we only need to prove the second part
			of \eqref{rb1}.
			By Lemma \ref{LP}, we have for $\rm{Re}\lambda>-\mu $ and
			$\lambda\neq \eps\eta_j(\xi)$  that the operator
			$ \lambda P_{A}-G^3_{\eps}(\xi)+\lambda P_B-G^4_{\eps}(\xi)$ is invertible on
			$L^2_{\xi}(\R^3_v)$ and satisfies
			$$
			[\lambda P_{A}-G^3_{\eps}(\xi)+\lambda P_B-G^4_{\eps}(\xi)]^{-1} =(\lambda -G^3_{\eps}(\xi))^{-1}P_A+(\lambda -G^4_{\eps}(\xi))^{-1}P_B,
			$$
			because the operator $\lambda P_{A}-G^3_{\eps}(\xi)$ is orthogonal to $\lambda P_B-G^4_{\eps}(\xi)$.
			Therefore, we can rewrite \eqref{Bd3} as
			$$
			\begin{aligned}
				\lambda-\tilde{\AA}_{\eps}(\xi)&=\[I+Y_1(\xi,\eps)\][\lambda P_{A}-G^3_{\eps}(\xi)+\lambda P_B-G^4_{\eps}(\xi)],\\
				Y_1(\xi,\eps)&=G^5_{\eps}(\xi)(\lambda -G^3_{\eps}(\xi))^{-1}P_A+G^5_{\eps}(\xi)(\lambda -G^4_{\eps}(\xi))^{-1}P_B.
			\end{aligned}
			$$
			
			For  $\eps |\xi| \leq r_1$, by \eqref{b_1(xi)}, \eqref{S_2} and \eqref{S_5} we can choose $y_1=y_1(r_1)>0$ such that it holds
			for $\mathrm{Re}\lambda\ge-\mu/2$ and $|\mathrm{Im}\lambda|\geq y_1$ that
			\be
			\|G^5_{\eps}(\xi)(\lambda -G^3_{\eps}(\xi))^{-1}P_A\|_{\xi}  \leq 1/4, \quad \|G^5_{\eps}(\xi)(\lambda -G^4_{\eps}(\xi))^{-1}P_B\| \leq 1/4. \label{bound_1}
			\ee
			This implies that the operator $I+Y_1(\xi,\eps)$ is invertible on
			$L^{2}_{\xi}(\R^3_v)\times \mathbb{C}^3_{\xi}\times \mathbb{C}^3_{\xi}$ and thus  $\lambda-\tilde{\AA}_{\eps}(\xi)$ is invertible on $L^{2}_{\xi}(\R^3_v)\times \mathbb{C}^3_{\xi}\times \mathbb{C}^3_{\xi}$ and satisfies
			$$
			(\lambda-\tilde{\AA}_{\eps}(\xi))^{-1}
			=[(\lambda -G^3_{\eps}(\xi))^{-1}P_A+(\lambda -G^4_{\eps}(\xi))^{-1}P_B]\[I+Y_1(\xi,\eps)\]^{-1}.
			$$
			Therefore, it holds for $\eps |\xi|\leq r_1$ that
			$$\rho(\tilde{\AA}_{\eps}(\xi))\supset \{\lambda\in\mathbb{C}\,|\,{\rm
				Re}\lambda\ge-\mu/2, |{\rm Im}\lambda|\ge y_1\}.$$

			Assume that $\min_{-1\le j\le 7} |\lambda-\eps\eta_j(\xi)|>\delta$ and $\mathrm{Re}\lambda\ge-\mu/2$. Then, by \eqref{S_2} and \eqref{S_5} we can choose $r_0=r_0(\delta)>0$ small enough so that \eqref{bound_1} still holds for $\eps |\xi|\leq r_0$,
			which implies that the operator
			$\lambda-\tilde{\mathbb{A}}_{\eps}(\xi)$ is invertible on $L^{2}_{\xi}(\R^3)\times \mathbb{C}^3_{\xi}\times \mathbb{C}^3_{\xi}$.
			Therefore, it holds for $\eps |\xi|\leq r_0$ that
			$$\rho(\tilde{\AA}_{\eps}(\xi))\supset\{\lambda\in\mathbb{C}\,|\, \min_{-1\le j\le 7} |\lambda-\eps\eta_j(\xi)|>\delta,\mathrm{Re}\lambda\ge-\mu/2\},$$
			which gives \eqref{sg4a}. This completes the proof of the lemma.
		\end{proof}

		\subsection{Eigenvalues in $\eps |\xi|\le r_0$}
		
		Now we prove the existence and establish  the asymptotic expansions of the eigenvalues of $\tilde{\AA}_{\eps}(\xi)$ for $\eps |\xi|$ small.
		In terms of \eqref{Axi}, the eigenvalue problem $\tilde{\AA}_{\eps}(\xi)U=\lambda  U$ for $U=(f,X,Y)\in L^{2}_{\xi}(\R^3_v)\times \mathbb{C}^3_{\xi}\times \mathbb{C}^3_{\xi}$
		can be written as
		\bma
		\lambda f  &=\(L-i\eps v\cdot\xi  -i\eps\frac{ v\cdot\xi}{|\xi|^2}P_{d}\)f-\eps v\chi_0\cdot(\omega\times X),\label{L_2}\\
		\lambda X&=-\eps\omega\times (f,v\chi_0)+i\eps\xi\times Y,\label{L_2a}\\
		\lambda Y&=-i\eps\xi\times X,\quad |\xi|\ne0. \nnm
		\ema
		We rewrite $f$ in the
		form $f=f_0+f_1$, where $f_0=P_0f $ and $f_1=(I-P_0)f=P_1f$.
		Then  \eqref{L_2} gives
		\bma
		&\lambda f_0=- i\eps P_0(v\cdot\xi)(f_0+f_1)-i\eps\frac{v\cdot\xi}{|\xi|^2}P_df_0-\eps v\chi_0\cdot(\omega\times X),\label{A_2}
		\\
		&\lambda f_1=Lf_1- i\eps P_1(v\cdot\xi)(f_0+f_1).\label{A_3}
		\ema
		By Lemma \ref{LP} and \eqref{A_3}, the microscopic part $f_1$ can be represented  by
		\bq
		f_1=i\eps R(\lambda,\eps \xi)P_1(v\cdot\xi) f_0 ,  \quad   \text{Re}\lambda>-\mu, \label{A_4}
		\eq
		where
		$$
		R(\lambda,\xi)=(L-\lambda -i P_1(v\cdot\xi))^{-1}.
		$$
		Substituting \eqref{A_4} into \eqref{L_2}, we obtain
		\bma
		\label{1A_5}
		\lambda f_0 &=-i\eps P_0(v\cdot\xi)f_0 -i\eps\frac{ v\cdot\xi}{|\xi|^2}P_d f_0 -\eps v\chi_0\cdot(\omega\times X)\nnm\\
		&\quad +\eps  P_0[(v\cdot\xi)R(\lambda,\eps\xi) P_1(v\cdot\xi)f_0].
		\ema
		To solve  the eigenvalue problem \eqref{1A_5}, we write $f_0\in N_0$  as
		$ f_0=\sum_{j=0}^4W_j\chi_j$. Substituting \eqref{A_4} into \eqref{A_2} and \eqref{L_2a}, we obtain the eigenvalue problem  for  $(\lambda=\eps z,W_0,W,W_4,X,Y)$ with $W=(W_1,W_2,W_3)$ as
		\bma
		z W_0&=-i(W\cdot\xi),  \label{1A_6}
		\\
		z W_j
		& =-i W_0\(\xi_j+\frac{\xi_j}{|\xi|^2}\)
		-i\sqrt{\frac23}W_4\xi_j -(\omega\times X)_j\nnm\\
		&\quad+\eps\sum^4_{k=1}W_k(R(\eps z,\eps\xi) P_1(v\cdot\xi)\chi_k,(v\cdot\xi)\chi_j),\quad j=1,2,3,\label{1A_7}
		\\
		z W_4
		&=-i\sqrt{\frac23}(W\cdot\xi)
		+\eps\sum^4_{k=1}W_k(R(\eps z,\eps\xi) P_1(v\cdot\xi)\chi_k,(v\cdot\xi)\chi_4),  \label{1A_8}\\
		z X&=-\omega\times W+i\xi\times Y, \label{A_7}
		\\
		z Y&= -i \xi\times X. \label{A_8}
		\ema
		
		We apply the following transform to simplify the system \eqref{1A_6}-\eqref{1A_8}.
		\begin{lem}[\cite{Li2}]\label{eigen_0}
			Let  $e_1=(1,0,0)$, $\xi=s\omega$ with $s\in \R$, $\omega=(\omega_1,\omega_2,\omega_3)\in \S^2$. Then, it holds for $1\le i,j\le 3$ and ${\rm Re}\lambda>-\mu$ that
			\bma
			(R(\lambda,\xi) P_1(v\cdot\xi)\chi_j,(v\cdot\xi)\chi_i)
			=&s^2(\delta_{ij}-\omega_i\omega_j)(R(\lambda,se_1) P_1(v_1\chi_2),v_1\chi_2)
			\nnm\\
			& +s^2\omega_i\omega_j(R(\lambda,se_1) P_1(v_1\chi_1),v_1\chi_1),\label{1T_1}
			\\
			(R(\lambda,\xi) P_1(v\cdot\xi)\chi_4,(v\cdot\xi)\chi_i)
			=&s^2\omega_i(R(\lambda,se_1) P_1(v_1\chi_4),v_1\chi_1),\label{1T_2}
			\\
			(R(\lambda,\xi) P_1(v\cdot\xi)\chi_i,(v\cdot\xi)\chi_4)
			=&s^2\omega_i(R(\lambda,se_1) P_1(v_1\chi_1),v_1\chi_4),\label{1T_3}
			\\
			(R(\lambda,\xi) P_1(v\cdot\xi)\chi_4,(v\cdot\xi)\chi_4)
			=&s^2(R(\lambda,se_1) P_1(v_1\chi_4),v_1\chi_4).\label{1T_4}
			\ema
		\end{lem}
		
		Substituting  \eqref{1T_1}--\eqref{1T_4} into  \eqref{1A_6}--\eqref{1A_8}, we obtain
		\bma
		z W_0&=-i s(W\cdot\omega),\label{1A_9}
		\\
		z W_j
		&=-i W_0\(s+\frac1s\)\omega_j
		-i s\sqrt{\frac23}W_4\omega_j
		+\eps s^2(W\cdot\omega)\omega_j R_{11}
		\nnm\\
		&\quad+\eps s^2(W_j-(W\cdot\omega)\omega_j)R_{22}
		+\eps s^2W_4\omega_i R_{41}
		-(\omega\times X)_j,
		\quad j=1,2,3,    \label{1A_10}
		\\
		z W_4
		&=-i s\sqrt{\frac23}(W\cdot\omega)
		+\eps s^2(W\cdot\omega)R_{14}
		+\eps s^2W_4 R_{44},
		\label{1A_11}
		\ema
		where
		\bq R_{ij}=R_{ij}(\eps z,\eps s)=:(R (\eps z,\eps se_1) P_1(v_1\chi_i),v_1\chi_j).\label{rij}\eq

		By taking $\omega\cdot$ and $\omega\times$ to \eqref{1A_10}, we simplify the system \eqref{1A_6}--\eqref{A_8} to the following two systems:
		\be
		\left\{\bln
		&z W_0=-i s(W\cdot\omega),
		\\
		&z (W\cdot\omega) =-i W_0\(s+\frac1
		s\)-i s\sqrt{\frac23}W_4+\eps s^2(W\cdot\omega)R_{11}+\eps s^2 W_4R_{41}.\label{1A_12}
		\\
		&z W_4
		=-i s\sqrt{\frac23}(W\cdot\omega)
		+\eps s^2(W\cdot\omega)R_{14}  +\eps s^2W_4 R_{44},
		\eln\right.
		\ee
		and
		\be
		\left\{\bln
		&(z-\eps s^2R_{22})(\omega\times W)=X,  \label{1A_12a}\\
		&z X=-(\omega\times W)+i\xi\times Y,
		\\
		& z Y= -i \xi\times X.
		\eln\right.
		\ee
		
		Denote by $U=(W_0,W\cdot\omega,W_4)$ a vector 
		in $\R^3$. The system \eqref{1A_12} can be written as $\mathbb{M}U=0$ with the matrix $ \mathbb{M}$ defined by
		\bq
		\mathbb{M}=\left(\ba
		z & i s & 0\\
		i(s+\frac1s)  &z-\eps s^2R_{11} &i s\sqrt{\frac23}-\eps s^2R_{41}
		\\    0  &i s\sqrt{\frac23}-\eps s^2R_{14} &z-\eps s^2 R_{44}
		\ea\right).\label{BM}
		\eq
		The equation  $\mathbb{M}U=0$ admits a non-trivial solution $U\neq 0$ for $\text{Re}\lambda>-\mu$ if and only if it holds $\rm{det}(\mathbb{M})=0$ for $\text{Re}\lambda>-\mu$.
		
		Then we multiply  $\eqref{1A_12a}_2$ by $z $ and using $\eqref{1A_12a}_1$, \eqref{rotat} to have
		\be
		(z^2+1+s^2)X=\eps s^2R_{22}(\omega\times W).  \label{1A_12b}
		\ee
		Multiplying \eqref{1A_12b} by $ z-\eps s^2R_{22} $ and using $\eqref{1A_12a}_1$, we obtain
		$$ (z^3-\eps s^2R_{22}z^2+(1+s^2)z-\eps s^4R_{22})X=0. $$
		Denote
		\bq
		D_0(z,s,\eps)=\det(\mathbb{M}), \quad  D_1(z,s,\eps)= z^3-\eps s^2R_{22}z^2+(1+s^2)z-\eps s^4R_{22}.   \label{D0a}
		\eq
		The eigenvalues $\lambda=\eps z$ can be solved by $D_0(z,s,\eps)=0$ and $D_1(z,s,\eps)=0$. The following two lemmas are about the solutions to the equations
		$D_0(z,s,\eps)=0$ and $D_1(z,s,\eps)=0$.

		\begin{lem}\label{eigen_2}There are two small constants $r_0,r_1>0$ such that the equation
			$D_0(z,s,\eps)=0$ has exactly three solutions
			$z_j=z_j(s,\eps)$, $j=-1,0,1$ for $\eps| s|\le r_0$ and $|z_j-\eta_j(s)|\le r_1 |s|$. They are $C^\infty$ functions of $s$ and $\eps$, which satisfy
			\be
			z_j(s,0)=\eta_j(s),\quad \pt_{\eps}z_j(s,0)=-b_j(s), \label{z2a}
			\ee
			where 
			\be \label{z4a}
			\left\{\bln
			\eta_j(s)&=ji\sqrt{1+\frac53s^2}, \quad j=-1,0,1, \\
			b_{\pm1}(s)&=-\frac12s^2(L^{-1}P_1(v_1\chi_1),v_1\chi_1)-\frac{s^4}{3+5s^2}(L^{-1}P_1(v_1\chi_4),v_1\chi_4),\\
			b_0(s)&=-\frac{3(s^2+s^4)}{3+5s^2}(L^{-1}P_1(v_1\chi_4),v_1\chi_4).
			\eln\right.
			\ee
			In particular, $z_j(s,\eps)$, $j=-1,0,1$ satisfy the following expansions
			\be
			z_j(s,\eps)=\eta_j(s)-\eps b_j(s)+O(1)(\eps^2 s^3).\label{z4b}
			\ee
		\end{lem}
		\begin{proof}
			By \eqref{BM} and  \eqref{D0a},
			\begin{align*}
				D_0(z,s,\eps)=&\,z^3-\eps z^2s^2(R_{11}+R_{44})-\eps(s^2+s^4)R_{44}\\
				&+z\bigg[1+\frac{5}{3}s^2+i\eps\sqrt{\frac{2}{3}}s^2(R_{14}+R_{41})+\eps^2s^4(R_{44}R_{11}-R_{14}R_{41})\bigg].
			\end{align*}
			It is easy to check that the equation $	D_0(z,s,0)=0$ has three solutions $\eta_j(s)=ji\sqrt{1+\frac53s^2}$ for $j=-1,0,1$. Moreover, it holds  for $\eps|z|,\eps|s|\ll 1$ that
				\bma
				\partial_zD_0(z,s,\eps)=&\,3z^2+1+\frac{5}{3}s^2-\eps zs^2(2R_{11}+2R_{44}+z\pt_{z}R_{11}+z\pt_{z}R_{44}) \nnm\\
				&-\eps(s^2+s^4)\pt_{z}R_{44}+i\eps \sqrt{\frac{2}{3}}s^2 (R_{14}+R_{41}+z\pt_{z}R_{14}+z\pt_{z}R_{41}) \nnm\\
				&+\eps^2s^4(R_{44}R_{11}-R_{14}R_{41}+z\pt_z (R_{44}R_{11})-z\pt_z (R_{14}R_{41})) \nnm\\
				=&\,3z^2+1+\frac{5}{3}s^2+O(\eps s^2(|z|+|s|)),\label{dD0}
				\\
				\partial_\eps D_0(z,s,\eps)=&\,-z^2s^2(R_{11}+R_{44}+\eps\pt_{\eps}R_{11}+\eps\pt_{\eps}R_{44})-(s^2+s^4)(R_{44}+\eps\pt_{\eps}R_{44}) \nnm\\
				&+i \sqrt{\frac{2}{3}}zs^2(R_{14}+R_{41}+\eps\pt_{\eps}R_{14}+\eps\pt_{\eps}R_{41}) \nnm\\
				&+\eps zs^4(2R_{44}R_{11}-2R_{14}R_{41}+\eps\pt_{\eps}(R_{44}R_{11})-\eps\pt_{\eps}(R_{14}R_{41})) \nnm\\
				=&\,-z^2s^2(R_{11}+R_{44}+O(\eps|z|+\eps| s|))-(s^2+s^4)(R_{44}+O(\eps|z|+\eps |s|)) \nnm\\
				&+i\sqrt{\frac{2}{3}}s^2z(R_{14}+R_{41})+ O(\eps s^2|z|(|z|+|s|)),\label{dD1}
				\ema
				where we had used
				\bmas
				\pt_zR_{ij}(\eps z,\eps s)&=\eps(R^2(\eps z,\eps se_1)P_1(v_1\chi_i),v_i\chi_j)=O(\eps),\\
				\pt_\eps R_{ij}(\eps z,\eps s)&=z(R^2(\eps z,\eps se_1)P_1(v_1\chi_i),v_i\chi_j)\\
				&\quad+s(R(\eps z,\eps se_1)v_1R(\eps z,\eps se_1) P_1(v_1\chi_i),v_i\chi_j)=O(|z|+|s|).
				\emas
			For $j=-1,0,1$, define
			$$
			M_j(z,s,\eps)=z-\(3\eta_j^2+1+\frac{5}{3}s^2\)^{-1}D_0(z,s,\eps).
			$$
			It follows from \eqref{dD0} and \eqref{dD1} that
			\bmas
			|\partial_z M_j(z,s,\eps)|&=\bigg|1-\(3\eta_j(s)^2+1+\frac{5}{3} s^2\)^{-1}\partial_{z}D_0(z,s,\eps)\bigg|\le C_0r_1,\\
			|\partial_{\eps} M_j(z,s,\eps)|&=\bigg|\(3\eta_j(s)^2+1+ \frac{5}{3}s^2\)^{-1}\partial_{\eps}D_0(z,s,\eps)\bigg|\le C_0s^2,
			\emas
			for $|z-\eta_j(s)| \le r_1|s|$ and $\eps|s|\le r_0$ with $r_0,r_1>0$ sufficiently small and $C_0,C_1>0$ independent of $s,\eps$. Then it is clear that a solution to $D_0(z,s,\eps)=0$ for any  fixed $(s,\eps)$ is a fixed point of the map $z\to M_j(z,s,\eps)$.
			For any $|z-\eta_j(s)|\leq r_1|s|$, $\eps s\leq r_0\ll 1$,
			\begin{align*}
				|M_j(z,s,\eps)-\eta_j(s)|&=|M_j(z,s,\eps)-M_j(\eta_j(s),s,0)|\\
				&=|M_j(z,s,\eps)-M_j(\eta_j(s),s,\eps)|+|M_j(\eta_j(s),s,\eps)-M_j(\eta_j(s),s,0)|\\
				&\leq  |\partial_z M_j(\widetilde{z},s,\eps)||z-\eta_j(s)|+|\partial_{\eps} M_j(\eta_j(s),s,\widetilde{\eps})||\eps|\le r_1|s|,\\
				|M_j(z_1,s,\eps)-M_j(z_2&,s,\eps)|\le |\partial_z M_j(\bar{z},s,\eps)||z_1-z_2|\le \frac12|z_1-z_2|.
			\end{align*}
			Hence, by the contraction mapping theorem, there exist exactly three functions $z_j(s,\eps)$, $j=-1,0,1$ for $\eps|s|\le r_0$ and $|z_j-\eta_j(s)|\le r_1 |s|$ such that $M_j(z_j(s,\eps),s, \eps)=z_j(s,\eps)$ and $z_j(s,0)=\gamma_j(s)$. This is equivalent to that $D_0(z_j(s,\eps),s,\eps)=0$.
			Morover, we have
			\begin{align*}
				&	\partial_\eps z_0(s,0)=-\frac{\partial_\eps D_0(0,s,0)}{\partial_z D_0(0,s,0)}=\frac{3(s^2+s^4)}{3+5s^2}R_{44}(0,0),\\
				&	 \partial_{\eps} z_{\pm1}(s,0)=-\frac{\partial_{\eps} D_0(\eta_{\pm1}(s),s,0)}{\partial_z D_0(\eta_{\pm1}(s),s,0)}=\frac12s^2R_{11}(0,0)+\frac{s^4}{3+5s^2}R_{44}(0,0).
			\end{align*}
			This implies \eqref{z4a}. Finally, we deal with \eqref{z4b}. Note that
			$$
			z_j(s,\eps)-\eta_j(s)=[M_j(z_j(s,\eps),s,\eps)-M_j(\eta_j(s),s,\eps)]+[M_j(\eta_j(s),s,\eps)-M_j(\eta_j(s),s,0)],
			$$
			we obtain
			\bmas
			&|z_j(s,\eps)-\eta_j(s)| \le |\pt_zM_j(\tilde{z},s,\eps)||z_j-\eta_j|+|\pt_{\eps}M_j(\eta_j ,s,\tilde{\eps})| \eps\\
			\Longrightarrow &|z_j(s,\eps)-\eta_j(s)| \le (1-|\pt_zM_j(\tilde{z},s,\eps)|)^{-1}|\pt_{\eps}M_j(\eta_j ,s,\tilde{\eps})|\eps\le C\eps s^2.
			\emas
			By \eqref{dD0} and \eqref{dD1}, we obtain for $|z-\eta_j(s)|\le C\eps s^2$ and $\eps s\le r_0$,
			\begin{align*}
				&	|\partial_z M_j(z,s,\eps)|=\bigg|1-\(3\eta_j^2+1+\frac{5}{3}s^2\)^{-1}\partial_zD_0(z,s,\eps)\bigg|\leq C \eps s,\\
				&	|\partial_\eps^2 M_j(z,s,\eps)|=\bigg|\(3\eta_j^2+1+\frac{5}{3}s^2\)^{-1}\partial_\eps^2D_0(z,s,\eps)\bigg|\leq C s^3,
			\end{align*}
			where we  had used
			$$
			\partial^2_{\eps}D_0(z,s,\eps) = O(1)(|z|^2s^2+s^2+s^4)(|z|+|s|)+ O(1)|z|s^4.
			$$
			Thus,
			\begin{align*}
				&\,z_j(s,\eps)-\eta_j(s)+  \eps b_j(s)\\
				=&\,[M_j(z_j(s,\eps),s,\eps)-M_j(\eta_j(s),s,\eps)]\\
				&+[M_j(z_j(s,\eps),s,\eps)-M_j(\eta_j(s),s,0)-\pt_{\eps}M_j(\eta_j(s),s,0) \eps]\\
				\le&\, |\pt_zM_j(\bar{z},s,\eps)||z_j-\eta_j|+|\pt^2_{\eps}M_j(\eta_j(s),s,\bar{\eps})|\eps^2\\
				\le&\, C\eps^2 s^3,
			\end{align*}
			which leads to \eqref{z4bb}. The proof of the lemma is then completed.
		\end{proof}
		
		\begin{lem}
			\label{eigen_1}
			There are two small constants $r_0,r_1>0$ such that the equation
			$D_1(z,s,\eps)=0$ has exactly three solutions
			$z_j=z_j(s,\eps)$, $j=-1,0,1$ for $\eps| s|\le r_0$ and $|z_j-\gamma_j(s)|\le r_1 |s|$. They are $C^\infty$ functions of $s$ and $\eps$, which satisfy
			\be
			z_j(s,0)=\gamma_j(s),\quad \pt_{\eps}z_j(s,0)=-a_j(s), \label{z2b}
			\ee
			where 
			\be \label{z4aa}
			\left\{\bln
			\gamma_j(s)&=ji\sqrt{1+ s^2},\quad j=-1,0,1,\\
			a_{\pm1}(s)&=- \frac{s^2}{2(1+s^2)}(L^{-1}P_1(v_1\chi_2),v_1\chi_2),\\
			a_0(s)&=-\frac{s^4}{1+s^2}(L^{-1}P_1(v_1\chi_2),v_1\chi_2).
			\eln\right.
			\ee
			In particular, $z_j(s,\eps)$, $j=-1,0,1$ satisfy the following expansions
			\be
			z_{j}(s,\eps)=\gamma_{j}(s)-\eps a_{j}(s)+ O\(\frac{\eps^2 s^3(j^2+s^2)}{1+s^2}\).\label{z4bb}
			\ee
		\end{lem}
		
		\begin{proof}
			By \eqref{D0a}, we have
			\bq
			D_1(z,s,0)=z(z^2+1+ s^2)=0 \label{aaa}\eq
			has three solutions $\gamma_j(s)=ji\sqrt{1+ s^2}$ for  $j=-1,0,1$.
			$D_1(z,s,\eps)$ is $C^\infty$ with respect to $(z,s,\eps)$ and satisfies for $\eps|z|,\eps|s|\ll 1$ that
			\bma
			\partial_{\eps}D_1(z,s,\eps)&=-z^2s^2 R_{22}
			-z^2\eps s^2 \partial_{\eps}R_{22}-s^4 R_{22}-\eps s^4\partial_{\eps}R_{22} \nnm\\
			&=-(z^2s^2+s^4) [R_{22}(0,0)+O(\eps|z|+\eps|s|)],\label{p_0}
			\\
			\partial_{z}D_1(z,s,\eps)&=3z^2 -2z\eps s^2 R_{22} +1+s^2-z^2\eps s^2 \partial_z R_{22}-\eps s^4 \partial_z R_{22} \nnm\\
			&=3z^2  +1+s^2+O(\eps s^3).\label{p_1a}
			\ema
			For $j=-1,0,1$, we define
			$$G_j(z,s,\eps)=z-(3\gamma_j(s)^2+1+ s^2)^{-1}D_1(z,s,\eps).$$
			It is straightforward to verify that a solution of $D_1(z,s,\eps)=0$ for any fixed $s$ and $\eps$ is a fixed point of $G_j(z,s,\eps)$.

			It follows from \eqref{p_0} and \eqref{p_1a} that
			\bmas |\partial_z G_j(z,s,\eps)|&=\left|1-(3\gamma_j(s)^2+1+ s^2)^{-1}\partial_{z}D_1(z,s,\eps)\right|\le C_0r_1,\\
			|\partial_{\eps} G_j(z,s,\eps)|&=\left|(3\gamma_j(s)^2+1+ s^2)^{-1}\partial_{\eps}D_1(z,s,\eps)\right|\le C_0\frac{j^2s^2+s^4}{1+s^2},
			\emas
			for $|z-\gamma_j(s)| \le r_1|s|$ and $\eps|s|\le r_0$ with $r_0,r_1>0$ sufficiently small and $C_0,C_1>0$ independent of $s,\eps$.
			This implies that for $|z-\gamma_j(s)|\le r_1 |s|$ and $\eps|s|\le r_0$ with $r_0\ll1$,
			\bmas |G_j(z,s,\eps)-\gamma_j(s)|&=|G_j(z,s,\eps)-G_j(\gamma_j(s),s,0)|\\
			&\le |G_j(z,s,\eps)-G_j(\gamma_j(s),s,\eps)|+|G_j(\gamma_j(s),s,\eps)-G_j(\gamma_j(s),s,0)|\\
			&\le |\partial_z G_j(\widetilde{z},s,\eps)||z-\gamma_j(s)|+|\partial_{\eps} G_j(\gamma_j(s),s,\widetilde{\eps})||\eps|\le r_1|s|,\\
			|G_j(z_1,s,\eps)-G_j(z_2&,s,\eps)|\le |\partial_z G_j(\bar{z},s,\eps)||z_1-z_2|\le \frac12|z_1-z_2|,
			\emas
			where $\widetilde{\eps}$ is between $0$ and $\eps$, $\widetilde{z}$ is  between $z$ and $\gamma_j(s)$, and $\bar{z}$ is  between $z_1$ and $z_2$.

			Hence by the contraction mapping theorem, there exist exactly three functions $z_j(s,\eps)$, $j=-1,0,1$ for $\eps|s|\le r_0$ and $|z_j-\gamma_j(s)|\le r_1 |s|$ such that $G_j(z_j(s,\eps),s, \eps)=z_j(s,\eps)$ and $z_j(s,0)=\gamma_j(s)$. This is equivalent to that $D_1(z_j(s,\eps),s,\eps)=0$.
			Moreover, by \eqref{p_0}--\eqref{p_1a} we have
			\bma
			\partial_{\eps} z_{0}(s,0)&=-\frac{\partial_{\eps} D_1(0,s,0)}{\partial_z D_1(0,s,0)}=\frac{s^4}{1+s^2}(L^{-1}P_1(v_1\chi_2),v_1\chi_2), \label{bbb}\\
			\partial_{\eps} z_{\pm1}(s,0)&=-\frac{\partial_{\eps} D_1(\beta_{\pm1}(s),s,0)}{\partial_z D_1(\beta_{\pm1}(s),s,0)}=\frac{s^2}{2(1+s^2)}(L^{-1}P_1(v_1\chi_2),v_1\chi_2) .\label{z3}
			\ema
			Combining \eqref{aaa}, \eqref{bbb} and \eqref{z3}, we obtain \eqref{z2b} and \eqref{z4aa}.
			
			Finally, we deal with \eqref{z4bb}.
			Note that
			$$
			z_j(s,\eps)-\gamma_j(s)=[G_j(z_j(s,\eps),s,\eps)-G_j(\gamma_j(s),s,\eps)]+[G_j(\gamma_j(s),s,\eps)-G_j(\gamma_j(s),s,0)],
			$$
			we obtain
			\bmas
			&|z_j(s,\eps)-\gamma_j(s)| \le |\pt_zG_j(\tilde{z},s,\eps)||z_j-\gamma_j|+|\pt_{\eps}G_j(\gamma_j ,s,\tilde{\eps})| \eps\\
			\Longrightarrow &|z_j(s,\eps)-\gamma_j(s)| \le (1-|\pt_zG_j(\tilde{z},s,\eps)|)^{-1}|\pt_{\eps}G_j(\gamma_j ,s,\tilde{\eps})|\eps\le C\eps s^2.
			\emas
			Moreover, by \eqref{p_0} and \eqref{p_1a} we obtain that for $|z-\gamma_j(s)|\le C\eps |s|^2$ and $\eps|s|\le r_0$,
			\bmas |\partial_z G_j(z,s,\eps)|&=\left|1-(3\gamma_j(s)^2+1+ s^2)^{-1}\partial_{z}D_1(z,s,\eps)\right|\le C\eps s,\\
			|\partial^2_{\eps} G_j(z,s,\eps)|&=\left|(3\gamma_j(s)^2+1+ s^2)^{-1}\partial^2_{\eps}D_1(z,s,\eps)\right|\le C\frac{s(j^2s^2+s^4)}{1+s^2},
			\emas
			where we  had used
			$$
			\partial^2_{\eps}D_1(z,s,\eps)= -(z^2s^2+s^4)(2\partial_{\eps}R_{22}+\eps \partial^2_{\eps}R_{22}).
			$$
			Thus,
			\bmas
			&\,z_j(s,\eps)-\gamma_j(s)+ \eps b_j(s)\\
			=&\,[G_j(z_j(s,\eps),s,\eps)-G_j(\gamma_j(s),s,\eps)]\\
			&+[G_j(\gamma_j(s,\eps),s,\eps)-G_j(\gamma_j(s),s,0)-\pt_{\eps}G_j(\gamma_j(s),s,0) \eps]\\
			\le&\, |\pt_zG_j(\bar{z},s,\eps)||z_j-\gamma_j|+|\pt^2_{\eps}G_j(\gamma_j(s),s,\bar{\eps})|\eps^2\\
			\le&\, C\frac{\eps^2 s^3}{1+s^2}(j^2+s^2),
			\emas
			which leads to \eqref{z4bb}. The proof of the lemma is then completed.
		\end{proof}

		With the help of Lemmas \ref{eigen_2} and \ref{eigen_1}, we have the eigenvalue $\lambda_j(|\xi|,\eps)$ and the corresponding eigenfunction $\mathcal{U}_j(\xi,\eps)$
		of the operator $\tilde{\AA}_{\eps}(\xi)$  for $\eps |\xi|\le r_0$ as follows.
		
		\begin{lem}\label{eigen_4a}
			{\rm (1)} There exists a small constant $r_0>0$ such that $ \sigma(\tilde{\AA}_{\eps}(\xi))\cap \{\lambda\in \mathbb{C}\,|\, \mathrm{Re}\lambda>-\mu /2\} $ consists of nine points $\{\lambda_j(s,\eps),\, -1\le j\le 7\}$ for   $\eps |s| \le  r_0$ and $s=|\xi|$. The eigenvalues $\lambda_j(s,\eps)$  are $C^\infty$ functions of $s$ and $\eps $, and admit the following asymptotic expansions for $\eps |s| \le r_0 $:
			{ \bma\label{specr0}
				\lambda_{j}(s,\eps) &=\eps \eta_j(s)- \eps^2b_{j}(s) +O(\eps^3s^3),\quad j=-1,0,1,4,5,6,7,\\
				\lambda_{k}(s,\eps) &=\eps \eta_k(s)- \eps^2b_{k}(s) +O\(\frac{\eps^3s^5}{1+s^2}\), \quad k=2,3,\label{specr01}
				\ema
				where}
			\be \label{bj}
			\left\{\bln
			\eta_{\pm1}(s)&=\pm i \sqrt{1+\frac53s^2} ,\quad \eta_j(s)=0,\quad j=0,2,3, \\
			\eta_4(s)&=\eta_5(s)=- i \mbox{$\sqrt{1+ s^2}$},\quad \eta_6(s)=\eta_7(s)= i \mbox{$\sqrt{1+ s^2}$},\\
			b_0(s)&=-\frac{3(s^2+s^4)}{3+5s^2}(L^{-1}P_1(v_1\chi_4),v_1\chi_4),\\
			b_{\pm1}(s)&=-\frac12s^2(L^{-1}P_1(v_1\chi_1),v_1\chi_1)-\frac{s^4}{3+5s^2}(L^{-1}P_1(v_1\chi_4),v_1\chi_4),\\
			b_2(s)&=b_3(s)=-\frac{s^4}{1+s^2}(L^{-1}P_1(v_1\chi_2),v_1\chi_2),\\
			b_k(s)& =-\frac12s^2(L^{-1}P_1(v_1\chi_2),v_1\chi_2), \quad k=4,5,6,7.
			\eln\right.
			\ee
			
			{\rm (2)} The eigenfunctions $\mathcal{U}_j(\xi,\eps)=(u_j(\xi,\eps),X_j(\xi,\eps),Y_j(\xi,\eps))$, $-1\le j\le 7$ are $C^\infty$  in $s$ and $\eps$, and satisfy for $\eps|s| \le r_0$:
			$$
			(\mathcal{U}_i, \mathcal{U}^*_j)_\xi=(u_i,\overline{u_j})_\xi-(X_i,\overline{X_j})-(Y_i,\overline{Y_j})=\delta_{ij},\quad -1\leq i,j\leq 7,
			$$
			where $ \mathcal{U}^*_j=(\overline{u_j},-\overline{X_j},-\overline{Y_j})$.
			Moreover, for $ j=-1,0,1,$
			\bq  \label{eigf2}
			\left\{\bln
			& P_0u_j(\xi,\eps)=h_{j}(\xi)+O(\eps s),\\
			& P_1u_j(\xi,\eps)=i\eps sL^{-1} P_1[(v\cdot\omega)  h_{j}(\xi)]+O(\eps^2 s^2),\\
			& X_j(\xi,\eps)=Y_j(\xi,\eps)\equiv 0,
			\eln\right.
			\eq
			for $k=2,3,$
			\bq  \label{eigf22}
			\left\{\bln
			&P_0u_k(\xi,\eps)= \frac{s}{\sqrt{1+s^2} }\[1+O\( \frac{\eps^2 s^4}{1+s^2}\)\] (v\cdot W^k )\chi_0,  \\
			&P_1u_k(\xi,\eps)=\frac{i \eps s^2}{\sqrt{1+s^2}} L^{-1}P_1[(v\cdot\omega)(v\cdot W^k )\chi_0]+O\(\frac{\eps^2 s^3}{\sqrt{1+s^2}}\),\\
			& X_k(\xi,\eps)= O\(\frac{\eps s^3}{1+s^{3}}\)(\omega\times W^k ),\quad   Y_k(\xi,\eps)= -\frac{i }{\sqrt{1+s^2}} \[1+O\( \frac{\eps^2 s^4}{1+s^2}\)\]W^k ,
			\eln\right.
			\eq
			and   for $k=4,5,6,7,$
			\bq
			\left\{\bln
			&P_0u_k(\xi,\eps)= \frac{1}{\sqrt{2(1+s^2)} }[1+O(\eps s)] (v\cdot W^k )\chi_0,  \\
			&P_1u_k(\xi,\eps)=\frac{i \eps s}{\sqrt{2(1+s^2)}} L^{-1}P_1[(v\cdot\omega)(v\cdot W^k )\chi_0]+O\(\frac{\eps^2 s^2}{\sqrt{1+s^2}}\),\\
			& X_k(\xi,\eps)=  \frac{1}{\sqrt{2(1+s^2)} }\[\eta_k+O( \eps s^2 )\](\omega\times W^k ),\\
			&  Y_k(\xi,\eps)= \frac{i  s}{\sqrt{2(1+s^2)}}\[1+O\(\frac{\eps s^2}{\sqrt{1+s^2}}\)\] W^k ,
			\eln\right.
			\eq
			where $h_j$, $ j=-1,0,1$ is defined in \eqref{hj},  and $W^k$, $2\le k\le 7$ are normal vectors satisfying $W^2=W^4=W^6$, $W^3=W^5=W^7$, and $W^k\cdot\omega=W^2\cdot W^3=0$.
		\end{lem}
		
		\begin{proof}
			The eigenvalue $\lambda_j(s,\eps)$ and the eigenfunction $\mathcal{U}_j(\xi,\eps)$ for $-1\le j\le 7$ can be constructed as follows.
			
			For $j=-1,0,1,$ we take $\lambda_j=\eps z_j(s,\eps)$ with  $ z_j(s,\eps)$ being the solution of the equation  $D_0(z,s,\eps)=0$ given in Lemma \ref{eigen_2}, and choose $X=Y=\omega\times W=0$ in \eqref{1A_12a}. Denote by $\Gamma_j(s,\eps)=(\tilde a_j,\tilde b_j,\tilde c_j) =:(W^j_0,\, (W\cdot\omega)^j,\, W^j_4)$  a  solution of system \eqref{1A_12} for $z=z_j(s,\eps)$, which means that it is an eigenvector of the matrix $\mathbb{M}$ defined in  \eqref{BM}. Then we have
			\begin{align*}
				\Gamma_j(s,\eps)=V_j(s)+O(\eps s),\quad j=-1,0,1,
			\end{align*}
			where $V_j(s)$ is defined in \eqref{VV}.
			Then we define the corresponding eigenfunction $\mathcal{U}_j(\xi,\eps)=(u_j(\xi,\eps),0,0)$, $j=-1,0,1$   by
			\bq \label{C_8}
			\left\{\bln
			&u_j(\xi,\eps)=  P_0u_j(\xi,\eps)+  P_1u_j(\xi,\eps),\\
			& P_0u_j(\xi,\eps)=\tilde a_j(s,\eps)\chi_0+\tilde b_j(s,\eps)(v\cdot\omega)\chi_0+
			\tilde c_j(s,\eps)\chi_4,\\
			& P_1u_j(\xi,\eps)=i\eps s(L-\eps z_j-i\eps s P_1(v\cdot\omega))^{-1} P_1[(v\cdot\omega)  P_0u_j(s,\omega)].
			\eln\right.
			\eq
			Here we remark that the microscopic part $P_1u_j$ is recovered by the macroscopic part $P_0u_j$ using \eqref{A_4}. Then we have
			$$
			\|P_0u_j(\xi,\eps)-h_j(\xi)\|= O(\eps s),
			$$
			which implies \eqref{eigf2}.
			
			For $2\le j\le 7,$ we take $\lambda_j=\eps y_j(s,\eps)$ with $y_2=y_3= z_{0}(s,\eps) $, $y_4=y_5= z_{-1}(s,\eps) $, $y_6=y_7=  z_{1}(s,\eps) $ and $ z_j(s,\eps)$ being the solution to the equation  $D_1(z,s,\eps)=0$ given in Lemma \ref{eigen_1}, and choose $W_0=W\cdot\omega=W_4=0$ in \eqref{1A_12}. And  the corresponding eigenfunctions $\mathcal{U}_j(\xi,\eps)=(u_j(\xi,\eps),X_j(\xi,\eps),Y_j(\xi,\eps))$, $2\le j\le 7$ are defined by
			\be  \label{C_2}
			\left\{\bln
			& u_j(\xi,\eps)=l_j (v\cdot W^j )\chi_0   +i \eps sl_j (L-\eps y_j-i\eps s P_1(v\cdot\omega))^{-1}P_1[(v\cdot\omega)(v\cdot W^j )\chi_0],\\
			& X_j(\xi,\eps)= \frac{\eps s^2l_jR_j(s,\eps)}{y_j^2+1+s^2} (\omega\times W^j ),\\
			& Y_j(\xi,\eps)= \frac{l_j }{is} \bigg[1+ \frac{\eps s^2y_jR_j(s,\eps)}{y_j^2+1+s^2}\bigg]W^j, \quad j=2,3,\\
			\eln\right.
			\ee
			and
			\be  \label{C_3}
			\left\{\bln
			& u_j(\xi,\eps)=l_j (v\cdot W^j )\chi_0   +i \eps sl_j (L-\eps y_j-i\eps s P_1(v\cdot\omega))^{-1}P_1[(v\cdot\omega)(v\cdot W^j )\chi_0],\\
			& X_j(\xi,\eps)=l_j \(y_j-\eps s^2R_j(s,\eps)\)(\omega\times W^j ),\\
			& Y_j(\xi,\eps)= \frac{i sl_j }{y_j}(y_j-\eps s^2R_j(s,\eps))W^j, \quad j=4,5,6,7,
			\eln\right.
			\ee
			where  $R_j(s,\eps)=R_{22}(\eps y_j,\eps s)$, and $l_j=l_j(s,\eps)$ is a complex value function that will be fixed later, and $W^k$, $2\le k\le 7$ are normal vectors satisfying $W^2=W^4=W^6$, $W^3=W^5=W^7$, and $W^k\cdot\omega=W^2\cdot W^3=0$.
			
			Rewrite the eigenvalue problem as
			$$\tilde{\AA}_{\eps}(\xi)\mathcal{U}_j(\xi,\eps)=\lambda_j(s,\eps)\mathcal{U}_j(\xi,\eps), \quad -1\leq j\leq 7.$$
			Taking the inner product $(\cdot,\cdot)_{\xi}$ of the above equation with $\mathcal{U}^*_k(\xi,\eps)$ and using the facts that
			\bgrs
			(\tilde{\AA}_{\eps}(\xi)U,V)_{\xi} =(U,\tilde{\AA}_{\eps}^*(\xi)V)_{\xi},\quad U,V\in D(\tilde{\AA}_{\eps}(\xi)),
			\\
			\tilde{\AA}_{\eps}^*(\xi)\mathcal{U}^*_j(\xi,\eps) =\overline{\lambda_j(s,\eps)}\mathcal{U}^*_j(\xi,\eps) ,
			\egrs
			we have
			$$
			(\lambda_j(s,\eps)-\lambda_k(s,\eps))(\mathcal{U}_j(\xi,\eps),\mathcal{U}^*_k(\xi,\eps))_{\xi}=0,\quad -1\le j, k\le 7.
			$$
			For $\eps|s|\le r_0$, we have $\lambda_j(s,\eps)\neq \lambda_k(s,\eps)$ for
			$j, k\in\{-1,0,1,2,4,6\}$ and $j\ne k$.  Thus,
			$$
			\(\mathcal{U}_j(\xi,\eps),\mathcal{U}^*_k(\xi,\eps)\)_{\xi}=0,\quad -1\leq j\neq k\leq 7.
			$$
			We can normalize $\mathcal{U}_j(\xi,\eps)$ by taking
			$$\(\mathcal{U}_j(\xi,\eps),\mathcal{U}^*_j(s,\eps)\)_{\xi} =1,\quad -1\le j\le 7.$$
			The coefficients $l_j(s,\eps)$ for $2\le j\le 7$  are determined by the normalization condition
			\bmas
			&l_j(s,\eps)^2\bigg(1-\eps^2s^2D_j -(y_j-\eps s^2R_{j})^2+\frac{s^2}{y^2_j}(y_j-\eps s^2R_{j})^2\bigg)=1,\quad j=4,5,6,7,\\
			& l_k(s,\eps)^2\bigg(1-\eps^2s^2D_k -\frac{\eps^2 s^4 R_k^2 }{(y_k^2+1+s^2)^2} +\frac{1}{s^2}\bigg(1+ \frac{\eps s^2y_kR_k }{y_k^2+1+s^2}\bigg)^2\bigg)=1, \quad k=2,3,
			\emas
			where $R_j =R_{22}(\eps y_j,\eps s)$, and
			$$
			D_j(s,\eps)=(R(\eps y_j,\eps s)P_1(v_1 \chi_2), R(\eps\overline{y_j},-\eps s)P_1(v_1 \chi_2)), \quad 2\le j\le 7.
			$$
			For $j=4,5,6,7$,
			$$
			l_j(s,\eps)^2\[1-O(\eps^2s^2)+(1+s^2)(1+O(\eps s))+s^2(1+O(\eps s))\]=1,
			$$
			which gives
			$$
			l_j(s,\eps)=\frac{1}{\sqrt{2(1+s^2)}}[1+O(\eps s)], \quad j=4,5,6,7.
			$$
			For $k=2,3,$
			$$
			l_k(s,\eps)^2\[1-O(\eps^2s^2)-O\(\frac{\eps^2s^4}{(1+s^2)^2}\)+\frac{1}{s^2 }\(1+O\(\frac{\eps^2s^4}{1+s^2}\)\)\]=1,
			$$
			which gives
			$$
			l_k(s,\eps)=\frac{s}{\sqrt{1+s^2}}\[1+O\(\frac{\eps^2 s^4}{1+s^2}\)\], \quad k=2,3.
			$$
			This completes the proof of the lemma.
		\end{proof}

		\subsection{Eigenvalues in $\eps |\xi|\ge r_1$}

		We now turn to study the asymptotic expansions of the eigenvalues
		and eigenvectors in the high frequency region.
		Firstly,  we have the following lemma.
		\begin{lem}[{\cite[Lemma 2.12]{YZ}}]\label{LP031}
			For any $\eps\in(0,1)$ and  $\delta>0$, if $\mathrm{Re} \lambda\geq -\nu_0+\delta$, then
			\begin{align*}
				&\|(\lambda-D_\eps(\xi))^{-1}K\|\leq C\delta^{-\frac{1}{2}}(1+\eps|\xi|)^{-\frac{1}{2}},\\
				&\|(\lambda-D_\eps(\xi))^{-1}\chi_j\|\leq C\delta^{-\frac{1}{2}}(1+\eps|\xi|)^{-\frac{1}{2}},\ \ \ j=0,1,2,3,4.
			\end{align*}
			Furthermore, there exists a sufficiently large $R_0>0$ such that $\lambda-\tilde\BB_\eps(\xi)$ is invertible for $\mathrm{Re}\lambda\geq -\nu_0+\delta$ and $\eps|\xi|>R_0$, and
			\begin{align*}
				\|(\lambda-\tilde\BB_\eps(\xi))^{-1}\chi_j\|\leq C\delta^{-\frac{1}{2}}(1+\eps|\xi|)^{-\frac{1}{2}}.
			\end{align*}
		\end{lem}
		
		Recalling
		the eigenvalue problem
		\bma
		\lambda f
		&=\tilde\BB_{\eps}(\xi)f-\eps v\chi_0\cdot(\omega\times X),\label{L_3}\\
		\lambda X&=-\eps\omega\times (f,v\chi_0)+i\eps\xi\times Y,\label{L_3a}\\
		\lambda Y&=-i\eps\xi\times X,\quad |\xi|\ne0.\label{A_13}
		\ema
		By Lemma \ref{LP031}, there exists a large constant $R_0>0$ such that the operator $\lambda-\tilde\BB_{\eps}(\xi)$ is invertible on $L^2_{\xi}(\R^3)$ for ${\rm Re}\lambda\ge-\nu_0/2$ and $\eps|\xi|>R_0$. Then it follows from \eqref{L_3} that
		\bq f=\eps (\tilde\BB_{\eps}(\xi)-\lambda)^{-1}v\chi_0\cdot(\omega\times X),\quad \eps |\xi|>R_0.\label{L_5}\eq
		By a rotation transform, we obtain
		\bma
		((\tilde\BB_{\eps}(\xi)-\lambda)^{-1}\chi_i,\chi_j)=&\,\omega_i\omega_j((\tilde\BB_{\eps}(se_1)-\lambda)^{-1}\chi_1,\chi_1) \nnm\\
		&+(\delta_{ij}-\omega_i\omega_j)((\tilde\BB_{\eps}(se_1)-\lambda)^{-1}\chi_2,\chi_2), \label{B_2a}
		\ema
		where $e_1=(1,0,0)$, $s=|\xi|$ and $\omega=\xi/|\xi|$. Substituting \eqref{L_5} and \eqref{B_2a} into \eqref{L_3a}, we obtain
		\be
		\lambda X =\eps^2((\tilde\BB_{\eps}(se_1)-\lambda)^{-1}\chi_2,\chi_2)X+i\eps\xi\times Y, \quad \eps s>R_0. \label{A_12}
		\ee
		Multiplying \eqref{A_12} by $\lambda$ and using \eqref{A_13}, we obtain
		$$ \(\lambda^2-\eps^2((\tilde\BB_{\eps}(se_1)-\lambda)^{-1}\chi_2,\chi_2)\lambda+\eps^2s^2\)X=0, \quad \eps s>R_0.$$
		Denote
		\bq D_2(z,s,\eps)=z^2-\eps((\tilde\BB_{\eps}(se_1)-\eps z)^{-1}\chi_2,\chi_2)z+s^2,\quad \eps s>R_0.\label{D3a}\eq
		
		The eigenvalues $\lambda=\eps z$ can be solved by $D_2(z,s,\eps)=0$.
		We study the equation \eqref{D3a} in the follwing lemma.
		
		\begin{lem}\label{eigen_5}
			There exists a large constant $r_1>0$ such that the equation $D_2(z,s,\eps)=0$ has two solutions $z_j(s,\eps)=ji s+\zeta_j(s,\eps) $, $j=\pm1$ for $\eps s>r_1$, where $\zeta_j(s,\eps)$ is a $C^{\infty}$ function in $s$ and $\eps$ for   $\eps s>r_1$  satisfying
			\be\label{eigen_h1}
			\frac{C_1}{ s}\leq -\mathrm{Re} \zeta_j(s,\eps)\leq \frac{C_2}{ s},\quad |\mathrm{Im}\zeta_j(s,\eps)|\leq C_3\frac{\ln \eps s}{ s},
			\ee
			for some $C_1,C_2,C_3>0$.
		\end{lem}
		\begin{proof}
			For any fixed $s,\eps$ satisfying $\eps s>R_1$, we define a function of $z$ as
			\bq\label{Gj}
			G_j(z,s,\eps)=\frac{1}{2}\left(R_0(z,s,\eps)+j\sqrt{R_0(z,s,\eps)^2-4s^2}\right),\quad\quad j=\pm 1,
			\eq
			where $R_0(z,s,\eps)=\eps((\tilde\BB_{\eps}(se_1)-\eps z)^{-1}\chi_2,\chi_2)$. One can check that a solution of $D_2(z,s,\eps)=0$ is a fixed point of the map $z\to G_j(z_j,s,\eps)$.
			
			For $\delta>0$ small enough, and $z,z_1,z_2\in B(jis,\delta)$, we have
			\begin{align*}
				|G_j(z,s,\eps)-jis|&\leq \frac{1}{2}|R_0(z,s,\eps)|+\frac{R_0(z,s,\eps)^2}{2|\sqrt{R_0(z,s,\eps)^2-4s^2}+2is|}\leq \delta,\\
				|G_j(z_1,s,\eps)-G_j(z_2,s,\eps)|&\leq \frac{1}{2}|R_0(z_1,s,\eps)-R_0(z_2,s,\eps)|\\
				&\quad+\frac{R_0(z_1,s,\eps)^2-R_0(z_2,s,\eps)^2}{2\sqrt{R_0(z_1,s,\eps)^2-4s^2}+2\sqrt{R_0(z_2,s,\eps)^2-4s^2}}\\
				&\leq C\eps^2|z_1-z_2|\leq \frac{1}{2}|z_1-z_2|.
			\end{align*}
			By the contraction mapping theorem, there exists a unique fixed point $z_j(s,\eps)$ satisfying $z_j(s,\eps)=G_j(z_j(s,\eps),s,\eps)$, which is a solution to the equation $D_2(z,s,\eps)=0$.
			Set $\zeta_j(s,\eps)=z_j(s,\eps)-jis$. Then
			\be\label{zej}
			|\zeta_j(s,\eps)|\leq \frac{1}{2}|R_0(z_j,s,\eps)|+\frac{R_0(z_j,s,\eps)^2}{2|\sqrt{R_0(z_j,s,\eps)^2-4s^2}+2is|}\leq C\eps (\eps s)^{-\frac{1}{2}}.
			\ee
			By \eqref{Gj}, we obtain
			$$
			\zeta_j(s,\eps)=\frac12 R_j(\zeta_j,s,\eps)+\frac12 \frac{jR_j^2(\zeta_j,s,\eps)}{\sqrt{R_j(\zeta_j,s,\eps)^2-4s^2}+2is}=:I_1+I_2,
			$$
			where
			$$
			R_j(\zeta_j,s,\eps)=R_0(z_j,s,\eps)=\eps((\tilde\BB_{\eps}(se_1)-\eps \zeta_j+\eps jis)^{-1}\chi_2,\chi_2).
			$$
			First, we estimate $I_1$. For this, we decompose
			\begin{align*}
				-	(\tilde\BB_{\eps}(se_1)-\eps \zeta_j+\eps jis)^{-1}&=(K-\nu(v)-i\eps sv_1 -i\eps s^{-1}v_1 P_{ d}-\eps \zeta_j+\eps jis)^{-1}\\
				&=X_j(s,\eps)+Z_j(s,\eps),
			\end{align*}
			where
			\be\label{XYZ}
			\left\{\bln
			&X_j(s,\eps)=(\nu(v)+i\eps s(v_1+j))^{-1},\\
			&Z_j(s,\eps)=\(\mathrm{Id}-Y_j(s,\eps)\)^{-1}Y_j(s,\eps)X_j(s,\eps),\\
			&Y_j(s,\eps)=X_j\(K -i\eps s^{-1}v_1 P_{ d}-\eps \zeta_j\).
			\eln\right.
			\ee
			Then we have
			$$
			I_1=-\frac{\eps}{2}\(X_j(s,\eps)\chi_2,\chi_2\)-\frac{\eps}{2}\(Z_j(s,\eps)\chi_2,\chi_2\)=:I_{3}+I_{4}.
			$$
			It holds that
			\begin{align*}
				I_3&=-\frac{\eps}{2}\int_{\mathbb{R}^3}\frac{\nu}{\nu^2+\eps^2s^2(v_1\pm1)^2}v_2^2 Mdv+\frac{\eps}{2}\int_{\mathbb{R}^3}\frac{i\eps s (v_1\pm 1)}{\nu^2+\eps^2s^2(v_1\pm1)^2}v_2^2 Mdv\\
				&=:\mathrm{Re} I_{3}+i\mathrm{Im} I_{3}.
			\end{align*}
			By changing variable $(u_1,u_2,u_3)=(\eps s (v_1\pm 1),v_2,v_3)$, we obtain for $\eps s\ge 2$ that
			\begin{align}
				-	\mathrm{Re} I_{3}&\geq \frac{1}{2s}\int_{\R^3}\frac{\nu_0 u_2^2}{\nu_1^2(1+(\frac{u_1}{\eps s}\mp1)^2+u^2_2+u^2_3)+u_1^2}e^{-\frac{1}{2}(\frac{u_1}{\eps s }\pm1)^2}e^{-\frac{u_2^2+u_3^2}{2}}du\nnm\\
				&\ge \frac{C}{s}\intr \frac{u_2^2}{\nu_1^2(3+2u_1^2+u^2_2+u^2_3)+u_1^2} e^{-\frac{u_1^2+u_2^2+u_3^2}2}du_1du_2du_3\ge \frac{C_1}{s},\label{rei3l}
				\\
				-	\mathrm{Re} I_{3}&\leq \frac{C}{ s}\int_{\R^3}\frac{1}{\nu_0^2+u_1^2}e^{-\frac{1}{4}(\frac{u_1}{\eps s }\pm1)^2}e^{-\frac{u_2^2+u_3^2}{4}}du_1du_2du_3\nnm\\
				&\le \frac{C}{s}\int^\infty_0 \frac{1}{\nu_0^2+u_1^2}du_1\le \frac{C_2}{s}, \label{rei3u}
			\end{align}
			and
			\begin{align}
				|\mathrm{Im} I_{3}|&\leq \frac{C}{ s}\int_{\R^3}\frac{|u_1|}{\nu_0^2+u_1^2}e^{-\frac{1}{4}(\frac{u_1}{\eps s }\pm1)^2}e^{-\frac{u_2^2+u_3^2}{4}}du_1du_2du_3\nonumber\\
				&\leq \frac{C}{ s}\int_{0}^{\eps s}\frac{u_1}{\nu_0^2+u_1^2}du_1+\frac{C}{\eps s^2}\int_{\eps s }^\infty e^{-\frac{1}{4}(\frac{u_1}{\eps s }\pm1)^2}du_1\nonumber\\
				&\leq C_3\frac{\ln \eps s}{s}. \label{imi3}
			\end{align}
			Then we consider $I_4$. Note that $(X_j(s,\eps)\chi_2,\chi_0)=0$. This together with \eqref{XYZ} implies that
			$$
			Z_j(s,\eps)\chi_2=\(\mathrm{Id}-Y_j(s,\eps)\)^{-1}X_j(s,\eps)(K-\eps \zeta_j)X_j(s,\eps)\chi_2.
			$$
			From \eqref{L_1} we have
			$$
			|k(v,u)|\leq C\frac{1}{|\bar u-\bar v|}e^{-\frac{|u-v|^2}{8}},\ \ \bar u=(u_2,u_3).
			$$
			Hence, we obtain
			\begin{align*}
				|KX_{\pm1}\chi_2|&\leq C\int_\R e^{-\frac{|u_1-v_1|^2}{8}}\frac{\nu+|(u_1\pm1)\eps s|}{\nu_0^2+|(u_1\pm1)\eps s|^2}e^{-\frac{u_1^2}{4}}du_1\int_{\R^2}\frac{1}{\bar u-\bar v}e^{-\frac{|\bar u-\bar v|^2}{8}}e^{-\frac{|\bar u|^2}{4}}d\bar u\\
				&\leq Ce^{-\frac{|v|^2}{8}}\int_\R \frac{1+|(u_1\pm1)\eps s|}{\nu_0^2+|(u_1\pm1)\eps s|^2}e^{-\frac{u_1^2}{8}}du_1\leq C\frac{\ln  \eps s}{\eps s}e^{-\frac{|v|^2}{8}}.
			\end{align*}
			This and \eqref{rei3l} yield
			$$
			\|X_{\pm1}KX_{\pm1}\chi_2\|^2\leq
			C\frac{\ln ^2\eps s}{\eps^2 s^2}\int_{\R^3}\frac{1}{\nu_0^2+\eps^2s^2(v_1\pm 1)^2}e^{-\frac{|v|^2}{4}}dv\leq C\frac{\ln ^2\eps s}{\eps^3 s^3}.
			$$
			By \eqref{zej} and Lemma \ref{LP031}, it holds for $\eps s>R_1$,
			$$
			\|(\mathrm{Id}-Y_{\pm1}(s,\eps))^{-1}\|\leq 2.
			$$
			Thus,
			\bma\label{i4}
			|I_4|&\leq C\eps|X_j(K -\eps \zeta_j)X_j\chi_2,(\mathrm{Id}-Y_j^*)^{-1}\chi_2| \nnm\\
			&\leq C\eps(\|X_jKX_j\chi_2\|+\eps|y_j|\|X_j^2\chi_2\|)\leq C\frac{\ln \eps s}{\eps^{\frac{1}{2}}s^{\frac32}}.
			\ema
			Finally, for $I_2$ we have
			\be\label{i2b}
			|I_2|\leq \frac{C}{s}|R_j(y_j,s,\eps)|^2\leq\frac{C}{s}|I_1|^2\leq \frac{\ln ^2\eps s}{s^3}.
			\ee
			Then \eqref{eigen_h1} follows from \eqref{rei3l}, \eqref{rei3u}, \eqref{imi3}, \eqref{i4} and \eqref{i2b}. 
		\end{proof}

		With the help of Lemma \ref{eigen_5}, we have the eigenvalues $\beta_j(|\xi|,\eps)$
		and the corresponding eigenvectors $\mathcal{V}_j(\xi,\eps)$ of the operator $\tilde{\AA}_{\eps}(\xi)$ for $\eps|\xi|\ge r_1$ as follows.
		
		\begin{lem}\label{eigen_4}
			{\rm (1)} There exists a constant $r_1>0$ such that the spectrum $\sigma(\tilde{\AA}_{\eps}(\xi))\cap \{\lambda\in\mathbb{C}\,|\,\mathrm{Re}\lambda>-\mu/2\}$  consists of four eigenvalues $\{\beta_j(s,\eps),\ j=1,2,3,4\}$ for  $\eps s>r_1$ and $s=|\xi|$. In particular, the eigenvalues $\beta_j(s,\eps)$ are $C^{\infty}$ functions in $s$ and $\eps $ and satisfy the following expansion for $\eps s>r_1$:
			\be   \label{specr1a}
			\left\{\bln
			&\beta_1(s,\eps) =\beta_2(s,\eps) = -\eps is+\eps\zeta_{-1}(s,\eps), \\
			&\beta_3(s,\eps) =\beta_4(s,\eps) = \eps is+\eps\zeta_{1}(s,\eps),
			\eln\right.
			\ee
			where $\zeta_{\pm1}(s,\eps)$ is a $C^{\infty}$ function in $s$ and $\eps $ for $\eps s>r_1$ satisfying
			\eqref{eigen_h1}.
			
			{\rm (2)} The eigenvectors $\mathcal{V}_j(\xi,\eps)=\(w_j(\xi,\eps),X_j(\xi,\eps),Y_j(\xi,\eps)\)$, $j=1,2,3,4$ are $C^{\infty}$ in $s$ and $\eps$,  and satisfy for $\eps s>r_1$:
			\bq   \label{eigf2a1}
			\left\{\bln
			&(\mathcal{V}_i ,\mathcal{V}^*_j )=(w_i,\overline{w_j})-(X_i,\overline{X_j})-(Y_i,\overline{Y_j})=\delta_{ij}, \quad 1\le i,j\le 4,\\
			&w_j(\xi,\eps)=\eps c_j(s,\eps)\(\beta_j(s,\eps)-L + i \eps s (v\cdot\omega)+i\eps\frac{  v\cdot\omega}{s} P_d\)^{-1}(v\cdot e_j)\chi_0,\\
			&X_j(\xi,\eps)=c_j(s,\eps)\omega\times \tilde e_j,\quad Y_j(\xi,\eps)=\frac{ i\eps s c_j(s,\eps)}{\beta_{j}(s,\eps)}\tilde e_j,
			\eln\right.
			\eq
			where  $ \mathcal{V}^*_j=(\overline{w_j},-\overline{X_j},-\overline{Y_j})$, and $\tilde e_k$, $k=1,2,3,4$ are normal vectors satisfying $\tilde e_1=\tilde e_3$, $\tilde e_2=\tilde e_4$, and $\tilde e_k\cdot\omega=\tilde e_1\cdot \tilde e_2=0$, and $c_j(s,\eps)$ are $C^{\infty}$ functions of $s$ and $\eps$ for $\eps s>r_1$ satisfying
			\be \label{eigf3a1}
			c_j(s,\eps) = i\frac1{\sqrt{2}}+\eps^2 O\(\frac{1}{\eps s}\).
			\ee
		\end{lem}
		
		\begin{proof}
			The eigenvalues $\beta_j(s,\eps)$ and the eigenvectors $\mathcal{V}_j(\xi,\eps) $  for $j=1,2,3,4$ can be constructed as follows. We take $\beta_1=\beta_2=\eps z_{-1}(s,\eps)$ and $\beta_3=\beta_4=\eps z_{1}(s,\eps)$, where $z_{\pm1}(s,\eps)$ are the solutions of the equation  $D_2(z,s,\eps)=0$ defined in Lemma \ref{eigen_5}.
			The corresponding eigenvectors $\mathcal{V}_j(\xi,\eps)=(w_j(\xi,\eps),X_j(\xi,\eps),Y_j(\xi,\eps))$, $1\le j\le 4$  are given by
			$$
			\left\{\bln
			&w_j(\xi,\eps)=\eps c_j(s,\eps)\big(\beta_j(s,\eps)-\tilde\BB_\eps(\xi)\big)^{-1}(v\cdot \tilde e_j\chi_0), \\
			&X_j(\xi,\eps)=c_j(s,\eps)\omega\times \tilde e_j,\quad Y_j(\xi,\eps)= \frac{ i \eps s c_j(s,\eps)}{\beta_j(s,\eps)} \tilde e_j ,
			\eln\right.
			$$
			where  $\tilde e_j$, $j=1,2,3,4$ are normal vectors satisfying $\tilde e_1=\tilde e_3$, $\tilde e_2=\tilde e_4$, and $\tilde e_j\cdot\omega=\tilde e_1\cdot \tilde e_2=0$.  This implies \eqref{eigf2a1} by the definition \eqref{B(xi)}. It is easy to verify that   $(\mathcal{V}_1 ,  \mathcal{V}_2^* )=(\mathcal{V}_3 ,  \mathcal{V}_4^* )=0,$ where $ \mathcal{V}^*_j=(\overline{w_j},-\overline{X_j},-\overline{Y_j})$ is the eigenvector of $\tilde{\AA}_{\eps}(\xi)^*$ corresponding to the eigenvalue $\overline{\beta_{j}}$.
			
			Rewrite the eigenvalue problem as
			$$
			\tilde{\AA}_{\eps}(\xi)\mathcal{V}_j(\xi,\eps) =\beta_j(s,\eps)\mathcal{V}_j(\xi,\eps), \quad j=1,2,3,4.
			$$
			By taking the inner product
			$(\cdot,\cdot)_{\xi}$ of above with $\mathcal{V}^*_j(s,\eps) $, and using the fact that
			\bgrs
			(\tilde{\AA}_{\eps}(\xi) U,V)_{\xi}=(U,\tilde{\AA}_{\eps}(\xi)^*V)_{\xi},\quad U,V\in  D(\tilde{\AA}_{\eps}(\xi)) ,\\
			\tilde{\AA}_{\eps}(\xi)^*\mathcal{V}^*_j(\xi,\eps)=\overline{\beta_j(s,\eps)} \mathcal{V}^*_j(\xi,\eps),
			\egrs
			we have
			$$
			(\beta_i(s,\eps)-\beta_{j}(s,\eps))\(\mathcal{V}_i(\xi,\eps),\mathcal{V}^*_j(\xi,\eps)\)_{\xi}=0,\quad 1\le i,j\le 4.
			$$
			Since  $\beta_k(s,\eps)\neq \beta_{j}(s,\eps)$ for  $k=1,3,\, j=2,4$ and $P_dw_j(\xi,\eps)=0$,
			we have the orthogonal relation
			$$
			\(\mathcal{V}_k(\xi,\eps),\mathcal{V}^*_j(\xi,\eps)\)_{\xi}=\(\mathcal{V}_k(\xi,\eps),\mathcal{V}^*_j(\xi,\eps)\)=0,\quad 1\leq k\neq j\leq 4.
			$$
			This can be normalized so that
			$$\(\mathcal{V}_j(\xi,\eps),\mathcal{V}^*_j(\xi,\eps)\)=1, \quad j=1,2,3,4.$$
			Precisely, the coefficients $c_j(s,\eps)$  are determined by the normalization condition:
			$$
			c_j^2(s,\eps)\bigg(-1+\frac{\eps^2s^2}{\beta^2_{j}(s,\eps)}+\eps^2D_j(s,\eps)\bigg)=1,\quad j=1,2,3,4,
			$$
			where $D_j(s,\eps)=((\tilde\BB_{\eps}(se_1)-\beta_j)^{-1}\chi_2, (\tilde\BB_{\eps}(-se_1)-\overline{\beta_j})^{-1} \chi_2) $.
			Since
			\bmas
			\frac{\beta_{j}^2(s,\eps)}{\eps^2 s^2}&=\(i+O\(\frac{ \ln \eps s}{ s^2}\)\)^2=-1 +O\(\frac{\ln \eps s}{  s^2}\),  \\
			D_j(s,\eps)&=O(1) \|(\tilde\BB_{\eps}(se_1)-\beta_j)^{-1}\chi_2\|^2=O\(\frac{1}{\eps s}\),
			\emas
			it follows that
			\bmas
			c_j^2(s,\eps)=&-\bigg(2-\bigg(\frac{\eps^2s^2}{\beta_j^2(s,\eps)}+1\bigg)-\eps^2D_j(s,\eps)\bigg)^{-1}\\
			=&-\frac12 +\eps^2O\(\frac{1}{\eps s}\).
			\emas
			Thus, we obtain \eqref{eigf2a1} and \eqref{eigf3a1} so that
			the proof of the theorem is completed.
		\end{proof}

		With the help of Lemmas~\ref{LP01}, \ref{spectrum2}, \ref{eigen_4a} and \ref{eigen_4}, we can make a detailed analysis on the semigroup $ e^{\frac{t}{\eps^2}\mathcal{\tilde{\AA}}_{\eps}(\xi)}$ in terms of an  argument similar to that of Theorem~3.4 in \cite{Li4}. For brevity, we omit the details of the proof.
		
		\begin{thm}\label{rate1}
			The semigroup $e^{\frac{t}{\eps^2}\tilde{\AA}_{\eps}(\xi)}$ with $\xi\neq0$ can be decomposed into
			\bq  e^{\frac{t}{\eps^2}\tilde{\AA}_{\eps}(\xi)}U=S_1(t,\xi,\eps)U+S_2(t,\xi,\eps)U+S_3(t,\xi,\eps)U,\quad \forall\,U\in L^2_{\xi}(\R^3_v)\times \C^3_{\xi}\times \C^3_{\xi},  \label{B_0}\eq
			where
			\bma
			S_1(t,\xi,\eps)U&=\sum^7_{j=-1}e^{\frac{t}{\eps^2}\lambda_j(|\xi|,\eps)}\(U, \mathcal{U}^*_j(\xi,\eps) \)_{\xi} \mathcal{U}_j(\xi,\eps)1_{\{\eps |\xi|\le r_0\}}, \label{S1}\\
			S_2(t,\xi,\eps)U&=\sum^4_{k=1}e^{\frac{t}{\eps^2}\beta_k(|\xi|,\eps)}\(U, \mathcal{V}^*_k(\xi,\eps) \) \mathcal{V}_k(\xi,\eps)1_{\{\eps |\xi|\ge r_1\}}, \label{S2}
			\ema
			with $(\lambda_j(|\xi|,\eps),\mathcal{U}_j(\xi,\eps))$ and $(\beta_k(|\xi|,\eps),\mathcal{V}_k(\xi,\eps))$ being the eigenvalue and eigenvector of the operator $\tilde{\AA}_{\eps}(\xi)$ for $\eps |\xi|\le r_0$ and $\eps|\xi|\ge r_1$ respectively,
			and $S_3(t,\xi,\eps)U $ satisfies for two constants $ d>0$ and $C>0$ independent of $\xi$ and $\eps$ that
			\be
			\|S_3(t,\xi,\eps)U\|_{\xi}\le Ce^{-\frac{ dt}{\eps^2}}\|U\|_{\xi}. \label{S3}
			\ee
		\end{thm}

		\section{Fluid approximation of semigroups}
		\setcounter{equation}{0}
		\label{sect3}
		
		In this section,  we give the first and second order fluid approximations of the semigroup $e^{\frac{t}{\eps^2}\AA_\eps}$, which will be used to prove the convergence and establish the convergence rate of the solution to  the VMB system towards the solution to the  NSMF system.

		For any $U_0=(f_0,E_0,B_0)\in H^l$, set
		\be
		e^{\frac{t}{\eps^2}\AA_{\eps}(\xi)}\hat U_0=\( \hat{f} ,-\frac{\i\xi}{|\xi|^2}(\hat{f},\chi_0)-\frac{\xi}{|\xi|}\times \hat{X},-\frac{\xi}{|\xi|}\times \hat{Y}\),
		\label{solution1}
		\ee
		where
		\bmas
		e^{\frac{t}{\eps^2}\tilde{\AA}_{\eps}(\xi)}\hat V_0
		&=( \hat{f},\hat{X},\hat{Y})\in L^2_{\xi}(\R^3_v)\times \mathbb{C}^3_{\xi}\times \mathbb{C}^3_{\xi},\\
		\hat V_0&=\(\hat f_0,\frac{\xi}{|\xi|}\times \hat E_0,\frac{\xi}{|\xi|}\times \hat B_0\).
		\emas
		
		For any  $U_0=(f_0,E_0,B_0) \in H^l$, set
		\be
		e^{\frac{t}{\eps^2}\AA_\eps}U_0 =(\mathcal{F}^{-1}e^{\frac{t}{\eps^2}\AA_\eps(\xi)}\mathcal{F})U_0. \label{solution2}
		\ee
		Then  $e^{\frac{t}{\eps^2}\AA_{\eps}}U_0$ is the solution to the system  \eqref{LVMB1}. By Lemma \ref{SG_1}, it holds that
		$$
		\|e^{\frac{t}{\eps^2} \AA_{\eps}} U_0\|^2_{H^l} =\intr (1+|\xi|^2)^l\|e^{\frac{t}{\eps^2}\tilde{\AA}_{\eps}(\xi)}\hat V_0\|^2_{\xi} d\xi\le \intr (1+|\xi|^2)^l\|\hat V_0\|^2_{\xi} d\xi
		=\|U_0\|^2_{H^l},
		$$
		where we had used $\|\hat V_0\|^2_{\xi}=\|\hat U_0\|^2.$
		
		\subsection{Semigroup of the linear NSMF system}
		
		In this subsection, we are going to study the solution to the linear NSMF system. First of all, we introduce the following linearized  NSMF system to \eqref{NSM_2} for $(n,m,q,E,B)$:
		\be \label{LNSM1}
		\left\{\bal
		\Tdx\cdot m=0,\quad \Tdx n+ \sqrt{\frac23}\Tdx q- E=0,\\
		\dt (m-\Delta^{-1}_xm)-\kappa_0\Delta_x m +\Tdx p=G_{1}, \\
		\dt (q-\sqrt{\frac23}n) - \kappa_1\Delta_x q= G_{2},\\
		E=\Tdx\Delta^{-1}_xn, \quad \Tdx\times B=m, \quad \Tdx\cdot B=0,
		\ea\right.
		\ee
		where $G_{1} \in \R^3$ and $G_{2}\in \R$ are given functions, $p$ is the pressure satisfying $p=\Delta^{-1}_x\divx G_{1}$, and the initial data $(n,m,q,E,B)(0)$  satisfy   \eqref{NSP_5i}.
		
		For any  $\hat{V}_0=(\hat{f}_0(\xi) ,\hat{E}_0(\xi),\hat{B}_0(\xi))\in N_0\times \C^3_{\xi}\times \C^3_{\xi}$, we define
		\be
		\tilde{Y}_1(t,\xi)\hat{V}_0=\sum_{j=0,2,3}e^{-b_j(|\xi|)t}\big(\hat{V}_0,\Lambda_j(\xi) \big)_{\xi}\Lambda_j(\xi),  \label{v1}
		\ee
		where $b_j(|\xi|)$ and  $\Lambda_j(\xi)$  $(j=0,2,3)$ are defined by \eqref{bj} and \eqref{hj} respectively.
		
		For any $\hat{U}=(\hat{f}(\xi),\hat{X}(\xi),\hat{Y}(\xi))\in L^2\times \C^3_{\xi}\times \C^3_{\xi} $, define
		\begin{align}\label{defQ}
			\Q(\xi)\hat U=\(\hat{f},-\frac{i\xi}{|\xi|^2}(\hat{f},\chi_0)-\frac{\xi}{|\xi|}\times \hat X,-\frac{\xi}{|\xi|}\times \hat Y\).
		\end{align}
		Then we have
		\be\label{normeq}
		\|\Q\hat U\|=\|\hat U\|_\xi.
		\ee
		
		For any $\hat{U}_0=(\hat{f}_0(\xi),\hat{E}_0(\xi),\hat{B}_0(\xi))\in N_0\times \C^3 \times \C^3 $,   set
		\bq
		Y_1(t,\xi)\hat U_0=\Q(\xi)\tilde{Y}_1(t,\xi)\hat V_0,
		\label{solution1a}
		\eq
		with
		$$
		\left.\bln V_0&=\(\hat{f}_0,\frac{\xi}{|\xi|}\times \hat E_0,\frac{\xi}{|\xi|}\times \hat B_0\).
		\eln\right.
		$$
		Set
		$$
		Y_1(t)U_0=(\mathcal{F}^{-1}Y_1(t,\xi)\mathcal{F})U_0. \label{v2}
		$$

		Then, we can represent the solution to the NSMF system \eqref{LNSM1}  by the semigroup $Y_1(t)$.
		
		\begin{lem} \label{sem}
			For any  $U_0=(f_0,E_0,B_0)\in L^2$ and $G_j\in L^1_t(L^2_x)$, $j=1,2$, we define $U=(f,E,B)$ by
			\be
			U(t,x,v)=Y_1(t){P_AU_0}+\intt Y_1(t-s)H_1(s)ds, \label{U_2}
			\ee
			where 
			$$
			H_1(t,x,v)=(G_{1}(t,x)\cdot v\chi_0+G_{2}(t,x)\chi_4,0,0).
			$$
			Let $(n,m,q)=((f,\chi_0),(f,v\chi_0),(f,\chi_4))$. Then $(n,m,q,E,B)(t,x)\in L^\infty_t(L^2_x)$ is the unique global solution to the linear NSM system \eqref{LNSM1} with the initial data $(n,m,q,E,B)(0)$   satisfying \eqref{NSP_5i}.
		\end{lem}
		
		\begin{proof}
			Let $s=|\xi|$ and $\omega=\xi/|\xi|$. By taking Fourier transform to \eqref{LNSM1}, we have
			\bma
			&i\xi\cdot \hat{m}=0,\quad (1+|\xi|^{-2})\hat{n} +\sqrt{\frac23}\hat{q}=0,\label{LNSP_1a}\\
			&(1+|\xi|^{-2})\dt \hat{m}+\kappa_0|\xi|^2 \hat{m}=\hat{G}_{1\perp},\label{LNSP_3a}\\
			&\dt \bigg(\hat{q}-\sqrt{\frac{2}{3}}\hat n\bigg) +\kappa_1|\xi|^2 \hat{q}=\hat{G}_{2},\label{LNSP_4a}
			\ema
			where $ \hat G_{1\perp}=-\omega\times\omega\times \hat G_1$.
			The initial data $(\hat{n},\hat{m},\hat{q})(0)$ satisfies
			\begin{align*}
				& \hat{m}(0,\xi)= (P_0\hat{f}_0,v\chi_0)_{\bot},\\
				&\hat n(0,\xi)=-\frac{\sqrt{6}s^2}{3+5s^2}\bigg(P_0\hat f_0, \chi_4-\sqrt{\frac{2}{3}}\chi_0\bigg),\\
				&\hat q(0,\xi)=\frac{ 1+s^2}{3+5s^2}\bigg(P_0\hat f_0, \chi_4-\sqrt{\frac{2}{3}}\chi_0\bigg).
			\end{align*}
			Then, it follows from \eqref{LNSP_4a} and \eqref{LNSP_1a} that
			\bma
			\hat q(t,\xi)&=e^{-b_0(s)t}\hat q(0,\xi)+\frac{3+3s^2}{3+5s^2}\int_0^te^{-b_0(s)(t-\tau)}\hat G_2(\tau)d\tau\nnm\\
			&=e^{-b_0(s)t}\big(\hat V_0,\Lambda_0(\xi)\big)_\xi\(h_0(\xi),\chi_4\)\nnm\\
			&\quad+ \int_0^te^{-b_0(s)(t-\tau)}\big(\hat H_1(\tau),\Lambda_0(\xi)\big)_\xi\(h_0(\xi),\chi_4\)d\tau,\label{n3}
			\ema
			and
			\bma
			\hat{n}(t,\xi)&=e^{-b_0(s)t} \big(\hat{V}_0,\Lambda_0(\xi)\big)_\xi(h_0(\xi),\chi_0)\nnm\\
			&\quad +\intt e^{-b_0(s)(t-\tau)} \big(\hat{H}_1(\tau),\Lambda_0(\xi)\big)_\xi(h_0(\xi),\chi_0)d\tau. \label{n4}
			\ema
			By \eqref{LNSP_1a} and \eqref{LNSP_3a} and noting that $ \hat m=\hat m_{\bot}$, we have
			\bma
			\hat{m}(t,\xi)&=e^{-b_2(s)t} \hat{m}(0)  +\frac{s^2}{1+s^2}\intt e^{-b_2(s)(t-\tau)} \hat{G}_1(\tau)_{\bot}d\tau\nnm\\
			&=\sum_{j=2,3}e^{-b_j(s)t} \big(\hat{V}_0,\Lambda_j(\xi)\big)_\xi(h_j(\xi),v\chi_0) \nnm\\
			&\quad+\sum_{j=2,3}\intt e^{-b_j(s)(t-\tau)} \big(\hat{H}_1(\tau),\Lambda_j(\xi)\big)_\xi(h_j(\xi),v\chi_0)d\tau. \label{n5}
			\ema
			Moreover, by \eqref{LNSM1} we have
			\begin{align}
				&\hat E=-is^{-1}\hat n\omega=-is^{-1}(\hat f,\chi_0) \omega,\label{n6}\\
				&\hat B=is^{-1}\omega\times \hat m=is^{-1}\omega\times (\hat f,v\chi_0).\label{n7}
			\end{align}
			Noting that  $(h_0(\xi),v\chi_0)=0$ and $(h_j(\xi),\chi_0)=(h_j(\xi),\chi_4)=0$, $j=2,3$, we can prove \eqref{U_2} by using \eqref{n3}--\eqref{n7}. This completes the proof of the lemma.
		\end{proof}

		\subsection{Fluid approximation of $e^{\frac{t}{\eps^2}\AA_{\eps}}$}

		We give the following  preparation lemmas which will be used to study the  fluid dynamical approximations of the semigroups $e^{\frac{t}{\eps^2}\AA_\eps}$.
		
		\begin{lem} \label{S2a}
			For any $V_0\in N_0\times \C^3_{\xi}\times \C^3_{\xi}$, we have
			\be
			\|S_{3}(t,\xi,\eps)V_0\| \le C\(\eps |\xi| 1_{\{\eps |\xi|\le r_0\}}+1_{\{\eps |\xi| \ge r_0\}}\)e^{-\frac{dt}{\eps^2}}\|V_0\|  ,\label{S6}
			\ee
			where $S_3(t,\xi,\eps)$ is given in Theorem {\rm \ref{rate1}.}
		\end{lem}
		
		\begin{proof}
			Define a projection $P_{\eps}(\xi)$ by
			$$P_{\eps}(\xi)V=\sum^3_{j=-1}\(V,\mathcal{U}^*_j(\xi,\eps)\)_{\xi} \mathcal{U}_j(\xi,\eps),\quad \forall V\in L^2(\R^3_v)\times \C^3_{\xi}\times \C^3_{\xi},$$
			where $ \mathcal{U}_j(\xi,\eps)$, $j=-1,0,1,2,3$ are the eigenfunctions of $\tilde{\AA}_{\eps}(\xi)$ defined by \eqref{Axi} for $\eps|\xi|\le r_0$.
			
			By Theorem \ref{rate1}, we can assert that
			\be S_1(t,\xi,\eps)=S(t,\xi,\eps)1_{\{\eps|\xi|\le r_0\}}P_{\eps}(\xi). \label{S1a}
			\ee
			Indeed, it follows from Lemma \ref{semigroup-1} that for $\eps|\xi|\le r_0$ and $\kappa>0$,
			\bmas
			e^{\frac{t}{\eps^2}\tilde{\AA}_{\eps}(\xi)}P_{\eps}(\xi)V =&\frac1{2\pi i}\int^{\kappa+ i\infty}_{\kappa- i\infty}
			e^{ \frac{\lambda t}{\eps^2}}(\lambda-\tilde{\AA}_{\eps}(\xi))^{-1}P_{\eps}(\xi)Vd\lambda\\
			=&\frac1{2\pi i}\sum^7_{j=-1}\int^{\kappa+ i\infty}_{\kappa- i\infty}
			e^{ \frac{\lambda t}{\eps^2}}(\lambda-\lambda_j(|\xi|,\eps))^{-1}d\lambda \(V,\mathcal{U}^*_j(\xi,\eps)\)_{\xi} \mathcal{U}_j(\xi,\eps)\\
			=&\sum^7_{j=-1}e^{\frac{t}{\eps^2}\lambda_j(|\xi|,\eps)}\(V,\mathcal{U}^*_j(\xi,\eps) \)_{\xi} \mathcal{U}_j(\xi,\eps)=S_1(t,\xi,\eps)V   .
			\emas
			By \eqref{S1a}, we have
			\be
			S_3(t,\xi,\eps)=S_{31}(t,\xi,\eps)+S_{32}(t,\xi,\eps),\label{S22}
			\ee
			where
			\bmas
			S_{31}(t,\xi,\eps)&=S(t,\xi,\eps)1_{\{\eps|\xi|\le r_0\}}(I- P_{\eps}(\xi)),\\
			S_{32}(t,\xi,\eps)&=S_3(t,\xi,\eps)1_{\{\eps|\xi|\ge r_0\}}.
			\emas
			It holds that
			$$
			\|S_{3j}(t,\xi,\eps)V\|_{\xi}\le Ce^{-\frac{dt}{\eps^2}}\| V\|_{\xi},\quad \forall  V\in L^2(\R^3_v)\times \C^3_{\xi}\times \C^3_{\xi},\,\,\,j=1,2.
			$$
			
			Since $\Lambda_j(\xi)$, $j=-1,0,\cdots,7$ are the orthonormal basis of $N_0\times \C^3_{\xi}\times \C^3_{\xi}$, it follows that for any $V_0\in N_0\times \C^3_{\xi}\times \C^3_{\xi}$,
			$$
			V_0-P_{\eps}(\xi)V_0
			=\sum^7_{j=-1}(V_0,\Lambda_j(\xi))_\xi \Lambda_j(\xi)-\sum^7_{j=-1}(V_0,\mathcal{U}^*_j(\xi,\eps) )_\xi \mathcal{U}_j(\xi,\eps),
			$$
			which together with { $\|\mathcal{U}_j(\xi,\eps)-\Lambda_j(\xi)\|_{\xi}=O(\eps|\xi|)$ and  $S_{31}(t,\xi,\eps)=S_{31}(t,\xi,\eps)(I- P_{\eps}(\xi))$} gives rise to
			\be
			\|S_{31}(t,\xi,\eps)V_0\|_{\xi}\le C\eps|\xi|1_{\{\eps|\xi|\le r_0\}}e^{-\frac{dt}{\eps^2}}\|V_0\|_{\xi}.\label{S4}
			\ee
			By combining \eqref{S22}--\eqref{S4}, we obtain \eqref{S6}. This completes the proof of the lemma.
		\end{proof}

		Then, we have the first order fluid approximation of the semigroup $e^{\frac{t}{\eps^2}\AA_\eps}$ as follows.

		\begin{lem} \label{fl1}
			For any $\eps,\sigma\ll 1$,  any integer $k,m\ge 0$ and $U_0=(f_0,E_0,B_0)\in L^2\cap L^1 $, it holds that
			\bma
			&\quad\left\|\nabla_x^k\left(e^{\frac{t}{\eps^2}\AA_\eps}U_0-Y_1(t){P_A}U_0\right) \right\|_{L^{\infty}}\nnm\\
			&\le C \bigg( \eps(1+t)^{-1 } +\(1+\frac t \eps\)^{-\frac{3-\sigma}2}\bigg)(\|U_0\|_{H^{k+3} }+\|U_0\|_{L^{1} })\nnm\\
			&\quad +C\eps^m (1+t)^{-m}\|\Tdx^{m}U_0\|_{H^{k+2}}, \label{limit1}
			\ema
			where $Y_1(t)$ is defined by \eqref{solution1a}, and $C>0$ is a constant independent of $\eps$.
			Moreover, if  $U_0=(f_0,E_0,B_0)\in L^2\cap L^1 $ satisfies \eqref{initial}, then
			\bma
			\left\|\nabla_x^k\left(e^{\frac{t}{\eps^2}\AA_\eps}U_0-Y_1(t){ P_A}U_0\right) \right\|_{L^{\infty}}
			&\le C \eps (1+t)^{-1 } (\|U_0\|_{H^{k+3} }+\|U_0\|_{L^{1} })\nnm\\
			&\quad +C\eps^m (1+t)^{-m}\|\Tdx^{m}U_0\|_{H^{k+2}}. \label{limit1a}
			\ema
		\end{lem}

		\begin{proof}
			For simplicity, we only prove the case of $k=0$.
			By \eqref{S1} and taking $\eps\le r_0/2$ with $r_0>0$ given in Lemma \ref{eigen_4}, we have
			\bma
			\left\|e^{\frac{t}{\eps^2}\AA_\eps}U_0-Y_1(t){ P_A}U_0\right\|&= \left\|\intr \left(\mathcal{Q}(\xi)e^{\frac{t}{\eps^2}\tilde{\AA}_\eps(\xi)}\hat V_0-\mathcal{Q}(\xi)\tilde{Y}_1(t,\xi){ P_A}\hat V_0 \right)e^{ix\cdot \xi}d\xi\right\|\nnm\\
			&\le \bigg\|\int_{|\xi|\leq \frac{r_0}{\eps}} \left(\mathcal{Q}(\xi)S_1(t,\xi,\eps)\hat{V}_0-\mathcal{Q}(\xi)\tilde{Y}_1(t,\xi){ P_A}\hat V_0 \right)e^{ix\cdot \xi}d\xi\bigg\|\nnm\\
			&\quad+\int_{|\xi|>\frac{r_1}{\eps} } \|S_2(t,\xi,\eps)\hat{V}_0 \|_{\xi}d\xi+\int_{ |\xi|>\frac{r_0}{\eps}} \|\tilde{Y}_1(t,\xi)\hat{V}_0 \|_{\xi}d\xi \nnm\\
			&\quad +\intr  \|S_3(t,\xi,\eps)\hat{V}_0 \|_{\xi}d\xi\nnm\\
			&=:I_1+I_2+I_3+I_4, \label{S_4aa}
			\ema
			where $\hat V_0=(\hat f_0,\omega\times \hat E_0,\omega\times \hat B_0)$ and $\mathcal{Q}(\xi)$ is defined in \eqref{defQ}.
			We first deal with $I_1$. By Theorem \ref{rate1} and \eqref{specr0},
			\begin{align*}
				I_1&\leq \sum_{j=-1}^7\int_{|\xi|\le \frac{r_0}{\eps} }\bigg\| e^{\frac{\lambda_j(|\xi|,\eps)t}{\eps^2} }\big(\hat V_0, \mathcal{U}^*_j \big)_{\xi} \mathcal{U}_j
				-e^{\frac{\eta_j(|\xi|)t}{\eps}-b_j(|\xi|)t}\big(\hat{V}_0,\Lambda_j \big)_{\xi}\Lambda_j \bigg\|_\xi d\xi\\
				&\quad+\sum _{j\neq 0,2,3}\bigg\|\int_{|\xi|\leq \frac{r_0}{\eps}}  e^{\frac{\eta_j(|\xi|)t}{\eps}-b_j(|\xi|)t}\big(\hat V_0,{ \Lambda}_j(\xi) \big)_\xi \Q(\xi){\Lambda}_j(\xi)e^{ix\cdot \xi}d\xi\bigg\| \\
				&=:I_{11}+I_{12}.
			\end{align*}
			{ We first consider $I_{11}$.
				By \eqref{specr0} and \eqref{specr01}, there exists $c>0$ such that
				\begin{align}
					&\left|e^{\frac{\lambda_j(s,\eps)t}{\eps^2} }-e^{\frac{\eta_j(s)t}{\eps}-b_j(s)t}\right|\leq Ce^{-cs^2t}\eps s^3t,\ \ \ j=-1,0,1,4,5,6,7,\label{eigdif1}\\
					&\left|e^{\frac{\lambda_k(s,\eps)t}{\eps^2} }-e^{\frac{\eta_k(s)t}{\eps}-b_k(s)t}\right|\leq Ce^{-\frac{cs^4t}{1+s^2}}\frac{\eps s^5}{1+s^2}t,\ \ \ k=2,3.\label{eigdif2}
				\end{align}
			}
			By \eqref{eigf2}, \eqref{eigf22} and \eqref{lamj}, we obtain that
			\begin{equation}\label{uj-lam}
				\|  \mathcal{U}_j (\xi,\eps)-\Lambda_j(\xi) \|_\xi = O(\eps |\xi|).
			\end{equation}
			Then we have
			\begin{align*}
				I_{11}
				&\leq C\int _{|\xi|\leq \frac{r_0}{\eps}}  \left\{e^{-c|\xi|^2t}(\eps |\xi|+\eps |\xi|^3t)+e^{- \frac{c|\xi|^4}{1+|\xi|^2}t}\(\eps |\xi|+\frac{\eps |\xi|^5}{1+|\xi|^2}t\)\right\}  \| \hat U_0 \| d\xi\\
				&\leq C\eps  \sup_{\xi\in \mathbb{R}^3}\| \hat U_0 \|\(\int _{\mathbb{R}^3} e^{-c|\xi|^2t} ( |\xi|+ |\xi|^3t) d\xi+\int_{|\xi|\le 1}e^{-c |\xi|^4 t}( |\xi|+ |\xi|^5 t)d\xi\)\\
				&\leq C\eps (t^{-2}+t^{-1}) \| U_0\|_{L^1}.
			\end{align*}
			Here we used the estimate $\|\hat V_0(\xi)\|_\xi =\| \hat U_0(\xi)\| \leq \| U_0\|_{L^1}$.
			On the other hand, by the H\"{o}lder's inequality,
			\begin{align*}
				I_{11}& \leq C\eps \bigg(\int _{|\xi|\leq \frac{r_0}{\eps}}  \frac{|\xi|^2}{ (1+|\xi|^2)^{6}} d\xi\bigg)^{\frac{1}{2}} \bigg( \int_ {\mathbb R^3} (1+|\xi|^2)^{6} \| \hat U_0\| ^2 d\xi\bigg)^{\frac{1}{2}}\\
				&\leq  C\eps \|  U_0\|_{H^3}.
			\end{align*}
			Hence, we conclude that
			\be\label{i11}
			I_{11} \leq C\eps (1+t)^{-1} ( \| U_0\|_{L^1} +\| U_0\|_{H^3}  ).\ee
			Then we consider $I_{12}$ associated to the eigenvalues with imaginary leading term.
			We write
			\begin{align*}
				I_{12}&\leq \sum _{j\neq 0,2,3}\bigg\|\intr e^{\frac{\eta_j(|\xi|)t}{\eps}-b_j(|\xi|)t}\big(\hat V_0,{ \Lambda}_j(\xi) \big)_\xi \Q(\xi){\Lambda}_j(\xi)e^{ix\cdot \xi}d\xi\bigg\|_{L^\infty} \\
				&\quad	+ \sum _{j\neq 0,2,3}\bigg\|\int_{|\xi|\geq \frac{r_0}{\eps}}  e^{\frac{\eta_j(|\xi|)t}{\eps}-b_j(|\xi|)t}\big(\hat V_0,{ \Lambda}_j(\xi) \big)_\xi \Q(\xi){\Lambda}_j(\xi)e^{ix\cdot \xi}d\xi\bigg\|_{L^\infty} \\
				&=:I_{121}+I_{122}.
			\end{align*}
			We apply the dispersive estimate in Lemma \ref{lemdisp} to obtain decay in time.
			By Young's inequality, we get for any $0<\sigma\ll 1$,
			\bma
			I_{121}&\leq C\sum_{j\neq 0,2,3}	\Big\| \mathcal{F}^{-1}\Big(e^{\frac{\eta_j(|\xi|)t}{\eps}}(1+|\xi|^2)^{-\frac{5}{4}}\Big)\Big\|_{L^{\frac{6}{\sigma}}_x}	\left\|\mathcal{F}^{-1}\(e^{-b_j(|\xi|)t}\)\right\|_{L^1_x} \nnm\\
			&\qquad \times\left\|\mathcal{F}^{-1}\left\{\big(\hat V_0,{ \Lambda}_j(\xi) \big)_\xi \Q(\xi){\Lambda}_j(\xi)(1+|\xi|^2)^{\frac{5}{4}}\right\}\right\|_{L^{\frac{6}{6-\sigma}}}.\label{121}
			\ema
			By Lemma \ref{lemdisp}, we have
			\be \label{term1}
			\Big\| \mathcal{F}^{-1}\Big(e^{\frac{ \eta_j(|\xi|)}{\eps}t}(1+|\xi|^2)^{-\frac{5}{4}}\Big)\Big\|_{L^{\frac{6}{\sigma}}_x}\leq C\(\frac{t}{\eps}\)^{-\frac{3-\sigma}{2}}.
			\ee
			On the other hand, by Lemma \ref{kern} we obtain
			\be \label{term2}
			\left\|\mathcal{F}^{-1}\(e^{-b_j(|\xi|)t}\)\right\|_{L^1_x}\leq C.
			\ee
			Note that
			$$
			\big(\hat V_0,\Lambda_j\big)_\xi=\big(\hat U_0,\Q(\xi)\Lambda_j\big)-\frac{1}{|\xi|^2}\big((\hat{n}_0-i\xi\cdot \hat{E}_0)\chi_0,\tilde{u}_j\big)=\big(\hat U_0,\Q(\xi)\Lambda_j\big),
			$$
			where $n_0=(f_0,\chi_0)$, and $\tilde{u}_j$ is the first component of $\Lambda_j$, and we have used the compatibility condition \eqref{com}. Then we can write
			$$
			\big(\hat V_0,{ \Lambda}_j(\xi) \big)_\xi \Q(\xi){\Lambda}_j(\xi)=\big(\hat U_0,\Q(\xi)\Lambda_j(\xi)\big) \Q(\xi){\Lambda}_j(\xi).
			$$
			By the definition of $\Lambda_j(\xi)$ in \eqref{lamj}, we have
			\be\label{multi1}
			|\nabla_\xi^m(\Q(\xi)\Lambda_j(\xi) )|\leq C_m |\xi|^{-m},\quad \forall m\in\mathbb{N}.
			\ee
			And it follows from the H\"{o}rmander-Mihlin multiplier Theorem (see Theorem \ref{thmhm}) and the Sobolev interpolation inequality (see Lemma \ref{lemitp}) that
			\bma
			\left\|\mathcal{F}^{-1}\left\{\big(\hat V_0,{ \Lambda}_j(\xi) \big)_\xi \Q(\xi){\Lambda}_j(\xi)(1+|\xi|^2)^{\frac{5}{4}}\right\}\right\|_{L^{\frac{6}{6-\sigma}}}&\leq C \|U_0\|_{W^{\frac{5}{2},\frac{6}{6-\sigma}}} \nnm\\
			&\leq C (\|U_0\|_{L^1}+\|U_0\|_{H^3}).\label{term3}
			\ema
			We conclude from \eqref{121}, \eqref{term1}, \eqref{term2} and \eqref{term3} that
			$$
			I_{121}\leq C \(\frac{t}{\eps}\)^{-\frac{3-\sigma}{2}}(\|U_0\|_{L^1}+\|U_0\|_{H^3}).
			$$
			For $I_{122}$, we have
			\begin{align*}
				I_{122}&\leq C \int _{|\xi|\geq \frac{r_0}{\eps}}e^{-c|\xi|^2t}\|U_0\|d\xi\\
				&\leq C \bigg(\int _{|\xi|\geq \frac{r_0}{\eps}}e^{-\frac{2cr_0^2t}{\eps^2}}(1+|\xi|^2)^{-2}d\xi\bigg)^\frac{1}{2} \bigg(\int _{|\xi|\geq \frac{r_0}{\eps}}(1+|\xi|^2)^2\|\hat U_0\|^2d\xi\bigg)^\frac{1}{2}\\
				&\leq C e^{-\frac{cr_0^2t}{\eps^2}}\|U_0\|_{H^2}.
			\end{align*}
			Hence, we obtain
			\begin{align}\label{i12i1}
				I_{12}\leq I_{121}+I_{122}\leq C\bigg(\(\frac{t}{\eps}\)^{-\frac{3-\sigma}{2}}+e^{-\frac{cr_0^2t}{\eps^2}}\bigg)(\|U_0\|_{L^1}+\|U_0\|_{H^3}). \end{align}
			On the other hand, we have
			\be\label{i12i2}
			I_{12}\leq C \|(1+|\xi|^2)^{-1} e^{\frac{t}{\eps^2}\lambda_j(|\xi|,\eps)} \|_{L^2_\xi}\|U_0\|_{H^2}\leq C\|U_0\|_{H^2}.
			\ee
			Combining \eqref{i12i1} and \eqref{i12i2}, we get
			\be\label{i12}
			I_{12}\leq C  \(1+\frac t \eps\)^{-\frac{3-\sigma}2} (\|U_0\|_{L^1}+\|U_0\|_{H^3}).
			\ee
			From \eqref{S2} and Lemma \ref{eigen_5},  we have
			\bma
			I_2&\leq C \int_{|\xi|\geq \frac{r_1}{\eps}}e^{-\frac{ct}{\eps |\xi|}}\|\hat V_0\|_\xi d\xi \nnm\\
			&\leq C \sup_{|\xi|\geq \frac{r_1}{\eps}} |\xi|^{-m}e^{-\frac{ct}{\eps |\xi|}} \left(\int_{|\xi|\geq \frac{r_1}{\eps}}|\xi|^{2m}(1+|\xi|^2)^2\|\hat V_0\|_\xi^2 d\xi\right)^\frac{1}{2}\nnm\\
			&\leq C \eps^m (1+t)^{-m}\|\Tdx^mU_0\|_{H^2}.\label{i2}
			\ema
			For $I_3$ and $I_4$, by  \eqref{v1} and Theorem \ref{rate1}, it holds that
			\be\label{i34}
			I_3 \leq C \int_{ |\xi|> \frac{r_0}{\eps}}e^{-c|\xi|^2t}\|\hat V_0\|_\xi d\xi\leq C\frac{\eps}{r_0} e^{-\frac{cr_0^2t}{\eps^2}}\|U_0\|_{H^3},
			\ee
			and
			\be\label{i35}
			I_4 \leq C \intr e^{-\frac{dt}{\eps^2}}\|\hat{V}_0\|_{\xi} d\xi\leq C e^{-\frac{dt}{\eps^2}}\|U_0\|_{H^2}.
			\ee
			Then we obtain \eqref{limit1} in view of  \eqref{i11}, \eqref{i12}, \eqref{i2}, \eqref{i34} and \eqref{i35}.
			
			Now, we consider \eqref{limit1a}.  If $U_0=(f_{0}, E_0,B_0)$ satisfies \eqref{initial}, by \eqref{hj}, one has
			$$
			\sum_{j\neq 0,2,3} \big(\hat V_0,\Lambda_j(\xi)\big)_\xi =0.
			$$
			Hence,  we obtain
			$$
			I_{12}=0.
			$$
			Since $f_0\in N_0,$ by Lemma \ref{S2a} we have
			\bma
			I_4&\le C\int_{|\xi|\le \frac{r_0}{\eps}}\eps |\xi|e^{-\frac{dt}{\eps^2}}\| \hat{V}_0\|_{\xi}d\xi+C\int_{|\xi|\ge \frac{r_0}{\eps}}e^{-\frac{dt}{\eps^2}}\| \hat{V}_0\|_{\xi}d\xi \nnm\\
			&\le C\eps e^{-\frac{dt}{\eps^2}}\(\int_{|\xi|\le \frac{r_0}{\eps}} |\xi| \| \hat{U}_0\| d\xi +\int_{|\xi|\ge \frac{r_0}{\eps}} |\xi|\| \hat{U}_0\| d\xi\)\nnm\\
			&\le C\eps e^{-\frac{dt}{\eps^2}}\|U_0\|_{H^3 }.\label{S_7a-1}
			\ema
			Then we get \eqref{limit1a}. This completes the proof of the lemma.
		\end{proof}
		
		\begin{remark}
			 From Lemma \ref{fl1}, we have
			\be
			\|e^{\frac{t}{\eps^2}\AA_{\eps}}P_AU_0-Y_1(t)P_AU_0-U^{osc}_{\eps}(t)\|_{L^{\infty} } \le C\eps (1+t)^{-1}\(\|U_0\|_{H^3}+\|U_0\|_{L^{1}}\), \label{limit7}
			\ee
			where $U^{osc}_{\eps}(t)=U^{osc}_{\eps}(t,x,v)$ is the high oscillation term defined by
			\be U^{osc}_{\eps}(t,x,v)=\sum_{j\ne 0,2,3} \mathcal{F }^{-1}\Big(e^{\frac{\eta_j(|\xi|)t}{\eps}-b_j(|\xi|)t}\big(\hat V_0,{ \Lambda}_j(\xi) \big)_\xi \Q(\xi){\Lambda}_j(\xi)\Big). \label{osc}\ee
		\end{remark}
		
		We have the second order fluid approximation of the semigroup $e^{\frac{t}{\eps^2}\AA_\eps}$ as follows.
		
		\begin{lem}\label{fl2}
			For any $\eps,\sigma\ll 1$, any integer $k,m\ge 0$ and  $U_0=(f_0,0,0)\in H^{k+4}\cap L^1 $ satisfying $P_0f_0=0$, we have
			\bma
			&\quad\bigg\|\nabla_x^k\left(\frac1{\eps}e^{\frac{t}{\eps^2}\AA_\eps}U_0-Y_1(t)Z_0\right)\bigg\|_{L^{\infty}} \nnm\\
			&\le C \bigg( \eps(1+t)^{-\frac54 }+ { \frac1{\eps}}e^{-\frac{dt}{\eps^2}}+\(1+\frac{t}{\eps}\)^{-\frac{3-\sigma}{2}}\bigg)(\|U_0\|_{H^{k+4} }+\|U_0\|_{L^{1} })\nnm\\
			&\quad +C \eps^{m}(1+t)^{-m}\|\Tdx^{m}U_0\|_{H^{k+3}}, \label{limit2}
			\ema
			where $Y_1(t)$ is defined by \eqref{solution1a}, $Z_0=( P_0(v\cdot\Tdx L^{-1}f_0), 0,0),$ and $ d, C>0$ are constants independent of $\eps$.
		\end{lem}
		
		\begin{proof}
			Again, we only prove the case when $k=0$.
			Similar to \eqref{S_4aa}, by \eqref{S1} and taking $\eps\le r_0/2$ with $r_0>0$ given in Lemma \ref{eigen_4}, we have for $U_0=(f_0,0,0), $
			\bma
			\left\|\frac1{\eps} e^{\frac{t}{\eps^2}\AA_\eps}U_0-Y_1(t)Z_0\right\| &\le \bigg\|\int_{|\xi|\leq\frac{r_0}{\eps}}\left(\frac1{\eps}\mathcal{Q}(\xi)S_1(t,\xi,\eps)\hat U_0-\mathcal{Q}(\xi)\tilde{Y}_1(t,\xi)\hat Z_0 \right)e^{ix\cdot \xi}d\xi\bigg\| \nnm\\
			&\quad+\frac1{\eps}\int_{|\xi|>\frac{r_1}{\eps} } \|S_2(t,\xi,\eps)\hat{U}_0 \|_{\xi}d\xi+ \int_{ |\xi|>\frac{r_0}{\eps}} \|\tilde{Y}_1(t,\xi)\hat{Z}_0 \|_{\xi}d\xi \nnm\\
			&\quad +\frac1{\eps}\intr \|S_3(t,\xi,\eps)\hat{U}_0 \|_{\xi}d\xi\nnm\\
			&=:J_1+J_2+J_3+J_4.\label{S_4ab}
			\ema
			We first deal with $J_1$. Note that $P_df_0=P_d(v\cdot \nabla_x L^{-1}f_0)=0$, hence
			$$ \big(\hat U_0, \mathcal{U}^*_j(\xi,\eps) \big)_\xi=\big(\hat{f}_0,P_1\overline{u_j(\xi,\eps)}\big), \quad \big(\hat{Z}_0,\Lambda_j(\xi) \big)_\xi=\big(iv\cdot\xi L^{-1}\hat{f}_0,\tilde{u}_j(\xi) \big),$$
			{ where $u_j$ and $\tilde{u}_j$ are the first components of the eigenfunctions $\mathcal{U}_j$ and $\Lambda_j$ respectively.}
			By \eqref{S1}, \eqref{v1} and \eqref{normeq}, we obtain{
				\begin{align*}
					J_1&\leq  \sum_{j=-1}^7\int_{|\xi|\le \frac{r_0}{\eps} }\bigg\|\frac1{\eps} e^{\frac{\lambda_j(|\xi|,\eps)t}{\eps^2} }\big(\hat{f}_0,P_1 \overline{u_j} \big) \mathcal{U}_j
					-e^{\frac{\eta_j(|\xi|)t}{\eps}-b_j(|\xi|)t}\big(iv\cdot\xi L^{-1}\hat{f}_0,\tilde{u}_j  \big)\Lambda_j \bigg\|_\xi d\xi\\
					&\quad+\sum _{j\neq 0,2,3}\bigg\|\int_{|\xi|\leq \frac{r_0}{\eps}} e^{\frac{\eta_j(|\xi|)t}{\eps}-b_j(|\xi|)t}\big(iv\cdot\xi L^{-1}\hat{f}_0,\tilde{u}_j  \big)\mathcal{Q}(\xi) {\Lambda}_j(\xi)e^{ix\cdot \xi}d\xi\bigg\|_{L^\infty}\\
					&=:J_{11}+J_{12}.
			\end{align*}}
			We first estimate $J_{11}$. 	We further decompose{
				\begin{align*}
					J_{11}=&\sum_{j=-1}^7\int_{|\xi|\le \frac{r_0}{\eps} }\frac1{\eps}\bigg|\Big( e^{\frac{\lambda_j(|\xi|,\eps)t}{\eps^2} }-e^{\frac{\eta_j(|\xi|)t}{\eps}-b_j(|\xi|)t}\Big)\big(\hat{f}_0,P_1\overline{u_j }\big)
					\bigg| d\xi\\
					&+\sum_{j=-1}^7\int_{|\xi|\le \frac{r_0}{\eps} }\bigg\|e^{\frac{\eta_j(|\xi|)t}{\eps}-b_j(|\xi|)t}\(\frac1{\eps} \big(\hat{f}_0,P_1\overline{u_j }\big) \mathcal{U}_j
					-\big(iv\cdot\xi L^{-1}\hat{f}_0,\tilde{u}_j  \big)\Lambda_j\)\bigg\|_\xi d\xi.
			\end{align*}}
			By Lemma \ref{eigen_4a}, we have $\|P_1u_j\|\leq O(\eps s)$. Hence, for any $U_0=(f_0,0,0)$ with $P_0f_0=0$, it holds
			\be\label{u00}
			\big|	\big(\hat{f}_0,P_1\overline{u_j(\xi,\eps)}\big)\big|\leq O(\eps s)\|\hat U_0\|.
			\ee
			By direct calculation, one has
			\be \label{uj-lam2}
			\left|\frac1\eps\big(\hat{f}_0,P_1\overline{u_j(\xi,\eps)}\big)-\big(iv\cdot\xi L^{-1}\hat{f}_0,\tilde{u}_j(\xi)  \big) \right|\leq O(\eps s^2)\|\hat U_0\|.
			\ee
			Then from \eqref{eigdif1}, \eqref{eigdif2}, \eqref{uj-lam}, \eqref{u00} and \eqref{uj-lam2}, we obtain
			\begin{align}\label{j11}
				J_{11}&\leq   C\int _{|\xi|\leq \frac{r_0}{\eps}}  \left\{e^{-c|\xi|^2t}(\eps |\xi|^2+\eps |\xi|^4t)+e^{- \frac{c|\xi|^4}{1+|\xi|^2}t}\(\eps |\xi|^2+\frac{\eps |\xi|^6}{1+|\xi|^2}t\)\right\}  \| \hat U_0 \| d\xi\nnm\\
				&\leq  C\eps  \sup_{\xi\in \mathbb{R}^3}\| \hat U_0 \|\(\int _{\mathbb{R}^3} e^{-c|\xi|^2t} ( |\xi|^2+ |\xi|^4t) d\xi+\int_{|\xi|\le 1}e^{-c |\xi|^4 t}( |\xi|^2+ |\xi|^6 t)d\xi\)\nnm\\
				&\leq C \eps (1+t)^{-\frac{5}{4}}(\|U_0\|_{L^1}+\|U_0\|_{H^4}).
			\end{align}
			For $J_{12}$, similar to \eqref{i12}, we get
			\be\label{j12}
			J_{12}\leq C\(1+\frac t \eps\)^{-\frac{3-\sigma}2} (\|U_0\|_{L^1}+\|U_0\|_{H^4}).
			\ee
			Then we consider $J_2$. { By \eqref{S2} and Lemma \ref{eigen_4}, it holds for $U_0=(f_0,0,0)$ with $P_{0}f_0=0$ that
				$$
				S_2(t,\xi,\eps)\hat{U}_0=\sum^4_{k=1}e^{\frac{t}{\eps^2}\beta_k(|\xi|,\eps)} \big(\hat{f}_0,  \overline{w_k(\xi,\eps)}  \big) \mathcal{V}_k(\xi,\eps) , \quad \eps|\xi|\ge r_1,
				$$
				which together with $\|w_k(\xi,\eps)\|^2= O(\eps |\xi|^{-1})$ gives
				\bma
				J_2&\le C\int_{ |\xi|\ge \frac{r_1}{\eps}} \frac{1}{\sqrt{\eps|\xi|}}e^{-\frac{ct}{\eps|\xi|}}\|\hat{U}_0\|  d\xi\nnm\\
				&\le C\frac{1}{\sqrt{r_1}}\sup_{|\xi|\ge \frac{r_1}{\eps}}\frac1{|\xi|^{m}}e^{-\frac{c t}{\eps|\xi|}}\(\int_{ |\xi|\ge \frac{r_1}{\eps}} |\xi|^{2m}(1+|\xi|^{2})^2\|\hat{U}_0\|^2  d\xi\)^{\frac12} \nnm\\
				&\le C \eps^{m}(1+t)^{-m}\|\Tdx^mU_0\|_{H^2 } . \label{j2}
				\ema}
			For $J_3$, we estimate it as follows.
			\begin{align}\label{j3}
				J_3&\leq C\int_{|\xi|\geq \frac{r_0}{\eps}}e^{-c|\xi|^2t}\|P_0(v\cdot \xi )L^{-1}\hat f_0\|d\xi\nnm\\
				&\leq C \int_{|\xi|\geq \frac{r_0}{\eps}}e^{-\frac{cr_0^2}{\eps^2}t}\|\hat f_0\| |\xi|^2d\xi\leq C\frac{\eps}{r_0} e^{-\frac{cr_0^2}{\eps^2}t}\|U_0\|_{H^4}.
			\end{align}
			Finally, by Lemma  \ref{S2a},
			\be\label{j4}
			J_4 \leq \frac{C}{\eps} \intr e^{-\frac{dt}{\eps^2}}\|\hat{U}_0\|  d\xi\leq \frac{C}{\eps}  e^{-\frac{dt}{\eps^2}}\|U_0\|_{H^2}.
			\ee
			By \eqref{j11}, \eqref{j12}, \eqref{j2}, \eqref{j3} and \eqref{j4}, we obtain \eqref{limit2}. This completes the proof of the lemma
		\end{proof}

		In the following lemma, we will present the time decay rates of the semigroup $e^{\frac{t}{\eps^2}\mathbb{A}_\eps}$.
		\begin{lem}\label{time}
			For any $\eps\ll 1$, $\alpha\in\mathbb{N}^3$, and integer $k\geq 0$ and $U_0=(f_0,E_0,B_0)$, the solution $U_\eps=(f_\eps,E_\eps,B_\eps)=e^{\frac{t}{\eps^2}\mathbb{A}_\eps}U_0$ to the linear system \eqref{LVMB1} satisfies that
			\begin{align}
				\|\partial_x^\alpha { P_A} e^{\frac{t}{\eps^2}\mathbb{A}_\eps} U_0\|_{L^2}&\leq C(1+t)^{-\frac38-\frac{m}{4}}(\|\partial_x^{\alpha'}U_0\|_{L^1}+\|\partial_x^\alpha U_0\|_{L^2})\nnm\\
				&\quad+C\eps^k(1+t)^{-k}\|\nabla_x^{|\alpha|+k} U_0\|_{L^2},\label{p2u}\\
				\|\partial_x^\alpha { P_B} e^{\frac{t}{\eps^2}\mathbb{A}_\eps} U_0\|_{L^2}&\leq C\(\eps(1+t)^{-\frac58-\frac{m}{4}}+e^{-\frac{dt}{\eps^2}}\)(\|\partial_x^{\alpha'}U_0\|_{L^1}+\|\partial_x^\alpha U_0\|_{ H^1})\nnm\\
				&\quad+C\eps^{k+1}(1+t)^{-k}\|\nabla_x^{|\alpha|+k} U_0\|_{L^2},\label{p3u}
			\end{align}
			where ${ P_A}, { P_B}$ are defined by \eqref{defp23}, $\alpha'\leq \alpha$, $m=|\alpha-\alpha'|$, and $d>0$ is the cosntant in \eqref{S6}.
			
			If $U_0=(f_0,E_0,0)$, then{
				\begin{align}
					\|\partial_x^\alpha { P_A} e^{\frac{t}{\eps^2}\mathbb{A}_\eps} U_0\|_{L^2}	&\leq C(1+t)^{-\frac58-\frac{m}{4}}(\|\partial_x^{\alpha'}U_0\|_{L^1}+\|\partial_x^\alpha U_0\|_{L^2})\nnm\\
					&\quad+C\eps^k(1+t)^{-k}\|\nabla_x^{|\alpha|+k} U_0\|_{L^2},\label{p2u0}\\
					\|\partial_x^\alpha { P_B} e^{\frac{t}{\eps^2}\mathbb{A}_\eps} U_0\|_{L^2}&\leq C\(\eps(1+t)^{-\frac78-\frac{m}{4}}+e^{-\frac{dt}{\eps^2}}\)(\|\partial_x^{\alpha'}U_0\|_{L^1}+\|\partial_x^\alpha U_0\|_{ H^1})\nnm\\
					&\quad+C\eps^{k+1}(1+t)^{-k}\|\nabla_x^{|\alpha|+k} U_0\|_{L^2}.\label{p3u0}
			\end{align}}
			Moreover, if $U_0=(f_0,0,0)$ with $P_0f_0=0$, then
			\begin{align}
				\|\partial_x^\alpha { P_A} e^{\frac{t}{\eps^2}\mathbb{A}_\eps} U_0\|_{L^2}&\leq C\(\eps(1+t)^{ -\frac58-\frac{m}{4}}+e^{-\frac{dt}{\eps^2}}\)(\|\partial_x^{\alpha'}U_0\|_{L^1}+\|\dxa U_0\|_{H^1})\nnm\\
				&\quad+C\eps^{k+1}(1+t)^{-k}\|\nabla_x^{|\alpha|+k} U_0\|_{L^2},\label{p2u1}\\
				\|\partial_x^\alpha { P_B} e^{\frac{t}{\eps^2}\mathbb{A}_\eps} U_0\|_{L^2}&\leq C\(\eps^2(1+t)^{ -\frac78-\frac{m}{4}}+e^{-\frac{dt}{\eps^2}}\)(\|\partial_x^{\alpha'}U_0\|_{L^1}+\|\dxa U_0\|_{H^2})\nnm\\
				&\quad+C\eps^{k+2}(1+t)^{-k}\|\nabla_x^{|\alpha|+k} U_0\|_{L^2}.\label{p3u1}
			\end{align}
		\end{lem}

		\begin{proof}
			By Theorem \ref{rate1}, we have 	for $\star\in\{A,B\}$,
			\bma\label{l2}
			\|	\partial_x ^\alpha P_\star e^{\frac{t}{\eps^2}\mathbb{A}_\eps}U_0\|_{L^2}^2&= \int_{\mathbb{R}^3}\|\xi^\alpha P_\star e^{\frac{t}{\eps^2}\mathbb{A}_\eps}\hat V_0\|_\xi^2d\xi\nnm\\
			&\leq \int_{|\xi|\leq \frac{r_0}{\eps}} \|\xi^\alpha P_\star S_1(t,\xi,\eps)\hat V_0\|_\xi^2d\xi+ \int_{|\xi|\geq \frac{r_1}{\eps}} \|\xi^\alpha P_\star S_2(t,\xi,\eps)\hat V_0\|_\xi^2d\xi\nnm\\
			&\quad  +\int_{|\xi|\leq \frac{r_0}{\eps}} \|\xi^\alpha S_3(t,\xi,\eps)\hat V_0\|_\xi^2d\xi,
			\ema
			where $\hat V_0=(\hat f_0,\omega\times \hat E_0,\omega\times \hat B_0)$.
			We estimate the right hand side of \eqref{l2} term by term. Note that
			\bma
			S_1(t,\xi,\eps) \hat V_0
			&=\sum_{j=-1,0,1}e^{-\frac{\eta_jt}{\eps}-b_jt+O(\eps s^3t)}\bigg\{\bigg[(\hat{f}_0,h_j)+\frac{i \hat{E}_0\cdot\omega}{|\xi|}(h_j,\chi_0)\bigg](h_j,0,0)+O(\eps s)\bigg\} \nnm\\
			&\quad+\sum_{j=2,3}e^{ -b_jt+ O(\frac{\eps s^5}{1+s^2}t)}\frac1{1+s^2} \nnm\\
			&\qquad\times \left\{[s(\hat{m}_0\cdot W^j)-i(\omega\times \hat{B}_0\cdot W^j)](sv\cdot W^j\chi_0,0,-iW^j)+O(\eps s)\right\} \nnm\\
			&\quad+\sum_{j=4,5,6,7}e^{-\frac{\eta_jt}{\eps} -b_jt+O(\eps s^3t)}\frac1{2(1+s^2)} \nnm\\
			&\qquad\times\left\{[\hat{m}_0\cdot W^j+\eta_j(\hat{E}_0\cdot W^j)+is(\omega\times \hat{B}_0\cdot W^j)](v\cdot W^j\chi_0,\eta_jW^j,is W^j)+O(\eps s)\right\}, \label{l3}
			\ema
			where
			$\hat{m}_0=(\hat{f}_0,v\chi_0)$, $h_j$, $j=-1,0,1$ are defined in \eqref{hj}, and  $W^j,$ $j=2,3,4,5,6,7$ are defined in Lemma \ref{eigen_4a}.
			Thus
			\bma
			&\quad\int_{|\xi|\leq \frac{r_0}{\eps}} \|\xi^\alpha{P_A} S_1(t,\xi,\eps)\hat V_0 \|_\xi^2d\xi \nnm\\
			&\leq  C	\int_{|\xi|\leq r_0}e^{-c_2|\xi|^4 t}|\xi|^{2|\alpha|}\|\hat V_0\|_\xi^2d\xi
			+C	\int_{|\xi|\geq r_0}e^{-c_1|\xi|^2 t}|\xi|^{2|\alpha|}\|\hat V_0\|_\xi^2d\xi \nnm\\
			&\leq C\sup_{|\xi|\leq r_0}\|\xi^{\alpha'}\hat U_0\|^2\int _{|\xi|\leq r_0}e^{-c_2|\xi|^4 t}  |\xi|^{2|\alpha-\alpha'|}d\xi+C e^{-c_1 t}\int_{\mathbb{R}^3} |\xi|^{2|\alpha|}\|\hat U_0\|^2 d\xi \nnm\\
			&	\leq C(1+t )^{-\frac{3}{4}-\frac{m}{2}}(\|\partial_x^{\alpha'}U_0\|_{L^1}^2+\|\partial_x^\alpha U_0\|_{L^2}^2),\label{p2s1}
			\ema
			where $\alpha'\leq \alpha$, $m=|\alpha-\alpha'|$, and $c_1,c_2>0$ are some genetic constants. Similarly,
			\begin{equation}\label{p3s1}
				\int_{|\xi|\leq \frac{r_0}{\eps}} \|\xi^\alpha {P_B} S_1(t,\xi,\eps)\hat V_0 \|_\xi^2d\xi\leq  C\eps^2(1+t )^{-\frac{5}{4}-\frac{m}{2}}(\|\partial_x^{\alpha'}U_0\|_{L^1}^2+\|\partial_x^\alpha U_0\|_{H^1}^2).
			\end{equation}
			Moreover, by \eqref{S2} and \eqref{specr1a} we have
			\bma
			\int_{|\xi|\geq \frac{r_1}{\eps}} \|\xi^\alpha { P_A} S_2(t,\xi,\eps)\hat V_0\|_\xi^2d\xi&\leq C	\int_{|\xi|\geq \frac{r_1}{\eps}} |\xi|^{2|\alpha|}e^{-\frac{ct}{\eps|\xi|}}\|\hat V_0\|_\xi^2d\xi \nnm\\
			&\leq C\sup_{|\xi|\geq \frac{r_1}{\eps}}|\xi|^{-2m}e^{-\frac{ct}{\eps|\xi|}}\int_{|\xi|\geq \frac{r_1}{\eps}}|\xi|^{2m+2|\alpha|}\|\hat U_0\|^2d\xi \nnm\\
			&\leq C\eps^{2m}(1+t)^{-2m}\|\nabla_x^{m+|\alpha|} U_0\|_{L^2}^2.\label{p2s2}
			\ema
			Similarly, one has
			\begin{equation}\label{p3s2}
				\int_{|\xi|\geq \frac{r_1}{\eps}} \|\xi^\alpha { P_B} S_2(t,\xi,\eps)\hat V_0 \|_\xi^2d\xi\leq C\eps^{2m+2}(1+t)^{-2m}\|\nabla_x^{m+|\alpha|} U_0\|_{L^2}^2,
			\end{equation}
			in view of the fact that $\|{ P_B}\mathcal{V}_k\|\leq C\eps$ for $\mathcal{V}_k$ defined in Lemma \ref{eigen_4}.
			
			Finally, by noting that $\|\hat V_0\|_\xi^2=\|\hat U_0\|^2$, we can apply \eqref{S3} to estimate the last term on the right hand side of \eqref{l2} as follows:
			\be\label{s3}
			\int_{|\xi|\leq \frac{r_0}{\eps}} \|\xi^\alpha S_3(t,\xi,\eps)\hat V_0\|_\xi^2d\xi\leq C e^{-\frac{2dt}{\eps^2}}\int _{\mathbb{R}^3}|\xi|^{2|\alpha|}\|\hat V_0\|_\xi^2d\xi\leq Ce^{-\frac{2dt}{\eps^2}}\|\partial_x ^\alpha U_0\|_{L^2}.
			\ee
			Then we obtain \eqref{p2u} from \eqref{p2s1}, \eqref{p2s2}, \eqref{s3}, and obtain \eqref{p3u} from \eqref{p3s1}, \eqref{p3s2} and \eqref{s3}. The estimates  \eqref{p2u0} and \eqref{p3u0} can be obtained similarly.
			
			When $U_0=(f_0,0,0)$ with $P_0f_0=0$, we have{
				\bma
				S_1(t,\xi,\eps) \hat V_0
				&=\sum_{j=-1,0,1}e^{-\frac{\eta_jt}{\eps}-b_jt+O(\eps s^3t)}\left\{ i\eps s(\hat{f}_0,L^{-1}P_1(v\cdot\omega)h_j) (h_j,0,0)+O(\eps^2 s^2)\right\} \nnm\\
				&\quad+\sum_{j=2,3}e^{ -b_jt+ O(\frac{\eps s^5}{1+s^2}t)}\frac1{1+s^2} \nnm\\
				&\qquad\times \left\{i\eps s^2(\hat{f}_0,L^{-1}P_1(v\cdot\omega)(v\cdot W^j)\chi_0)(sv\cdot W^j\chi_0,0,-iW^j)+O(\eps^2 s^3)\right\} \nnm\\
				&\quad+\sum_{j=4,5,6,7}e^{-\frac{\eta_jt}{\eps} -b_jt+O(\eps s^3t)}\frac1{2(1+s^2)} \nnm\\
				&\qquad\times\left\{i\eps s(\hat{f}_0,L^{-1}P_1(v\cdot\omega)(v\cdot W^j)\chi_0)(v\cdot W^j\chi_0,\eta_jW^j,is W^j)+O(\eps^2 s^2)\right\}.
				\ema
			}
			Then we obtain
			\begin{align*}
				&\quad\int_{|\xi|\leq \frac{r_0}{\eps}}\|\xi^\alpha P_AS_1(t,\xi,\eps)\hat V_0\|_\xi^2d\xi \nnm\\
				&\leq  C\eps^2	\int_{|\xi|\leq r_0}e^{-c_2|\xi|^4 t}|\xi|^{2|\alpha|+2}\|\hat V_0\|_\xi^2d\xi
				+C	\eps^2\int_{|\xi|\geq r_0}e^{-c_1|\xi|^2 t}|\xi|^{2|\alpha|+2}\|\hat V_0\|_\xi^2d\xi \nnm
				\\
				&\leq C\eps^2 (1+t )^{-\frac{5}{4}-\frac{m}{2}} (\|\partial_x^{\alpha'}U_0\|_{L^1}^2+\|\partial_x^\alpha U_0\|_{H^1}^2 ).
			\end{align*}
			Similarly,
			$$
			\int_{|\xi|\leq \frac{r_0}{\eps}}\|\xi^\alpha P_BS_1(t,\xi,\eps)\hat V_0\|_\xi^2d\xi
			\leq C\eps^4(1+t )^{-\frac{7}{4}-\frac{m}{2}} (\|\partial_x^{\alpha'}U_0\|_{L^1}^2+\|\partial_x^\alpha U_0\|_{H^2}^2 ).
			$$
			Thus, we obtain \eqref{p2u1}, \eqref{p3u1} and complete the proof of the lemma.
		\end{proof}

		\begin{lem}
			\label{timev} For any  $q\in[1,2]$,  $\alpha\in\mathbb{N}^3$, and any $U_0=(u_0,E_0,B_0)$ with $u_0=n\chi_0+m\cdot v\chi_0+q\chi_4$ satisfying \eqref{NSP_5i}, then
			\be \label{Y1UL21}
			\|\dxa Y_1(t)U_0\|_{L^2}\leq C(1+t )^{-\frac{3}{4}(\frac{1}{q}-\frac{1}{2})-\frac{m}{4}}\|\partial_x^{\alpha'}U_0\|_{L^q}+Ct^{-\frac{k}{2}}e^{-ct}\|\partial_x^{\alpha''} U_0\|_{L^2},
			\ee
			where $\alpha',\alpha''\leq \alpha$, $m=|\alpha-\alpha'|$, and $k=|\alpha-\alpha''|$.
			
			If $U_0=(u_0,E_0,0)$, then
			\be\label{Y1UL2}
			\|\dxa Y_1(t)U_0\|_{L^2}\leq C(1+t )^{-\frac{5}{4}(\frac{1}{q}-\frac{1}{2})-\frac{m}{4}}\|\partial_x^{\alpha'}U_0\|_{L^q}+Ct^{-\frac{k}{2}}e^{-ct}\|\partial_x^{\alpha''} U_0\|_{L^2},
			\ee
			and for any $r_1,r_2\in(1,\infty)$,
			\be \label{de1}
			\|\dxa Y_1(t)U_0\|_{L^\infty}\leq  C (1+t)^{-\frac{3}{4r_1}-\frac{m+1}{4}}\|\partial_x^{\alpha'}U_0\|_{L^{r_1}}+Ct^{-\frac{3}{2r_2}-\frac{k}{2}}e^{-ct}\|\partial_x^{\alpha''}U_0\|_{L^{r_2}}.
			\ee
		\end{lem}
		
		\begin{proof}
			Note that
			\bma
			\tilde{Y}_1(t,\xi)\hat{V}_0&= e^{ -b_0(|\xi|)t }\bigg[(\hat{u}_0,h_0)+\frac{i \hat{E}_0\cdot\omega}{|\xi|}(h_0,\chi_0)\bigg](h_0,0,0) \nnm\\
			&\quad+\sum_{j=2,3}e^{ -b_j(|\xi|)t }\frac1{1+|\xi|^2} [|\xi|(\hat{m}_0\cdot W^j)-i(\omega\times \hat{B}_0\cdot W^j)](|\xi|v\cdot W^j\chi_0,0,-iW^j) ,
			\ema where  $h_0$ is defined in \eqref{hj}, and  $W^j,$ $j=2,3 $ are defined in Lemma \ref{eigen_4a}.
			By \eqref{solution1a} and Hausdorff-Young inequality, we obtain
			\bmas
			\|\dxa Y_1(t)U_0\|_{L^2}^2&=\int _{\mathbb{R}^3}\|\xi^\alpha\tilde Y_1(t,\xi) \hat V_0 \|_\xi ^2d\xi\\
			&\leq C\left(\int _{|\xi|\leq 1}+\int _{|\xi|\geq 1}\right)|\xi|^{2|\alpha|}(e^{-2c|\xi|^2t}+e^{-2c|\xi|^4t})\|\hat V_0 \|_\xi ^2d\xi\\
			&\leq C\left(\int_{|\xi|\leq 1}|\xi|^{2pm}e^{-2cp|\xi|^4t}d\xi\right)^\frac{1}{p}\left(\int_{|\xi|\leq 1}|\xi|^{2p'|\alpha'|}\|\hat U_0\|^{2p'}d\xi\right)^\frac{1}{p'}\\
			&\ \ + C\sup_{|\xi|\geq 1}(|\xi|^{2k}e^{-2c|\xi|^2t})\int_{|\xi|\geq 1} |\xi|^{2|\alpha''|}\|\hat U_0\|^2d\xi\\
			&\leq C(1+t)^{-\frac34(\frac{2}{q}-1)-\frac{m}{2}}\|\partial^{\alpha'}_xU_0\|_{L^q}^2+t^{-k}e^{-2ct}\|\partial^{\alpha''}_xU_0\|_{L^2}^2,
			\emas
			where $\frac{1}{p}+\frac{1}{p'}=1$, $\frac{1}{2p'}+\frac{1}{q}=1$ and $q\in[1,2]$.
			This yields \eqref{Y1UL21}.
			
			If $U_0=(u_0,E_0,0)$, then
			\bmas
			\|\dxa Y_1(t)U_0\|_{L^2}^2&=\int _{\mathbb{R}^3} \|\xi^\alpha\tilde Y_1(t,\xi) \hat V_0 \|_\xi ^2d\xi\\
			&\leq C\left(\int _{|\xi|\leq 1}+\int _{|\xi|\geq 1}\right)(|\xi|^{2|\alpha|}e^{-2c|\xi|^2t}+|\xi|^{2|\alpha|+2}e^{-2c|\xi|^4t})\|\hat V_0 \|_\xi ^2d\xi\\
			&\leq C(1+t)^{-\frac54(\frac{2}{q}-1)-\frac{m}{2}}\|\partial^{\alpha'}_xU_0\|_{L^q}^2+t^{-k}e^{-2ct}\|\partial^{\alpha''}_xU_0\|_{L^2}^2.
			\emas
			This implies \eqref{Y1UL2}.
			Moreover, for the $L^\infty$ norm, by Young's inequality we  have
			\bma\label{Y1U}
			&\quad	\|\dxa Y_1(t)U_0\|_{L^\infty}\nnm
			\\
			&	\leq C\left\|\mathcal{F}^{-1} \(e^{-b_0(|\xi|)t}\xi^{\alpha-\alpha''}\)\right\|_{L^{r'_2}}\left\|\mathcal{F}^{-1}\(\xi^{\alpha''}(\hat V_0, {\Lambda}_0)_\xi\mathcal{Q}(\xi) \Lambda_0\)\right\|_{L^{r_2}}\nnm\\
			& \quad +C\sum_{j=2,3}\left\|\mathcal{F}^{-1} \(\chi(\xi)|\xi|\xi^{\alpha-\alpha'}e^{-b_j(|\xi|)t}\)\right\|_{L^{r_1'}}\left\|\mathcal{F}^{-1}\( |\xi|^{-1}\xi^{\alpha''}(\hat V_0, {\Lambda}_j)_\xi\mathcal{Q}(\xi) \Lambda_j\)\right\|_{L^{r_1}}\nnm\\
			& \quad +C\sum_{j=2,3}\left\|\mathcal{F}^{-1} \({(1-\chi(\xi))}\xi^{\alpha-\alpha''}e^{-b_j(|\xi|)t}\)\right\|_{L^{r_2'}}\left\|\mathcal{F}^{-1}\(\xi^{\alpha''}(\hat V_0, {\Lambda}_j)_\xi\mathcal{Q}(\xi) \Lambda_j\)\right\|_{L^{r_2}},
			\ema
			where $r_i\in(1,\infty)$,  $r'_i=\frac{r_i}{r_i-1}$, $i=1,2$, and $\chi$ is a cutoff function.
			By Lemma \ref{kern} and Lemma \ref{kern1}, we obtain for any $p\in(1,\infty)$,
			\begin{align}
				&\left\|\mathcal{F}^{-1} \(e^{-b_0(|\xi|)t}\xi^{\alpha-\alpha'}\)\right\|_{L^{p}}\leq C t^{-\frac{3}{2}(1-\frac{1}{p})-\frac{m}{2}},\label{lp1}\\
				&\left\|\mathcal{F}^{-1} \(\chi(\xi)|\xi|\xi^{\alpha-\alpha'}e^{-b_j(|\xi|)t}\)\right\|_{L^{p}}\leq C(1+t)^{-\frac{3}{4}{(1-\frac{1}{p})}-\frac{m+1}{4}},\label{lp22}\\
				&\left\|\mathcal{F}^{-1} \({(1-\chi (\xi))}\xi^{\alpha-\alpha''}e^{-b_j(|\xi|)t}\)\right\|_{L^{p}}\leq Ct^{-\frac{3}{2}({1-\frac{1}{p}})-\frac{k}{2}},\quad j=2,3.\label{lp23}
			\end{align}
			Applying Lemma \ref{thmhm}, we have for any $r\in(1,\infty)$,
			\be \label{lp3}
			\left\|\mathcal{F}^{-1}\(\xi^{\alpha'}(\hat V_0, {\Lambda}_j)_\xi\mathcal{Q}(\xi) \Lambda_j\)\right\|_{L^r}\leq C\|\partial_x^{\alpha'}U_0\|_{L^r},\ \ \ j=0,2,3.
			\ee
			Specially, if $U_0=(f_0,E_0,0)$, then for $j=2,3$,
			\begin{align}
				&\quad \left\|\mathcal{F}^{-1}\( |\xi|^{-1}\xi^{\alpha'}(\hat V_0, {\Lambda}_j)_\xi\mathcal{Q}(\xi) \Lambda_j\)\right\|_{L^r}\nonumber\\
				&=\bigg\|\mathcal{F}^{-1}\bigg( \frac{\xi^{\alpha'}}{1+|\xi|^2} (\hat{m}_0\cdot W^j)(|\xi|v\cdot W^j\chi_0,0,-iW^j) \bigg)\bigg\|_{L^r}\nonumber\\
				&\leq C\|\partial_x^{\alpha'}U_0\|_{L^r}.\label{lp4}
			\end{align}
			Combining \eqref{Y1U}-\eqref{lp4}, we obtain \eqref{de1}.
			This completes the proof of the lemma.
		\end{proof}

		\section{Diffusion limit}
		\setcounter{equation}{0}
		\label{sect4}
		In this section, we study the diffusion limit of the solution to the nonlinear VMB system \eqref{VMB4}--\eqref{VMB4d} based on the fluid approximations of the solution to the linear VMB system given in Section 3.
		
		Since the operators $\AA_{\eps}$  generates a  contraction semigroup in $H^k$, the solution $U_{\eps}(t)=( f_{\eps},E_{\eps},B_{\eps})(t)$  to the VMB system \eqref{VMB4}--\eqref{VMB2i} can be represented by
		\be
		U_{\eps}(t)=e^{\frac{t}{\eps^2}\AA_{\eps}}U_0+\intt e^{\frac{t-s}{\eps^2}\AA_{\eps}} \(\Lambda_{1}(s)+\frac1{\eps}\Lambda_{2}(s)\)  ds, \label{ue}
		\ee
		where $U_0=( f_{0},E_{0},B_{0})$, and the nonlinear terms $\Lambda_k$, $k=1,2$ are given by
		\be\label{defG12}
		\begin{aligned}
			&\Lambda_{1}=(\Lambda_{11},0,0), \quad \Lambda_{11}=\frac12v\cdot E_{\eps}f_{\eps}-E_{\eps}\cdot\Tdv f_{\eps}-(v\times B_{\eps})\cdot\Tdv  f_{\eps}, \\
			&\Lambda_{2}=(\Lambda_{21},0,0), \quad \Lambda_{21}=\Gamma(f_{\eps},f_{\eps}).
		\end{aligned}
		\ee
		Also, by Lemma \ref{sem}, the solution $U_1(t)=(u_1,E,B )(t)$ with $u_1 =n\chi_0+m\cdot v\chi_0+q\chi_4$  to the NSMF system  \eqref{NSM_2}--\eqref{NSP_5i} can be represented by
		\be
		U_1(t)=Y_1(t){ P_A}U_0+\intt Y_1(t-s) \big(H_1(s)+\operatorname{div}_xH_2(s)\big)ds,\label{ue1}
		\ee
		where
		\begin{equation}
			\label{defH12}
			\begin{aligned}
				&	H_1=(H_{11},0,0), \quad 	H_{11}=(nE+m\times B)\cdot v\chi_0+\sqrt{\frac23}m\cdot E\chi_4,\\
				&	H_2 =(H_{21},0,0), \quad H_{21}=-(m\otimes m)\cdot v\chi_0-\frac53 mq\chi_4.
			\end{aligned}
		\end{equation}


		\subsection{Energy estimate}

		We first obtain some energy estimates.
		Let $N\ge 1$ be a positive integer and $U_{\eps}=(f_{\eps},E_{\eps},B_{\eps})$, and
		\bma
		E_{N,k}(U_{\eps})&=\sum_{|\alpha|+|\beta|\le N}\|\nu^k\dxa\dvb f_{\eps}\|^2_{L^2}+\sum_{|\alpha|\le N}\|\dxa(E_{\eps},B_{\eps})\|^2_{L^2_x},\label{energy3}
		\\
		H_{N,k}(U_{\eps})&= \frac{1}{\eps^2}\sum_{|\alpha|+|\beta|\le N}\|\nu^k\dxa\dvb
		P_1f_{\eps}\|^2_{L^2}+\sum_{1\le|\alpha|\le N}\|\dxa (E_{\eps},B_{\eps})\|^2_{L^2_x}\nnm\\
		&\quad +\sum_{|\alpha|\le N-1}\|\dxa\Tdx   P_0f_{\eps}\|^2_{L^2}+\|P_{d}f_{\eps}\|^2_{L^2_x},
		\\
		D_{N,k}(U_{\eps})&=\sum_{|\alpha|+|\beta|\le N}\frac1{\eps^2}\|\nu^{\frac12+k}\dxa\dvb  P_1f_{\eps}\|^2_{L^2}+\sum_{1\le |\alpha|\le N-1}\|\dxa E_{\eps}\|^2_{L^2_x}+\sum_{2\le |\alpha|\le N-1}\|\dxa B_{\eps}\|^2_{L^2_x}\nnm\\
		&\quad+\sum_{|\alpha|\le N-1}\|\dxa\Tdx  P_0f_{\eps}\|^2_{L^2} +\|P_d f_\eps\|_{L^2}^2,
		\ema
		for $k\ge 0$. For brevity, we write $E_N(U_{\eps})=E_{N,0}(U_{\eps})$, $H_N(U_{\eps})=H_{N,0}(U_{\eps})$    and $D_N(U_{\eps})=D_{N,0}(U_{\eps})$ for $k=0$.

		Firstly, by taking the inner product between $\chi_j\ (j=0,1,2,3,4)$ and \eqref{VMB4}, we obtain  a  compressible Euler-Maxwell type system of $(n_{\eps},m_{\eps},q_{\eps})=((f_{\eps},\chi_0),(f_{\eps},v\chi_0),(f_{\eps},\chi_4))$ and $E_{\eps},B_{\eps}$:
		\bma
		&\dt n_{\eps}+\frac{1}{\eps}\divx  m_{\eps}=0,\label{G_3}\\
		&\dt  m_{\eps}+\frac{1}{\eps}\Big(\Tdx n_{\eps}+ \sqrt{\frac23}\Tdx q_{\eps}- E_\eps\Big)= n_{\eps} E_{\eps}+\frac{1}{\eps}m_{\eps}\times B_{\eps} -\frac{1}{\eps}( v\cdot\Tdx( P_1f_{\eps}), v\chi_0),\label{G_5}\\
		&\dt q_{\eps}+\frac{1}{\eps}\sqrt{\frac23}\divx m_{\eps}=\sqrt{\frac23} E_{\eps}\cdot m_{\eps} -\frac{1}{\eps}(v\cdot\Tdx( P_1f_{\eps}), \chi_4 ), \label{G_6}\\
		& \dt E_{\eps}- \frac{1}{\eps}\Tdx\times B_{\eps}=-  \frac{1}{\eps}m_\eps,  \quad
		\dt B_{\eps}+ \frac{1}{\eps}\Tdx\times E_{\eps}=0.
		\ema
		
		Taking the microscopic projection $ P_1$ on \eqref{VMB4}, we have
		\be
		\dt( P_1f_{\eps})+ \frac{1}{\eps}P_1(v\cdot\Tdx  P_1f_{\eps})-\frac{1}{\eps^2}L( P_1f_{\eps})=- \frac{1}{\eps}P_1(v\cdot\Tdx P_0f_{\eps})+ P_1  H_{\eps}.\label{GG1}
		\ee
		By \eqref{GG1}, we can express the microscopic part $ P_1f_{\eps}$ as
		\bq   \frac{1}{\eps}P_1f_{\eps}= L^{-1}[\eps\dt( P_1f_{\eps})+P_1(v\cdot\Tdx  P_1f_{\eps})-\eps P_1  H_{\eps}]+ L^{-1} P_1(v\cdot\Tdx P_0f_{\eps}). \label{p_1}\eq
		Substituting \eqref{p_1} into \eqref{G_3}--\eqref{G_6}, we obtain
		a compressible Navier-Stokes-Maxwell type system
		\begin{equation}\label{nsm}
			\begin{aligned}
				&\dt n_{\eps}+\frac{1}{\eps}\divx  m_{\eps}=0,\\
				&\dt  (m_{\eps}+\eps R_1)+\frac{1}{\eps}\Big(\Tdx n_{\eps}+ \sqrt{\frac23}\Tdx q_{\eps}- E_\eps\Big)\\
				& =\kappa_0 \Big(\Delta_x m_{\eps}+\frac13\Tdx{\rm div}_x m_{\eps}\Big)+n_{\eps} E_{\eps}+\frac{1}{\eps}m_{\eps}\times B_{\eps}+R_3,\\
				&\dt (q_{\eps}+\eps R_2)+\frac{1}{\eps}\sqrt{\frac23}\divx m_{\eps}=\kappa_1 \Delta_x q_{\eps}+\sqrt{\frac23} E_{\eps}\cdot m_{\eps}+R_4,\\
				& \dt E_{\eps}- \frac{1}{\eps}\Tdx\times B_{\eps}=-  \frac{1}{\eps}m_\eps,  \quad
				\dt B_{\eps}+ \frac{1}{\eps}\Tdx\times E_{\eps}=0,
			\end{aligned}
		\end{equation}
		where the remainder terms $R_1, R_2, R_3, R_4$ are given by
		\bmas
		R_1&=( v\cdot\Tdx L^{-1}( P_1f_{\eps}),v\chi_0), \ \ R_2=( v\cdot\Tdx L^{-1}( P_1f_{\eps}),\chi_4),\\
		R_3&=-( v\cdot\Tdx L^{-1}[ P_1(v\cdot\Tdx  P_1f_{\eps})-\eps P_1 H_{\eps}],v\chi_0),\\
		R_4&=-( v\cdot\Tdx L^{-1}[ P_1(v\cdot\Tdx  P_1f_{\eps})-\eps P_1 H_{\eps}],\chi_4).
		\emas
		
		By  the similar  argument used in \cite{Duan4,Li4,Strain}, we have the existence and the energy estimate for the solution $U_{\eps}=(f_{\eps},E_{\eps},B_{\eps})$ to the VMB system \eqref{VMB4}--\eqref{VMB2i}.
		
		\begin{lem}[Macroscopic dissipation] \label{macro-en} Given $ N\ge 3$.  Let $(n_{\eps},m_{\eps},q_{\eps},E_{\eps},B_{\eps})$ be the strong solutions to \eqref{nsm}. Then, there are two constants $s_0,s_1>0$ such that for any $\eps\in (0,1)$,
			\bmas
			&\Dt \sum_{|\alpha|\le N-1}s_0(\|\dxa(n_{\eps}, m_{\eps},q_{\eps},E_\eps,B_\eps)\|^2_{L^2_x}+2\eps\intr \dxa R_1\dxa m_{\eps}dx+2\eps\intr \dxa R_2\dxa q_{\eps}dx)\nnm\\
			&+\Dt \sum_{|\alpha|\le N-1}4\eps\intr \dxa m_{\eps} \dxa\Tdx n_{\eps}dx-\Dt  \sum_{1\leq |\alpha|\le N-1}8\eps\int_{\R^3}\partial_x^\alpha m_\eps\partial_x^\alpha E_\eps dx\\
			&-\Dt  \sum_{1\leq |\alpha|\le N-2}3\eps\int_{\mathbb{R}^3}\partial_x^\alpha E_\eps \dxa(\nabla_x\times B_\eps) dx-\Dt s_0\sqrt{\frac23}\int_{\R^3}m_\eps^2 q_\eps dx\\
			&+\sum_{|\alpha|\le N-1}( \|\dxa\Tdx (n_{\eps}, m_{\eps},q_{\eps})\|^2_{L^2_x}+\|\dxa n_\eps\|_{L^2_x}^2)+\sum_{1\le |\alpha|\le N-1}\|\dxa E_{\eps}\|^2_{L^2_x}+\sum_{2\le |\alpha|\le N-1}\|\dxa B_{\eps}\|^2_{L^2_x}
			\nnm\\
			\le &\, C\sqrt{E_N(U_{\eps})}D_N(U_{\eps})+C\sum_{|\alpha|\le N-1}\|\dxa\Tdx P_1f_{\eps}\|^2_{L^2},
			\emas
			where $C>0$ is a constant independent of $\eps$.
		\end{lem}
		
		
		\begin{lem}[Microscopic dissipation]
			\label{micro-en}
			Given $N\ge 3$. Let $U_{\eps}=(f_{\eps},E_{\eps},B_{\eps})$ be a strong solution to VMB system  \eqref{VMB4}--\eqref{VMB5}.
			Then, there are constants $p_k>0$, $1\le k\le N$ such that for any $\eps\in (0,1)$,
			\bmas
			&\quad \Dt \sum_{|\alpha|\le N} (\|\dxa f_{\eps}\|^2_{L^2}+\|\dxa (E_{\eps},B_{\eps})\|^2_{L^2_x})-\sqrt{\frac23}\Dt\int_{\mathbb{R}^3}m_\eps^2 q_\eps dx+\frac{\mu}{\eps^2} \sum_{|\alpha|\le N}\|\nu^{\frac12} \dxa P_1f_{\eps}\|^2_{L^2}\nnm\\
			&\le C\sqrt{E_N(U_{\eps})}D_N(U_{\eps}), 
			\\
			&\quad \Dt \sum_{1\le k\le N}p_k\sum_{|\beta|=k \atop |\alpha|+|\beta|\le N}\|\dxa\dvb P_1f_{\eps} \|^2_{L^2} +\frac{1}{\eps^2} \sum_{|\beta|\ge 1 \atop |\alpha|+|\beta|\le N}\|\nu^{\frac12}\dxa\dvb P_1f_{\eps} \|^2_{L^2}\nnm\\
			&\le C\sum_{|\alpha|\le N-1}\|\dxa\Tdx f_{\eps} \|^2_{L^2} +C\sqrt{E_N(U_{\eps})}D_N(U_{\eps}), 
			\emas
			where $C>0$ is a constant independent of $\eps$.
		\end{lem}

		\begin{lem}\label{energy1}  Let $ N\ge 3$. For any $\eps\in (0,1)$, there exists a small constant $\delta_0>0$ and energy functionals $\mathcal{E}_{N}(U_{\eps})\sim E_{N}(U_{\eps})$ and  $\mathcal{H}_{N}(U_{\eps})\sim H_{N}(U_{\eps})$  such that if the initial data $U_0=(f_{0},E_{0},B_{0})$ satisfies  $E_N(U_0)\le \delta_0^2$, then the  system \eqref{VMB4}--\eqref{VMB2i} admits a unique global solution $U_{\eps}=(f_{\eps},E_{\eps},B_{\eps})$ satisfying the following  energy estimate:
			\be
			\Dt \mathcal{E}_N(U_{\eps}(t))+  D_N(U_{\eps}(t)) \le 0, \label{G_1}
			\ee
			\be
			\Dt \mathcal{H}_N(U_{\eps}(t))+  D_N(U_{\eps}(t)) \le C\|\nabla_x P_0 f_\eps\|_{L^2}^2. \label{G_2}
			\ee
			Moreover, there exist energy functionals $\mathcal{E}_{N,1}(U_{\eps})\sim E_{N,1}(U_{\eps})$  and $\mathcal{H}_{N,1}(U_{\eps})\sim H_{N,1}(U_{\eps})$ such that if the initial data $U_0$ satisfies $E_{N,1}(U_0)\le \delta_0^2$, then
			\be
			\Dt \mathcal{E}_{N,1}(U_{\eps}(t))+  D_{N,1}(U_{\eps}(t)) \le 0, \label{G_1b}
			\ee
			\be
			\Dt \mathcal{H}_{N,1}(U_{\eps}(t))+  D_{N,1}(U_{\eps}(t)) \le C\|\nabla_x P_0 f_\eps\|_{L^2}^2. \label{G_2b}
			\ee
		\end{lem}
		
		With the help of Lemmas \ref{macro-en}, \ref{micro-en} and \ref{energy1}, we have the time decay rate of $U_{\eps}=(f_{\eps},E_{\eps},B_{\eps})$ as follows.
		
		\begin{lem}\label{time7} Let $N\ge 3$. For any $\eps\in (0,1)$,  there exists a small constant $\delta_0>0$ such that if the initial data $U_0=(f_{0},E_{0},B_{0})$ satisfies that $E_{ N+1,1}(U_0)+\|U_0\|^2_{L^1 }\le \delta_0^2$, then the solution
			$U_{\eps}(t)=(f_{\eps},E_{\eps},B_{\eps})$ to the system \eqref{VMB4}--\eqref{VMB2i} has the following time-decay rate estimates:
			\be
			\|f_{\eps}(t)\|_{X^N_1}+\| (E_{\eps},B_{\eps})(t)\|_{H^N_x}\le C\delta_0 (1+t)^{-\frac38}, \label{G_11}
			\ee
			where $C>0$ is a constant independent of $\eps$. In particular, we have
			\be
			\|P_1f_{\eps}(t)\|_{H^{N-2}}\le C\delta_0\(\eps(1+t)^{-\frac58}+e^{-\frac{dt}{\eps^2}}\),\label{G_55}
			\ee
			where $d,C>0$ are two constants independent of $\eps$.
		\end{lem}
		\begin{proof}
			Define
			$$
			Q_{\eps}(t)= \sup_{0\le s\le t}\Big\{(1+s)^{\frac38}E_{N,1}(U_{\eps}(s))^{\frac12}
			\Big\}.
			$$
			We claim that
			\be
			Q_{\eps}(t)\le C\delta_0.  \label{assume}
			\ee
			It is straightforward  to verify that the estimate \eqref{G_11}  follows  from \eqref{assume}.

			Since  ${ P_A}\Lambda_{2} =0$,
			it follows from Lemma \ref{time} and \eqref{ue}  that,
			\bma
			\| U_{\eps} (t)\|_{L^2 }&\le
			C(1+t)^{-\frac38}\(\| U_0\|_{L^{2} }+\|U_0\|_{L^{1} }+\|\Tdx U_0\|_{L^{2} }\) \nnm\\
			&\quad+C \int^{t}_0 \( (1+t-s)^{ -\frac58}+\frac1{\eps}e^{-\frac{d(t-s)}{\eps^2}}\)\big(\| \Lambda_1(s)\|_{L^{2}\cap L^1 }\nnm\\
			&\qquad +\|\Tdx \Lambda_1(s)\|_{L^{2} }+\| \Lambda_2(s)\|_{H^{1}\cap L^1 } +\|\Tdx \Lambda_2(s)\|_{L^{2} } \big)ds\nnm\\
			&\le
			C\delta_0(1+t)^{-\frac38}+CQ_{\eps}(t)^2\intt(1+t-s)^{-\frac58} (1+s)^{-\frac34}ds\nnm\\
			&\le
			C\delta_0(1+t)^{-\frac38}+CQ_{\eps}(t)^2 (1+t)^{-\frac38},   \label{density_1a}
			\ema
			where we had used
			\be
			\| \Lambda_k(s)\|_{H^{N-1} }+\| \Lambda_k(s)\|_{L^{1} } \le CQ_{\eps}(t)^2 (1+s)^{-\frac34} , \quad k=1,2.\label{h3}
			\ee
			Let $0<l< 1$ and $n\ge 3$. Multiplying \eqref{G_1b} by $(1 + t)^l$ and
			then taking time integration over $[0, t]$ gives
			\bma
			&(1+t)^lE_{n,1}(U_{\eps})(t)+ \intt (1+s)^lD_{n,1}(U_{\eps})(s)ds
			\nnm\\
			&\le CE_{n,1}(U_0)+Cl\intt (1+s)^{l-1}E_{n,1}(U_{\eps})(s)ds\nnm\\
			&\le CE_{n,1}(U_0)+Cl\intt (1+s)^{l-1}D_{n+1,1}(U_{\eps})(s)ds\nnm\\
			&\quad+Cl\intt (1+s)^{l-1}(\|P_0f_{\eps}(s)\|^2_{L^2}+\|B_{\eps} \|^2_{H^1_x}+\|E_\eps\|_{L^2_x}^2)ds, \label{l5}
			\ema
			where we had used
			$$E_{n,1}(U_{\eps})\le CD_{n+1,1}(U_{\eps})+C(\|P_0f_{\eps} \|^2_{L^2}+\|B_{\eps} \|^2_{H^1_x}+\|E_\eps\|_{L^2_x}^2).$$
			And it follows from \eqref{G_1b} that
			\be
			E_{n+1,1}(U_{\eps})(t)+ \intt D_{n+1,1}(U_{\eps})(s)ds \le  CE_{n+1,1}(U_{0}). \label{l6}
			\ee
			By \eqref{l5}--\eqref{l6}, we obtain
			\bmas
			&\quad (1+t)^lE_{n,1}(U_{\eps})(t)+ \intt (1+s)^lD_{n,1}(U_{\eps})(s)ds
			\nnm\\
			&\le CE_{n+1,1}(U_0)+Cl\intt (1+s)^{l-1}(\|P_0f_{\eps}(s)\|^2_{L^2}+\|B_{\eps} \|^2_{H^1_x}+\|E_\eps\|_{L^2_x}^2)ds
			\emas
			for $0<l<1$ and $n\ge 3$. Taking $l=3/4+\sigma$ for a fixed constant $0<\sigma<1/4$  yields
			\bmas
			&(1+t)^{\frac34+\sigma}E_{n,1}(U_{\eps})(t)+ \intt (1+s)^{\frac34+\sigma}D_{n,1}(U_{\eps})(s)ds
			\nnm\\
			\le& CE_{n+1,1}(U_0)+C(\delta_0+Q_{\eps}^2(t))^2 \intt (1+s)^{-\frac14+\sigma}(1+s)^{-\frac34}ds\nnm\\
			\le& CE_{n+1,1}(U_0)+C(1+t)^{\sigma}(\delta_0+Q_{\eps}^2(t))^2.
			\emas  This gives
			$$
			E_{n,1}(U_{\eps})(t)\le C(1+t)^{-\frac34}(\delta_0+Q_{\eps}^2(t))^2.\label{J_6z}
			$$
			Then we obtain
			$$
			Q_{\eps}(t)\le C\delta_0+CQ_{\eps}(t)^2 ,
			$$
			which leads to \eqref{assume} provided $\delta_0>0$ sufficiently small.
			Moreover,
			it follows from Lemma \ref{time} and \eqref{h3} that
			\bma
			\| P_1f_{\eps} (t)\|_{H^k }&\le
			C\(\eps(1+t)^{-\frac58}+e^{-\frac{dt}{\eps^2}}\)\(\| U_0\|_{H^{k+1} }+\|U_0\|_{L^{1} }+\|\Tdx U_0\|_{H^{k} }\)\nnm\\
			&\quad+C\int^{t}_0 \( \eps(1+t-s)^{-\frac78}+\frac1{\eps}e^{-\frac{d(t-s)}{\eps^2}}\)\big(\| \Lambda_1(s)\|_{H^{k+1}\cap L^1 } \nnm\\
			&\qquad+ \|\Tdx \Lambda_1(s)\|_{H^{k}}+\|\Lambda_2(s)\|_{H^{k+2}\cap L^{1} }+ \|\Tdx \Lambda_2(s)\|_{H^{k}} \big)ds\nnm\\
			&\le
			C(\delta_0 +Q_{\eps}(t)^2)\(\eps(1+t)^{-\frac58}+e^{-\frac{dt}{\eps^2}}\), \label{micro_1a}
			\ema
			for  $k\le N-2$. This yields \eqref{G_55}.
			The proof of the lemma is completed.
		\end{proof}
		
		By \eqref{ue1} and Lemma \ref{timev}, we can show
		\begin{lem}\label{energy4} Let $N\ge 2$. There exists a small constant $\delta_0>0$ such that if $\|U_{0}\|_{H^N}+\|U_0\|_{L^1}\le \delta_0$, then
			the NSMF system \eqref{NSM_2}--\eqref{NSP_5i} admits a unique global solution $\tilde{U}(t,x)=(n,m,q,E,B)(t,x)$ satisfying
			\be
			\|\tilde{U} (t)\|_{H^N } \le C\delta_0 (1+t)^{-\frac38}. \label{time8}
			\ee
		\end{lem}
		\begin{proof}
			Define
			$$
			Q(t)=\sup_{0\le s\le t}\left\{(1+s)^{\frac38}\|\tilde{U}(s)\|_{H^N }\right\}.
			$$
			Then, it follows from Lemma \ref{timev} and \eqref{ue1} that
			\bma
			\| \tilde U (t)\|_{H^N}&\le  C(1+t)^{-\frac38}(\| { P_A}U_0\|_{H^N}+\|{P_A}U_0\|_{L^1})
			\nnm\\
			&\quad+C \intt \((1+t-s)^{ -\frac58}+(t-s)^{-\frac12}e^{-c(t-s)}\)\nnm\\
			&\qquad\times \sum_{i=1,2}(\| H_i(s)\|_{H^N} +\| H_i(s)\|_{ L^1 })ds
			\nnm\\
			&\le  C\delta_0(1+t)^{-\frac38}+CQ(t)^2 (1+t)^{-\frac38},\label{density_2}
			\ema
			where $H_i, i=1,2$ are defined in \eqref{defH12} and satisfy
			$$
			\sum_{i=1,2}(\| H_i(s)\|_{H^N} +\| H_i(s)\|_{ L^1 })\leq C\|\tilde U(s)\|_{H^N}^2\leq C(1+s)^{-\frac{3}{4}}Q(t)^2,\quad \forall s\in[0,t].
			$$
			It follows from \eqref{density_2} that
			$$
			Q(t)\le C\delta_0 ,
			$$
			provided $\delta_0>0$ sufficiently small. This proves \eqref{time8}.
			The existence of the solution can be proved by the contraction mapping theorem, the details are omitted.
			The proof of the lemma is completed.
		\end{proof}
		
		Theorem \ref{thm1.1}  follows from Lemmas \ref{time7} and \ref{energy3}.

		\subsection{Optimal convergence rate}
		\begin{lem}[\cite{Li1}]\label{gamma1}For any $i,j=1,2,3,$ it holds that
			\bmas
			\Gamma_*(v_i\chi_0,v_j\chi_0) &=-\frac12LP_1(v_iv_j\chi_0), \\
			\Gamma_*(v_i\chi_0,|v|^2\chi_0) &=-\frac12LP_1(v_i|v|^2\chi_0), \\
			\Gamma_*(|v|^2\chi_0,|v|^2\chi_0)&=-\frac12LP_1(|v|^4\chi_0).
			\emas
			where
			$$\Gamma_*(f,g) =\frac12[\Gamma(f,g)+\Gamma(g,f)].$$
		\end{lem}

		\begin{proof}[\underline{\bf Proof of Theorem \ref{thm1.2}}]
			First, we deal with \eqref{limit0}. Define
			$$
			\Pi_{\eps}(t)=\sup_{0\le s\le t}\bigg(\eps|\ln \eps|^4(1+s)^{-\frac{5-\sigma}8}+ \bigg(1+\frac{s}{\eps}\bigg)^{-1}\bigg)^{-1}\| U_{\eps}(s)-U_1(s) \|_{ L^\infty}.
			$$
			We claim that
			\be
			\Pi_{\eps}(t) \le C\delta_0 ,\quad \forall\, t>0. \label{limit6}
			\ee
			It is easy to verify that the estimate  \eqref{limit0} follows  from \eqref{limit6}.
			By \eqref{ue} and \eqref{ue1},
			\begin{align*}
				U_\eps(t)-U_1(t)&=
				e^{\frac{t}{\eps^2}\AA_\eps}U_0-Y_1(t){P_A}U_0+\intt \(e^{\frac{t-s}{\eps^2}\AA_{\eps}} \Lambda_{1}(s)-Y_1(t-s)H_1(s)\)ds\\
				&\quad +\intt  \(\frac1{\eps}e^{\frac{t-s}{\eps^2}\AA_{\eps}}\Lambda_{2}(s)-Y_1(t-s)\operatorname{div}_xH_2(s)\)  ds\\
				&=:I_1+I_2+I_3.
			\end{align*}
			By Lemma \ref{fl1}, one has
			\bma	
			\left\|I_1\right\|_{L^\infty}
			&\le C\bigg( \eps(1+t)^{-1 }+ \(1+\frac t \eps\)^{-\frac{3-\sigma}2}\bigg)(\|U_0\|_{H^3}+\|U_0\|_{L^1}) \nnm\\
			&\le C\bigg( \eps(1+t)^{-1 }+\(1+\frac t \eps\)^{-\frac{3-\sigma}2}\bigg)\delta_0.\label{i1es}
			\ema
			For $I_2$, we further decompse
			\begin{align*}
				I_2&=\intt\left(e^{\frac{t-s}{\eps^2}\AA_{\eps}} \Lambda_{1}(s)-Y_1(t-s){P_A}\Lambda_1(s)\right)ds\\
				&\quad+\intt\left( Y_1(t-s){ P_A}\Lambda_1(s) -Y_1(t-s)H_1(s)\right)ds\\
				&=:I_{21}+I_{22}.
			\end{align*}
			By Lemma \ref{fl1} and \eqref{h3},  we obtain
			\bma\label{i21es}
			\|I_{21}\|_{L^\infty}&\leq C\intt \bigg( \eps(1+t-s)^{-1}+\(1+\frac {t-s} \eps\)^{-\frac{3-\sigma}2}\bigg)\|\Lambda_1(s)\|_{H^3\cap L^1} ds \nnm\\
			&\leq C\delta_0^2\intt \bigg( \eps(1+t-s)^{-1}+\(1+\frac {t-s} \eps\)^{-\frac{3-\sigma}2}\bigg)(1+s)^{ -\frac34} ds\nnm\\
			&\leq  C \delta_0^2\eps(1+t)^{-\frac58},
			\ema where we had used
			$$ \|\Lambda_1(s)\|_{H^3\cap L^1}\le CE_{4,1}(s)\le  C \delta_0^2 (1+t)^{-\frac34}, \quad \forall s\in [0,t],$$
			For $I_{22}$, due to
			$$
			P_0\Lambda_{11} =(n_\eps E_\eps+m_\eps\times B_\eps)\cdot v\chi_0+\sqrt{\frac23}m_\eps\cdot E_\eps\chi_4 ,
			$$
			we obtain from Lemma \ref{timev} and Lemma \ref{time7} that
			\bma\label{i22es}
			\|	I_{22}\|_{L^\infty}& \leq C \intt \( (t-s)^{-\frac34}+(t-s)^{-\frac58}\)\|P_0\Lambda_{11}(s)-H_{11}(s)\|_{L^2}ds \nnm\\		&\leq C \intt \( (t-s)^{-\frac34}+(t-s)^{-\frac58}\)\|U_\eps-U_1\|_{L^\infty}\(\|U_1\|_{H^2}+\|U_\eps\|_{H^2}\)ds \nnm\\
			&\leq C\delta_0\Pi_{\eps}(t)\intt \((t-s)^{-\frac34}+(t-s)^{-\frac58}\)\nnm\\
			&\qquad\times \bigg({\eps|\ln\eps|^4(1+s)^{-\frac{5-\sigma}8}+ \(1+\frac{s}{\eps}\)^{-1}}\bigg)(1+s)^{-\frac38}ds.
			\ema
			We have
			\begin{align}\label{j0}
				A_0&=:\eps|\ln \eps|^4\intt \((t-s)^{-\frac34}+(t-s)^{-\frac58}\)(1+s)^{-\frac{8-\sigma}8}ds\nnm\\
				&\leq C\eps|\ln \eps|^4(1+t)^{-\frac{5-\sigma}8}.
			\end{align}
			It remains to estimate
			\begin{align*}
				A_1=:\intt \((t-s)^{-\frac34}+(t-s)^{-\frac58}\) \(1+\frac{s}{\eps}\)^{-1}(1+s)^{-\frac38}ds.
			\end{align*}
			For $t\leq\epsilon$, we have
			\bma
			A_1
			&\leq C\int^{t}_0(t-s)^{-\frac34}ds\le C\leq C\(1+\frac{t}{\epsilon}\)^{-1}.\label{dlth2-3-1-5}
			\ema
			For $t\geq\epsilon$, we have
			\bma
			A_1&=\(\int^{\frac{t}{2}}_0+\int_{\frac{t}{2}}^t\)\((t-s)^{-\frac34}+(t-s)^{-\frac58}\)  \(1+\frac{s}{\eps}\)^{-1} (1+s)^{-\frac38}ds\nnm\\
			&\leq  C(t^{-\frac{3}{4}}+t^{-\frac58}) \int^{\frac t2}_0 \(1+\frac{s}{\eps}\)^{-1}(1+s)^{-\frac38}ds\nnm\\
			&\quad+C\(1+\frac{t}{\eps}\)^{-1} (1+t)^{-\frac38}\int_{\frac{t}{2}}^t\((t-s)^{-\frac34}+(t-s)^{-\frac58}\)ds\nnm\\
			& \leq C\epsilon|\ln\eps|^4 (1+t)^{-\frac{5}{8}}+C\(1+\frac{t}{\epsilon}\)^{-1},\label{dlth2-3-1-7}
			\ema
			where we have used
			\begin{align*}
				&\quad (t^{-\frac{3}{4}}+t^{-\frac58})	\int_0^{\frac{t}{2}}\(1+\frac{s}{\eps}\)^{-1}(1+s)^{-\frac38}ds\\
				& \leq \begin{cases}
					2t^{-\frac{3}{4}} \int_0^{\frac{t}{2}}\(1+\frac{s}{\eps}\)^{-1}ds\leq 	2t^{-\frac{3}{4}} \eps|\ln\eps|\leq C\eps|\ln \eps|^4+C\(1+\frac{t}{\eps}\)^{-1},\ \ \ \eps<t\leq 1,\\
					2t^{-\frac58} \(\int_0^1\(1+\frac{s}{\eps}\)^{-1}ds+\eps\int_1^t s^{-\frac{11}{8}}ds\)\leq C\eps|\ln\eps|(1+t)^{-\frac58},\quad\quad\quad t>1.
				\end{cases}
			\end{align*}
			We conclude from \eqref{i22es}-\eqref{dlth2-3-1-7} that
			\bma
			\|I_{22}\|_{L^\infty}&\leq C\delta_0\Pi_{\eps}(t)(A_0+A_1)\nonumber\\
			&\leq C\delta_0\Pi_{\eps}(t)\bigg(\epsilon|\ln\eps|^4 (1+t)^{-\frac{5-\sigma}{8}}+C\(1+\frac{t}{\epsilon}\)^{-1}\bigg).\label{i22es'}
			\ema
			Finally, to estimate $I_3$, we decompose
			\begin{align*}
				I_3&=\intt \(\frac1{\eps}e^{\frac{t-s}{\eps^2}\AA_{\eps}}\Lambda_{2}(s)-Y_1(t-s)Z_0(s)\)  ds\\
				&\quad +\intt  \(Y_1(t-s)Z_0(s)-Y_1(t-s)\operatorname{div}_xH_2(s)\)  ds\\
				&=:I_{31}+I_{32},
			\end{align*}
			where $Z_0=(P_0(v\cdot \nabla_x L^{-1}\Lambda_{21}),0,0)$. Thus, it follows from
			Lemma \ref{fl2} and \eqref{h3} that
			\bma
			\|I_{31}\|_{L^{\infty}}&\leq C\intt \bigg(\eps(1+t-s)^{-\frac{5}{4}}+\frac{1}{\eps}e^{-\frac {d(t-s)}{\eps^2}}+\left(1+\frac{t-s}{\eps}\right)^{-\frac{3-\sigma}2}\bigg) \|\Lambda_{2}(s)\|_{H^4\cap L^1} ds \nnm\\
			&\leq  C\delta_0^2\intt \bigg(\eps(1+t-s)^{-\frac{5}{4}}+\frac{1}{\eps}e^{-\frac {d(t-s)}{\eps^2}}+\left(1+\frac{t-s}{\eps}\right)^{-\frac{3-\sigma}2}\bigg) (1+s)^{-\frac34}ds \nnm\\
			&\leq C\delta_0^2\eps(1+t)^{-\frac34}.\label{i31es}
			\ema
			For $I_{32}$, we decompose
			\begin{align*}
				P_0(v\cdot \nabla_x L^{-1}\Lambda_{21})&=	P_0(v\cdot \nabla_x L^{-1}\Gamma_*(P_0f_\eps,P_0f_\eps))+2P_0(v\cdot \nabla_x L^{-1}\Gamma_*(P_0f_\eps,P_1f_\eps))\\
				&\quad+P_0(v\cdot \nabla_x L^{-1}\Gamma_*(P_1f_\eps,P_1f_\eps))\\
				&=: J_1+J_2+J_3,
			\end{align*}
			where $2\Gamma_*(f,g)=\Gamma(f,g)+\Gamma(g,f)$.
			By Lemma \ref{gamma1}, we obtain (cf. \cite{Li1})
			\be
			J_1=-\sum^3_{i, j=1}\pt_i(m^i_{\eps} m^j_{\eps})v_j\chi_0+\frac13\sum^3_{i,j=1}\pt_j( m^i_{\eps})^2 v_j\chi_0 -\frac53\sum^3_{j=1}\pt_j(m^j_{\eps}q_{\eps}) \chi_4. \label{J1}
			\ee
			
			For any $ g_0\in L^2(\R^3_x),$ set $G_0(x,v)=(v\chi_0\cdot\Tdx g_0,0,0)$. Since
			$(\hat{G}_0,\Lambda_j(\xi))_{\xi}=0,$  $j=0,2,3,$
			we have by \eqref{v1} that
			$$
			\tilde{Y}_1(t,\xi)\hat{G}_0 =\sum_{j=0,2,3} e^{-b_jt}\big(\hat{G}_0,\Lambda_j(\xi)\big)_{\xi}\Lambda_j(\xi) =0.
			$$
			This  and \eqref{J1} give
			$$
			Y_1(t)(J_1,0,0) =Y_1(t)(\divx J_4,0,0),
			$$
			where
			$$J_4=\((m_{\eps}\otimes m_{\eps})\cdot v\chi_0+\frac53 (q_{\eps}m_{\eps})\chi_4,0,0\).$$
			Then we have
			\begin{align*}
				\|I_{32}\|_{L^\infty}&\leq  \intt \|Y_1(t-s)(\divx(J_4-H_{21})(s),0,0)\|_{L^\infty}ds \nnm\\
				&\quad + \intt \|Y_1(t-s)(J_2(s)+ J_3(s),0,0)\|_{L^\infty}ds.
			\end{align*}
			Then by Lemma \ref{timev} and Sobolev inequality, we obtain
			\begin{align*}
				&	\|Y_1(t-s)(\divx(J_4-H_{21})(s),0,0)\|_{L^\infty}\leq C(t-s)^{-\frac34}\|(J_4-H_{21})(s)\|_{ L^3\cap L^6},\\
				& \|Y_1(t-s)( J_k(s),0,0)\|_{L^\infty}\leq C (1+t-s)^{-\frac58} \| J_k(s)\|_{H^2\cap L^1}, \quad k=2,3.
			\end{align*}
			Combining this with Lemma \ref{time7},  we obtain
			\bma
			\|I_{32}\|_{L^\infty}
			&\leq  C\intt(t-s)^{ -\frac34}\|J_4(s)- H_{21}(s)\|_{ L^3\cap L^6} ds \nnm\\
			& \quad+ C\intt{ (1+t-s)^{-\frac58} }(\| J_2(s)\|_{H^2\cap L^1}+\| J_3(s)\|_{H^2\cap L^1}) ds \nnm\\
			&\leq C\intt (t-s)^{-\frac34} \| U_\eps(s)-U_1(s)\|_{L^\infty}  (\| U_\eps(s)\|_{H^2}+\|U_1(s)\|_{H^{2}}) ds \nnm\\
			&\quad\quad+C\intt { (1+t-s)^{-\frac58} }\|f_\eps(s)\|_{H^3}\|P_1f_\eps(s)\|_{H^3}ds \nnm\\
			&\leq C\delta_0 \Pi_{\eps}(t)\intt(t-s)^{ -\frac34}\bigg(\eps|\ln\eps|^4(1+s)^{-\frac{5-\sigma}8}+ \bigg(1+\frac{s}{\eps}\bigg)^{-1}\bigg)(1+s)^{-\frac38} ds \nnm\\
			&\quad\quad+ C\delta_0^2\intt { (1+t-s)^{-\frac58} }\(\eps(1+s)^{-\frac58}+e^{-\frac{ds}{\eps^2}}\)(1+s)^{-\frac38} ds \nnm\\
			&\leq C(\delta_0^2
			+\delta_0\Pi_{\eps}(t)) \bigg(\eps|\ln\eps|^4(1+t)^{-\frac{5-\sigma}8}+\(1+\frac{t}{\eps}\)^{-1}\bigg).\label{i32es}
			\ema
			Therefore, it follows from \eqref{i1es}, \eqref{i21es}, \eqref{i22es'}, \eqref{i31es} and \eqref{i32es} that
			\begin{align*}
				\Pi_{\eps}(t)\leq C(\delta_0+\delta_0^2+\delta_0\Pi_{\eps}(t)).
			\end{align*}
			By taking $\delta_0$ small enough, we obtain \eqref{limit6}, which implies \eqref{limit0}.
			
			Next, we prove \eqref{limit_1a}. Set
			\begin{align*}
				\Omega_\eps(t)=\sup_{0\leq s\leq t}\eps^{-1}(1+s)^\frac{5-\sigma}8 \|U_\eps(s)-U_1(s)\|_{L^\infty}.
			\end{align*}
			If $U_0$ satisfies \eqref{initial}, then by Lemma \ref{fl1}, we have
			\be \label{i1new}
			\|I_1\|_{L^{\infty}}\le C\eps(1+t)^{-1}(\|U_0\|_{H^3}+\|U_0\|_{L^1}) \le C\eps\delta_0(1+t)^{-1}.
			\ee
			By \eqref{i21es} and \eqref{i31es}, we have
			\be \label{iii}
			\|I_{21}\|_{L^{\infty}}+\|I_{31}\|_{L^{\infty}}\le C \delta_0^2\eps(1+t)^{-\frac58}.
			\ee
			For $I_{22}$, similar to \eqref{i22es}, we obtain
			\begin{equation}\label{i22es1}
				\begin{aligned}
					\|	I_{22}\|_{L^\infty}
					&\leq C \intt \( (t-s)^{-\frac34}+(t-s)^{-\frac58}\)\|U_\eps-U_1\|_{L^\infty}\(\|U_1\|_{H^2}+\|U_\eps\|_{H^2}\)ds\\
					&\leq C\delta_0\Omega_{\eps}(t)\eps\intt \((t-s)^{-\frac34}+(t-s)^{-\frac58}\)(1+s)^{-\frac{8-\sigma}8}ds\\
					&\leq C\delta_0\Omega_{\eps}(t) \eps(1+t)^{-\frac{5-\sigma}8}.
				\end{aligned}
			\end{equation}
			Moreover, from Lemma \ref{time7}, we get
			\bma
			\|I_{32}\|_{L^\infty}&\leq C\intt (t-s)^{-\frac34} \| U_\eps(s)-U_1(s)\|_{L^\infty}  (\| U_\eps(s)\|_{H^3}+\|U_1(s)\|_{H^{3}}) ds\nnm\\
			&\quad+C\intt { (1+t-s)^{-\frac58} } \|f_\eps(s)\|_{H^3}\|P_1f_\eps(s)\|_{H^3}ds \nnm\\
			&\leq C\eps\delta_0 \Omega_\eps(t)\intt  (t-s)^{ -\frac34} (1+s)^{-\frac{8-\sigma}{8}}ds \nnm\\
			&\quad+ C\delta_0^2\intt { (1+t-s)^{-\frac58} } \(\eps(1+s)^{-\frac58}+e^{-\frac{ds}{\eps^2}}\)(1+s)^{-\frac{3}{8}}ds \nnm\\
			&\leq C\eps(\delta_0^2
			+\delta_0\Omega_\eps(t))(1+t)^{ -\frac{5-\sigma}8}. \label{i32es1}
			\ema
			Combining \eqref{i1new}, \eqref{iii},  and  \eqref{i32es1}, we obtain
			\begin{align*}
				\Omega_\eps(t)\leq C(\delta_0+\delta_0^2+\delta_0\Omega_\eps(t)).
			\end{align*}
			This implies \eqref{limit_1a} by taking $\delta_0$ small enough. Then we complete the proof.
		\end{proof}

		\section{Appendix}
		\begin{lem}[\cite{Nel}]
			\label{lemdisp}
			Let $a\geq\frac{5}{4}$, $p\in[2,\infty]$, then for any $t>0$, there holds
			$$
			\left\|\int_{\mathbb{R}^3}e^{ix\cdot \xi}e^{-i\sqrt{1+|\xi|^2}t}(1+|\xi|^2)^{-a}d\xi\right\|_{L^p_x}\leq C t^{\frac{3}{p}-\frac{3}{2}}.
			$$
		\end{lem}
		\begin{lem}[Hausdorff--Young inequality \cite{Grafakos}]\label{lemHY}
			For $p\in[1,2]$, and any function $f\in L^p(\mathbb{R}^d)$, it holds
			$$
			\|\mathcal{F}f\|_{L^\frac{p}{p-1}}\leq C\|f\|_{L^p}.
			$$
		\end{lem}
		\begin{thm}[H\"{o}rmander--Mihlin multiplier Theorem \cite{Grafakos}]\label{thmhm} Let $m(\xi)$ be a complex-valued bounded function in $\mathbb{R}^d\backslash\{0\}$ satisfying
			$$
			|\nabla^\alpha_\xi m(\xi)|\leq C|\xi|^{-\alpha},
			$$
			for all multi-indices $|\alpha|\leq \frac{d}{2}+1$. Then for any $1<p<\infty$, the following estimate hold:
			$$
			\|\mathcal{F}^{-1}(m(\xi)\hat f(\xi))\|_{L^p}\leq C_p \|f\|_{L^p},\quad  \forall f\in\mathcal{S}(\mathbb{R}^d).
			$$
		\end{thm}
		For any $s\in\mathbb{R}$, $1<p<\infty$, the general  Sobolev space  $W^{s,p}$ is defined as the space of all tempered distributions $u\in\mathcal{S}(\mathbb{R}^d)$ with the property that
		$$
		\|u\|_{W^{s,p}}=:\|\mathcal{F}^{-1}\((1+|\xi|^2)^\frac{s}{2}u\)\|_{L^p}<\infty.
		$$
		We have the following interpolation lemma in the general Sobolev space.
		\begin{lem}[{Interpolation  inequality  \cite[section 2.4.2]{Triebel}}]\label{lemitp}
			For any $s_0,s_1\in\mathbb{R}$, $1<p_0,p_1<+\infty$, and $0<\theta<1$, it holds
			$$
			\|f\|_{W^{s,p}}\leq C\|f\|_{W^{s_0,p_0}}^{1-\theta}\|f\|_{W^{ s_1,p_1}}^\theta,
			$$
			where $s$ and $p$ are determined by
			$$
			s=(1-\theta)s_0+\theta s_1,\quad  \frac{1}{p}=\frac{1-\theta}{p_0}+\frac{\theta}{p_1}.
			$$
		\end{lem}
		\begin{lem}\label{kern}
			For  $a(\xi):\mathbb{R}^d\to \mathbb{R}$, if
			\begin{align*}
				&a(\xi)\leq -c_0|\xi|^2,\\
				&|\nabla^k a(\xi)|\leq C_k |\xi|^{2-k}, \quad 0\leq k\leq d+m+3,
			\end{align*}
			for some constants $c_0,C_k>0$, and $m\in\mathbb{N}$.
			Then for
			$$
			G(t,x)=\int_{\mathbb{R}^d}\xi^m\exp(a(\xi)t+i\xi\cdot x)d\xi,
			$$
			it holds
			$$
			\|G(t,\cdot)\|_{L^p}\leq Ct^{-\frac{d}{2}(1-\frac{1}{p})-\frac{m}{2}}.
			$$
		\end{lem}
		\begin{proof}
			Since $a(\xi) \leq -c_0 |\xi|^2$, we have
			\be\label{e11}
			\| G(t,\cdot) \|_{L^\infty} \leq \int_{\mathbb{R}^d} \exp(-c_0 |\xi|^2t) |\xi|^md\xi\leq Ct^{-\frac{d+m}{2}}.
			\ee
			Note that $\exp(i\xi\cdot x)=-\frac{ix\cdot \nabla_\xi\exp(i\xi\cdot x)}{|x|^2}$. If $m=0$, using integration by parts, we obtain
			\begin{align*}
				\int_{\mathbb{R}^d}\exp(a(\xi)t + i\xi \cdot x) d\xi &=i|x|^{-2}\int_{\mathbb{R}^d}x\cdot\nabla_\xi\( \exp(a(\xi)t) \)\exp( i\xi \cdot x) d\xi\\
				&=i|x|^{-2}\int_{\mathbb{R}^d}x\cdot\nabla_\xi\( \exp(a(\xi)t) \)\exp( i\xi \cdot x) \chi(\xi/\delta)d\xi\\
				&\quad +i|x|^{-2}\int_{\mathbb{R}^d}x\cdot\nabla_\xi\( \exp(a(\xi)t) \)\exp( i\xi \cdot x)(1- \chi(\xi/\delta))d\xi\\
				&=:A_1+A_2,
			\end{align*}
			where $\chi$ is a cutoff function and  $\delta=|x|^{-1}$. Note that
			$$
			\left|A_1\right|\leq Ct|x|^{-1}\int_{\mathbb{R}^d}|\xi| \chi(\xi/\delta)d\xi\leq Ct|x|^{-1}\delta^{d+1}.
			$$	
			For $A_2$, using  integration by parts $d+2$ times, we get
			\begin{align*}
				|A_2|&\leq 	|x|^{-(d+3)}	\int_{\mathbb{R}^d}\left|\nabla_\xi^{d+2}\(\nabla_\xi\big(\exp(a(\xi)t )\big)(1-\chi(\xi/\delta))\) \right|d\xi \\
				&\leq Ct|x|^{-(d+2)}.
			\end{align*}
			Hence,
			$$
			\left| \int_{\mathbb{R}^d} \exp(a(\xi)t + i\xi \cdot x) d\xi \right| \leq |A_1|+|A_2|\leq Ct|x|^{-(d+2)}.
			$$
			Combining this with \eqref{e11}, we obtain
			$$
			\left| \int_{\mathbb{R}^d} \exp(a(\xi)t + i\xi \cdot x) d\xi \right| \leq \frac{Ct}{(|x| + t^{\frac{1}{2}})^{d+2}},
			$$
			which implies
			$$
			\|G(t,\cdot)\|_{L^p}\leq Ct^{-\frac{d}{2}(1-\frac{1}{p})}.
			$$
			Similarly, if $m>0$, we can write
			\begin{align*}
				G(t,x)&=\int_{\mathbb{R}^d}\xi^m\exp(a(\xi)t+i\xi\cdot x)\chi(\xi/\delta)d\xi+\int_{\mathbb{R}^d}\xi^m\exp(a(\xi)t+i\xi\cdot x)(1-\chi(\xi/\delta))d\xi\\
				&=G_1(t,x)+G_2(t,x),
			\end{align*}
			where $\delta=|x|^{-1}$. Then
			$$
			|G_1(t,x)|\leq \int_{\mathbb{R}^d}|\xi|^m\chi(\xi/\delta)d\xi\leq C|x|^{-(d+m)}.
			$$
			Using integration by parts, we obtain
			$$
			|G_2(t,x)\leq |x|^{-(d+m+1)}\int_{\mathbb{R}^d}\left|\nabla_\xi^{d+m+1}\(\xi^m\exp(a(\xi)t)(1-\chi(\xi/\delta))\)\right|d\xi\leq C|x|^{-(d+m)}.
			$$
			Hence we obtain
			$$
			|G(t,x)|\leq C\frac{1}{(|x|+t^\frac{1}{2})^{d+m}}.
			$$
			This implies
			\begin{align*}
				\|G(t,\cdot)\|_{L^p}\leq Ct^{-\frac{d}{2}(1-\frac{1}{p})-\frac{m}{2}}.
			\end{align*}
			This completes the proof of the lemma.
		\end{proof}
		\begin{lem}\label{kern1}
			For $p(\xi):\mathbb{R}^d\to \mathbb{R}$, if
			\begin{align*}
				&p(\xi)\leq -c_0(|\xi|^4\mathbf{1}_{|\xi|\leq 1}+|\xi|^2\mathbf{1}_{|\xi|\geq 1}),\\
				&|\nabla^k p(\xi)|\leq C_k (|\xi|^{4-k}\mathbf{1}_{|\xi|\leq 1}+|\xi|^{2-k}\mathbf{1}_{|\xi|\geq 1}), \quad 0\leq k\leq d+m+6,
			\end{align*}
			for some constants $c_0,C_k>0$, and $m\in\mathbb{N}$.
			Then for
			\bmas
			K_1(t,x)&=\int_{\mathbb{R}^d}\xi^m|\xi|^j\exp(p(\xi)t+i\xi\cdot x)\chi(\xi)d\xi, \quad j=0,1,
			\\
			K_2(t,x)&=\int_{\mathbb{R}^d}\xi^m
			\exp(p(\xi)t+i\xi\cdot x)(1-\chi(\xi))d\xi,
			\emas
			where $\chi$ is a cutoff function satisfying $\{|\xi|\leq 1\}\subset\operatorname{supp}\chi\subset\{|\xi|\leq 2\}$.
			It holds
			\begin{align*}
				&	\|K_1(t,\cdot)\|_{L^p}\leq C(1+t)^{-\frac{d}{4}{(1-\frac{1}{p})}-\frac{m+j}{4}},\\
				&	\|K_2(t,\cdot)\|_{L^p}\leq Ct^{-\frac{d}{2}({1-\frac{1}{p}})-\frac{m}{2}}e^{-\frac{c_0t}{2}},\quad  p\in[1,+\infty].
			\end{align*}
		\end{lem}
		\begin{proof}
			It is easy to check that
			\begin{align}\label{linf}
				&	\|K_1(t,\cdot )\|_{L^\infty}\leq \int_{\mathbb{R}^d} \exp(-c_0 |\xi|^4t)|\xi|^{m+j}\mathbf{1}_{|\xi|\leq 2}d\xi\leq C(1+t)^{-\frac{d+m+j}{4}},\\
				&		\|K_2(t,\cdot )\|_{L^\infty}\leq \int_{\mathbb{R}^d} \exp(-c_0 |\xi|^2t)|\xi|^{m}\mathbf{1}_{|\xi|\geq 1}d\xi\leq Ct^{-\frac{d+m}{2}}e^{-\frac{c_0t}{2}}.\label{linf1}
			\end{align}
			We claim that $|K_1(t,x)|\leq C|x|^{-(d+m+j)}$. If $|x|\leq 1$, the proof is trivial. It suffices to consider $|x|\geq 1$. We further decompose
			\begin{align*}
				K_1(t,x)&=\int_{\mathbb{R}^d}\xi^m|\xi|^j\exp(p(\xi)t+i\xi\cdot x)\chi(\xi)\chi(\xi/\delta)d\xi\\
				&\quad+\int_{\mathbb{R}^d}\xi^m|\xi|^j\exp(p(\xi)t+i\xi\cdot x)\chi(\xi)(1-\chi(\xi/\delta))d\xi\\
				&=:K_{11}(t,x)+K_{12}(t,x),
			\end{align*}
			where $\delta=|x|^{-1}$. Then
			\begin{align*}
				|K_{11}(t,x)|\leq C \int_{\mathbb{R}^d}|\xi|^{m+j}\chi(\xi/\delta)d\xi\leq C|x|^{-(d+m+j)}.
			\end{align*}
			Using integration by parts, we have
			\begin{align*}
				|K_{12}(t,x)|&\leq C|x|^{-(d+m+3)}\int_{\mathbb{R}^d}\left|\nabla_\xi^{d+m+3}\(\xi^m|\xi|^j\exp(p(\xi)t)\chi(\xi)(1-\chi(\xi/\delta))\)\right|d\xi\\
				&\leq C|x|^{-(d+m+j)}.
			\end{align*}
			Hence, we obtain
			\begin{align*}
				|K_1(t,x)|\leq C|x|^{-(d+m+j)}.
			\end{align*}
			Combining this with \eqref{linf}, we obtain
			\begin{align*}
				\|K_1(t,\cdot)\|_{L^p}\leq C(1+t)^{-\frac{d}{4}{(1-\frac{1}{p})}-\frac{m+j}{4}}.
			\end{align*}
			Similarly, for $K_2$, denote $\delta=|x|^{-1}$ for any $x\in\mathbb{R}^d$, we have
			\begin{align*}
				|K_2(t,x)|\leq&C\int_{\mathbb{R}^d}|\xi|^{m}(1-\chi(\xi))\chi(\xi/\delta)d\xi\\
				&+ C|x|^{-d+m+1}\int_{\mathbb{R}^d}\left|\nabla_\xi^{d+m+1}\(\xi^m\exp(p(\xi)t)(1-\chi(\xi))(1-\chi(\xi/\delta))\)\right|d\xi\\
				&\leq C|x|^{-(d+m)}e^{-\frac{c_0t}{2}}.
			\end{align*}
			Combining this with \eqref{linf1}, we obtain
			\begin{align*}
				\|K_2(t,\cdot)\|_{L^p}\leq Ct^{-\frac{d}{2}(1-\frac{1}{p})-\frac{m}{2}}e^{-\frac{c_0t}{2}}.
			\end{align*}
			This completes the proof of the lemma.
		\end{proof}
		\begin{lem} [\cite{Pazy}]\label{S_1-1}
			Let $A$ be a densely defined closed linear operator on the Hilbert space $H$. If both
			$A$ and its adjoint operator $A^*$ are dissipative, then $A$ is the
			infinitesimal generator of a $C_0$-semigroup  on $H$.
		\end{lem}
		\begin{lem}[\cite{Pazy}]\label{semigroup-1}
			Let $A$ be the infinitesimal generator of a $C_0$-semigroup $T(t)$ satisfying $\|T(t)\|\le Me^{\kappa t}$. Then, it holds for $f\in D(A^2)$ and $\sigma>\max(0,\kappa )$ that
			$$
			T(t)f =\frac1{2\pi \i}\int^{\sigma+\i\infty}_{\sigma-\i\infty} e^{\lambda t}(\lambda-A)^{-1}f d\lambda.
			$$
		\end{lem}

		\bigskip

        \noindent {\bf Data availability:}
		No data was used for the research described in the article.

        \bigskip
        
		\noindent {\bf Acknowledgements:}
		K. Chen is supported by the  the Research Centre for Nonlinear Analysis at The Hong Kong Polytechnic University. M. Zhong was supported by the special foundation for Guangxi Ba Gui Scholars,  and the National Natural Science Foundation of China  grants No. 12171104.

		
	\end{document}